\documentclass{compositio}

\renewcommand{\contentsname}{Contents\\{\footnotesize\normalfont(A table
of contents should normally not be included)}}
\usepackage[english]{babel}
\usepackage{newlfont}
\usepackage{amssymb,amsfonts}
\usepackage{mathtools,braket,mathrsfs}
\usepackage[all]{xy}

\usepackage[colorinlistoftodos,bordercolor=black,
backgroundcolor=orange!70!white,linecolor=black,]{todonotes}

\newcommand{\bb}{\mathbb}

\newcommand{\msf}{\mathsf}
\newcommand{\scr}{\mathscr}
\newcommand*{\bfcdot}{\scalebox{0.6}{$\bullet$}}

\theoremstyle{definition}
\newtheorem{theorem}{Theorem}[section]
\newtheorem{lemma}[theorem]{Lemma}
\newtheorem{proposition}[theorem]{Proposition}
\newtheorem{corollary}[theorem]{Corollary}
\newtheorem{definition}[theorem]{Definition}

\theoremstyle{remark}
\newtheorem*{remark}{Remark}

\begin{document}

\title{Twisted triple product p-adic L-functions and Hirzebruch-Zagier cycles}
\author{Michele Fornea}
\email{michele.fornea@mail.mcgill.ca}
\address{McGill University, Montreal, Canada.}
\date{October 10, 2017}

\begin{abstract}
Let $L/F$ be a quadratic extension of totally real number fields. For any prime $p$ unramified in $L$, we construct a $p$-adic $L$-function interpolating the central values of the twisted triple product $L$-functions attached to a $p$-nearly ordinary family of unitary cuspidal automorphic representations of $\text{Res}_{L\times F/F}(\text{GL}_{2})$. Furthermore, when $L/\bb{Q}$ is a real quadratic number field and $p$ is a split prime, we prove a $p$-adic Gross-Zagier formula relating the value of the $p$-adic $L$-function outside the range of interpolation to the syntomic Abel-Jacobi image of generalized Hirzebruch-Zagier cycles.
\end{abstract}

\maketitle

\vspace*{6pt}\tableofcontents  


\section{Introduction}
\label{sec:introduction}
This article is part of the program pioneered by Darmon and Rotger in \cite{DR}, \cite{DR2} devoted to studying the $p$-adic variation of arithmetic invariants for automorphic representations on higher rank groups, with the aim of shedding some light on the relation between $p$-adic $L$-functions and  Euler systems with applications to the equivariant BSD-conjecture. 

Given a totally real number field $F$, the starting point of the program is to find a reductive group $G$ having $\text{GL}_{2,F}$ as a direct factor together with an automorphic $L$-function for which there is an explicit formula for the central $L$-value. The expectation is that there exists a trascendental period for which the ratio between the special value and the period becomes a meaningful algebraic number varying $p$-adically. More precisely, these modified central $L$-values should determine a $p$-adic meromorphic function by interpolation. 
In the present work, we consider the group $G_{L\times F}=\text{Res}_{L\times F/F}(\text{GL}_{2,L\times F})$ for $L/F$ a quadratic extension of totally real number fields. Piatetski-Shapiro and Rallis \cite{PS-R} studied the analytic properties of the twisted triple product $L$-function attached to cuspidal representations of $G_{L\times F}$ and Ichino \cite{I} proved a formula for its central value, generalizing earlier work of Harris-Kudla \cite{HK}. The first part of the paper is devoted to the construction of a $p$-adic $L$-function, called \emph{twisted triple product $p$-adic $L$-function}.

Several far-reaching conjectures suggest a strong link between automorphic  $L$-functions and algebraic cycles: relevant cycles should live on a Kuga-Sato variety whose \'etale cohomology realizes the Galois representation (conjecturally) attached to the automorphic representation of $G$, out of which one constructs the $L$-function. Furthermore, as the central $L$-values should vary $p$-adically after a modification by an appropriate period, by tinkering these cycles it should be possible to produce Galois cohomology classes that $p$-adically intepolate into a \emph{big cohomology class}, giving rise to the $p$-adic $L$-function via Perrin-Riou's machinery. Note that such $p$-adic $L$-function and big cohomology class are defined using completely different inputs, an automorphic and a geometric one; the sole fact that in certain cases it is possible to prove these approches produce the same object is in itself an amazing confirmation of the power of the existing conjectures. 

The relation between $p$-adic $L$-functions and Euler systems, as we just sketched it, can be very hard to prove since it requires, among various things, a deep understanding of the cohomology of semistable models of Shimura varieties. Therefore, we decided to dedicate the second part of this work to the more humble goal of showing that the $p$-adic $L$-function, built using the automorphic input, encodes geometric information of some kind. More precisely, we compute some values of the $p$-adic $L$-function in terms of the syntomic Abel-Jacobi image of \emph{generalized Hirzebruch-Zagier cyles}. Our result is evidence that the (twisted) triple product $p$-adic $L$-function and the generalized Hirzebruch-Zagier cycles are the right objects to consider in the framework determined by $G_{L\times F}$ and the twisted triple product $L$-function. 

In the remaining of the introduction we present our results in more detail. We fix, once and for all, a $p$-adic embedding $\iota_p:\overline{\bb{Q}}\hookrightarrow\overline{\bb{Q}}_p$ for every rational prime $p$, and a complex embedding $\iota_\infty:\overline{\bb{Q}}\hookrightarrow\bb{C}$.

\subsubsection{The $p$-adic $L$-function.}  Let $L/F$ be a quadratic extension of totally real number fields, $\frak{Q}\triangleleft\cal{O}_L$ and $\frak{N}\triangleleft\cal{O}_F$ ideals. The primitive eigenforms $\msf{g}\in S_{\ell,x}(\frak{Q}, L; \overline{\bb{Q}})$ and $\msf{f}\in S_{k,w}(\frak{N}, F; \overline{\bb{Q}})$ generate irreducible cuspidal automorphic representations $\pi, \sigma$ of $G_L(\bb{A}), G_F(\bb{A})$ respectively. We denote by $\pi^u$, $\sigma^u$ their unitarizations and define a representation of $\text{GL}_2(\bb{A}_{L\times F})$ by $\Pi=\pi^u\otimes\sigma^u$. Let $\rho:\Gamma_F\to S_3$ be the homomorphism mapping the absolute Galois group of $F$ to the symmetric group over $3$ elements associated with the \'etale cubic algebra $(L\times F)/F$. The $L$-group $^L(G_{L\times F})$ is given by the semi-direct product $\widehat{G}\rtimes\Gamma_F$ where $\Gamma_F$ acts on $\widehat{G}=\text{GL}_2(\bb{C})^{\times3}$ through $\rho$. 
When the central character $\omega_\Pi$ of $\Pi$ satisfies ${\omega_\Pi}_{\lvert\bb{A}_F^\times}\equiv 1$, one can define the twisted triple product $L$-function $L(s,\Pi,r)$ of $\Pi$ via the representation $r$ of $^L(G_{L\times F})$ on $\bb{C}^2\otimes\bb{C}^2\otimes\bb{C}^2$, which restricts to the natural $8$-dimensional representation of $\widehat{G}$ and through which $\Gamma_F$ acts via $\rho$ permuting the vectors. 
\begin{definition}
We say that the weights $(\ell-2t_L,t_L-x)\in\bb{N}[I_L]^2$, $(k-2t_F,t_F-w)\in\bb{N}[I_F]^2$  are \emph{unbalanced} if there exists $r\in\bb{N}[I_L]$ such that $k=(\ell+2r)_{\lvert F}$ and $w=(x+r)_{\lvert F}$.
\end{definition}

Let $\eta:\bb{A}_F^\times\to\bb{C}^\times$ be the idele character attached to the quadratic extension $L/F$ by class field theory. Suppose that the weights of $\msf{g}$ and $\msf{f}$ are unbalanced and that the local $\epsilon$-factors satisfy 
\[
\epsilon_v\bigg(\frac{1}{2},\Pi_v,r_v\bigg)\eta_v(-1)=+1\qquad\forall v\ \text{finite}\ \text{place}\ \text{of}\ F.
\] Then Theorem $\ref{thm: Itch}$ and Lemma $\ref{lemma:choice test vector}$ show that the non-vanishing of the central $L$-value $L(\frac{1}{2},\Pi,r)$ is equivalent to the existence of \emph{test vectors} $\breve{\msf{g}},\breve{\msf{f}}$ in $\pi, \sigma$, respectively, of some common level $M$ prime to $\frak{N}\cdot\text{N}_{L/F}(\frak{Q})\cdot d_{L/F}$; that is, $\breve{\msf{g}},\breve{\msf{f}}$ are cuspforms such that the Petersson inner product \begin{equation}\label{period}
\text{I}(\phi)=\left\langle\zeta^*\left(\delta^r\breve{\msf{g}}\right),\breve{\msf{f}}^*\right\rangle
\end{equation}
does not vanish. In other words, we can take $(\ref{period})$ as an avatar of the central $L$-value and use it to construct the $p$-adic $L$-function.
\begin{remark}
	The assumption on local $\epsilon$-factors at the finite places of $F$ can be satisfied by requiring the ideals $\text{N}_{L/F}(\frak{Q})\cdot d_{L/F}$ and $\frak{N}$ to be coprime and by asking all prime ideals dividing $\frak{N}$ to split in $L/F$. 
\end{remark}
 We choose a rational prime $p$ unramified in $L$, coprime to the levels $\frak{Q},\frak{N}$, and $\Theta_{L/F}$ an element in $\prod_{\tau\in I_F}\{\mu\in I_L\lvert\ \mu_{\lvert F}=\tau\}$. The fixed $p$-adic embedding $\iota_p:\overline{\bb{Q}}\hookrightarrow\overline{\bb{Q}}_p$ allows us to associate to $\Theta_{L/F}$ a set $\cal{P}$ of prime $\cal{O}_L$-ideals above $p$. Moreover, we suppose $\msf{g}$ and $\msf{f}$ to be $p$-nearly ordinary and we denote by $\cal{G}$ (resp. $\cal{F}$) the Hida family with coefficients in $\mathbf{I}_\cal{G}$ (resp. $\mathbf{I}_\cal{F}$) passing through the nearly ordinary $p$-stabilization of $\msf{g}$ (resp. $\msf{f}$).
 \begin{definition}
 	The set of \emph{unbalanced crystalline points}, denoted $\Sigma^\msf{f}_\text{cry}=\Sigma^\msf{f}_\text{cry}(\mathbf{I}_\cal{G}, \mathbf{I}_\cal{F})$, is the subset of arithmetic points $(\text{P},\text{Q})\in\cal{A}(\mathbf{I}_\cal{G})\times\cal{A}(\mathbf{I}_\cal{F})$ of unbalanced weights, trivial characters and such that the specialization of the Hida families are old at $p$; that is, they are the $p$-stabilization of eigenforms of prime-to-$p$ level: $\cal{G}_\text{P}=\msf{g}_\text{P}^{(p)}$ and $\cal{F}_\text{Q}=\msf{f}_\text{Q}^{(p)}$. In particular, we have a family of unitary automorphic representations $\{\Pi_{\text{P},\text{Q}}\}_{(\text{P},\text{Q})\in \Sigma^\msf{f}_\text{cry}}$ of  prime-to-$p$ level.
 \end{definition}
 
We set $\mathbf{K}_{\cal{G}}=\mathbf{I}_{\cal{G}}\otimes \bb{Q}$ (resp. $\mathbf{K}_{\cal{F}}=\mathbf{I}_{\cal{F}}\otimes\bb{Q}$) and we define a $\mathbf{K}_{\cal{G}}$-adic cuspform $\breve{\cal{G}}$ (resp. $\mathbf{K}_{\cal{F}}$-adic cuspform $\breve{\cal{F}}$) passing through the nearly ordinary $p$-stabilization of the test vectors $\breve{\msf{g}},\breve{\msf{f}}$ as in \cite{DR} Section 2.6. Set $\mathbf{R}_{\cal{G}}=O[[\cal{O}_{L,\cal{P}}^\times]]\otimes_O\mathbf{K}_{\cal{G}}$, Lemma $\ref{existence}$ ensure the existence of a meromorphic function $
\mathscr{L}_p(\breve{\cal{G}},\breve{\cal{F}})\in \mathbf{R}_{\cal{G}}\otimes_{\mathbf{A}_F}\text{Frac}(\mathbf{K}_{\cal{F}})$ whose value at a point $(\text{P},\text{Q})\in\Sigma^\msf{f}_\text{cry}$ is
\[
\mathscr{L}_p\big(\breve{\cal{G}},\breve{\cal{F}}\big)(\text{P},\text{Q})=\frac{1}{\msf{E}(\msf{f}_\text{Q}^*)}\frac{\big\langle e_\text{n.o.}\zeta^*\big(d^{r}\breve{\msf{g}}_\text{P}^{[\cal{P}]}\big), \breve{\msf{f}}_\text{Q}^{*}\big\rangle}{\big\langle \msf{f}_\text{Q}^{*}, \msf{f}_\text{Q}^{*}\big\rangle}.
\] 
Here the number $\msf{E}(\msf{f}_\text{Q}^*)$ is defined by $\msf{E}(\msf{f}_\text{Q}^*)=(1-\beta_{\msf{f}_\text{Q}^*}\alpha_{\msf{f}_\text{Q}^*}^{-1})$ for $\alpha_{\msf{f}_\text{Q}^*},\beta_{\msf{f}_\text{Q}^*}$ the inverses of the roots of the Hecke polynomial for $T(p)$ and $r\in\bb{N}[I_L]$ is the unique element, with $r_\mu=0$ if $\mu\notin\Theta_{L/F}$, that makes the weights of $\zeta^*(d^r\breve{\msf{g}}_\text{P})$ and $\breve{\msf{f}}_\text{Q}$ match.
We are justified in calling $\mathscr{L}_p\big(\breve{\cal{G}},\breve{\cal{F}}\big)$ a $p$-adic $L$-function because it interpolates the algebraic avatar $(\ref{period})$ of central $L$-values $L(\frac{1}{2},\Pi_{\text{P},\text{Q}},r)$ for $(\text{P},\text{Q})\in\Sigma^\msf{f}_\text{cry}$, as the next theorem shows.
\begin{theorem}
Consider the partition $\cal{Q}_\text{inert}\coprod\cal{Q}_\text{split}$ of the set of $\cal{O}_F$-prime ideals above $p$ determined by the splitting behavior of the primes in the quadratic extension $L/F$. For all $(\text{P},\text{Q})\in\Sigma^\msf{f}_\text{cry}$, the twisted triple product $p$-adic $L$-function   $
\mathscr{L}_p(\breve{\cal{G}},\breve{\cal{F}})\in \mathbf{R}_{\cal{G}}\otimes_{\mathbf{A}_F}\text{Frac}(\mathbf{K}_{\cal{F}})$ satisfies the equality 
\[
\mathscr{L}_p\big(\breve{\cal{G}},\breve{\cal{F}}\big)(\text{P},\text{Q})=\frac{1}{\msf{E}(\msf{f}_\text{Q}^*)}
\left(\prod_{\wp\in\cal{Q}_\text{inert}}\cal{E}_\wp(\msf{g}_\text{P},\msf{f}_\text{Q}^*)
\prod_{\wp\in\cal{Q}_\text{split}}\frac{\scr{E}_\wp(\msf{g}_\text{P},\msf{f}_\text{Q}^*)}{\scr{E}_{0,\wp}(\msf{g}_\text{P},\msf{f}_\text{Q}^*)}\right)
\frac{\big\langle \zeta^*\left(\delta^r\breve{\msf{g}}_\text{P}\right), \breve{\msf{f}}_\text{Q}^{*}\big\rangle}{\big\langle \msf{f}_\text{Q}^{*} \msf{f}_\text{Q}^{*}\big\rangle},
\]
where the Euler factors appearing in the formula are defined in Lemmas $\ref{lemma: inert}$ and $\ref{lemma: split}$, and $r\in\bb{N}[I_L]$ is the unique element, with $r_\mu=0$ if $\mu\notin\Theta_{L/F}$, that makes the weights of $\zeta^*(\delta^r\breve{\msf{g}}_\text{P})$ and $\breve{\msf{f}}_\text{Q}$ match. 
\end{theorem}

\subsubsection{A $p$-adic Gross-Zagier formula.}
The second part of the paper deals with the evaluation of the $p$-adic $L$-function outside the range of interpolation. From now on, we assume $L/\bb{Q}$ to be a real quadratic number field.
 \begin{definition}
	A triple of integers $(a,b,c)\in\bb{Z}^3$, $a,b,c\ge2$, is said to be balanced if none among $a,b,c$ is greater or equal than the sum of the other two. We say that the weights $(\ell-2t_L,t_L-x)\in\bb{N}[I_L]^2$, $(k-2,1-w)\in\bb{N}[I_\bb{Q}]^2$ are \emph{balanced} if there exists $r\in\bb{N}[I_L]$, $r>0$, such that $k=(\ell-2r)_{\lvert F}$, $w=(x-r)_{\lvert F}$ and the triple of integers $(\ell_1,\ell_2,k)$ is balanced.
\end{definition}
\begin{definition}
	The set of \emph{balanced crystalline points}, denoted $\Sigma^\text{bal}_\text{cry}=\Sigma^\text{bal}_\text{cry}(\mathbf{I}_\cal{G}, \mathbf{I}_\cal{F})$, is the subset of arithmetic points $(\text{P},\text{Q})\in\cal{A}(\mathbf{I}_\cal{G})\times\cal{A}(\mathbf{I}_\cal{F})$ of balanced weights, trivial characters and such that the specialization of the Hida families are old at $p$. The set of balanced old points is a disjoint union, indexed by balanced triples $(\ell,k)$, of subsets $\Sigma^\text{bal}_\text{cry}(\ell,k)$ consisting of points whose weights have the form $(\ell-2t_L,t_L-x)\in\bb{N}[I_L]^2$, $(k-2,1-w)\in\bb{N}[I_\bb{Q}]^2$.
 \end{definition}

For old balanced points $(\text{P},\text{Q})\in\Sigma^\text{bal}_\text{cry}$, the global sign of the functional equation of $L(s,\Pi_{\text{P},\text{Q}},r)$ is $-1$. This forces the vanishing of the central value, which one expects to be accounted for by the family of generalized Hirzebruch-Zagier cycles . Interestingly, the (twisted) triple product $p$-adic $L$-function is not forced to vanish on $\Sigma^\text{bal}_\text{cry}$ and we can try to compute its values there. Let $(\ell,k)$  be a balanced triple such that either $\ell$ is not parallel or $(\ell,k)=(2t_L,2)$. Let $\cal{A}\to\mathbf{Sh}_M(G^*_L)$ be the universal abelian surface over the Shimura variety for $G^*_L$ of level $M$ and let $\cal{E}\to\mathbf{Sh}_M(\text{GL}_{2,\bb{Q}})$ be the universal elliptic curve over the modular curve of level $M$, both defined over some open subset of $\text{Spec}(\cal{O}_E)$, where $E/\bb{Q}$ is a large enough finite Galois extension. For all but finitely many primes $p$, let $\wp\triangleleft\cal{O}_E$ be the prime above $p$ induced by the fixed $p$-adic embedding $\iota_p$, and consider $\scr{U}_{\ell-4}\times_{\cal{O}_{E,\wp}}\scr{W}_{k-2}$ a smooth and proper compactification of $\cal{A}^{\lvert\ell\rvert-4}\times\cal{E}^{k-2}$. The generalized Hirzebruch-Zagier cycle of weight $(\ell,k)$ is a De Rham null-homologous cycle \[\Delta_{\ell,k}\in \text{CH}_{r+1}(\scr{U}_{\ell-4}\times_{\cal{O}_{E,\wp}}\scr{W}_{k-2})_0\otimes_\bb{Z}L\]
of dimension $r+1=\frac{\lvert\ell\rvert+k-4}{2}$. Given a pair of eigenforms $\breve{\msf{g}}_\text{P}\in S_{\ell,x}(M, L; E)$ and $\breve{\msf{f}}_\text{Q}\in S_{k,w}(M; E)$ we can produce cohomology classes $\omega_\text{P}$ and $\eta_\text{Q}$,  as in Definition $\ref{definition of classes}$, such that $\pi_1^*\omega_\text{P}\cup\pi_2^*\eta_\text{Q}\in \text{F}^{\lvert\ell\rvert-2-s} H_\text{dR}^{\lvert\ell\rvert+k-3}\big(U_{\ell-4}\times_{E_\wp}W_{k-2}\big)$ where $s=\frac{\lvert\ell\rvert-k-2}{2}$; that is, the cohomology class $\pi_1^*\omega_\text{P}\cup\pi_2^*\eta_\text{Q}$ lives in the domain of the syntomic Abel-Jacobi image of $\Delta_{\ell,k}$,
\[
\text{AJ}_p(\Delta_{\ell,k}):\text{F}^{\lvert\ell\rvert-2-s} H_\text{dR}^{\lvert\ell\rvert+k-3}\big(U_{\ell-4}\times_{E_\wp}W_{k-2}\big)\longrightarrow E_\wp,
\]
and we can compute the number $ \text{AJ}_p(\Delta_{\ell,k})(\pi_1^*\omega_\text{P}\cup\pi_2^*\eta_\text{Q})$ as follows.
\begin{theorem}\label{intro GZ formula}
	Let $L/\bb{Q}$ be a real quadratic field and $(\ell,k)$ a balanced triple. Let $p$ be a prime splitting in $L$ for which the generalized Hirzebruch-Zagier cycle $\Delta_{\ell,k}$ is defined. Then for all $(\text{P},\text{Q})\in\Sigma^\text{bal}_\text{cry}(\ell,k)$  we have
	\[
	\scr{L}_p(\breve{\cal{G}},\breve{\cal{F}})(\text{P},\text{Q})=\frac{(-1)^s}{s!\msf{E}(\msf{f}_\text{Q}^*)}\frac{\scr{E}_p(\msf{g}_\text{P},\msf{f}_\text{Q}^*)}{\scr{E}_{0,p}(\msf{g}_\text{P},\msf{f}_\text{Q}^*)}\text{AJ}_p(\Delta_{\ell,k})(\pi_1^*\omega_\text{P}\cup\pi_2^*\eta_\text{Q}).
	\]
\end{theorem}
\begin{remark}
The assumption on the splitting behaviour of $p$ in $L/\bb{Q}$ should not be necessary. It could be dispensed with by showing the overconvergence of the $p$-adic cuspform $d_\mu^{1-\ell_\mu}(\breve{\msf{g}}_\text{P}^{[p]})$ for $\mu\in I_L$. It seems resonable to believe that by generalizing the recent work of Andreatta and Iovita \cite{preprintAI} one could prove such a result.
\end{remark}

From our Theorem $\ref{intro GZ formula}$ and the impressive work of Liu \cite{YLiu}, we can deduce a corollary in the spirit of the Block-Kato conjecture. Let $A$ be an elliptic curve over $L$ of conductor $\frak{Q}$ and $E$ a rational elliptic curve of conductor $\frak{N}$, both without complex multiplication over $\overline{\bb{Q}}$; let $\msf{g}_A,\msf{f}_E$ be the cuspforms attached to them by modularity. For a rational prime $p$, if $\msf{g}_A,\msf{f}_E$ are $p$-nearly ordinary we denote by $\cal{G}, \cal{F}$ the Hida families passing through $\cal{G}_{\text{P}_A}=\msf{g}_A$ and $\cal{F}_{\text{Q}_E}=\msf{f}_E$, respectively. Following \cite{YLiu} we denote by $(\text{M}_{A,E})_p$ the Galois representation $\text{As}\text{V}_p(A)(-1)\otimes_{\bb{Q}_p}\text{V}_p(E)$ of the absolute Galois group of $\bb{Q}$. 
\begin{corollary}
Suppose that $\frak{N}$ and $\text{N}_{L/\bb{Q}}(\frak{Q})\cdot d_{L/\bb{Q}}$ are coprime ideals and that all the primes dividing $\frak{N}$ split in $L$. For all but finitely many primes $p$ that are split in $L$ and such that $\msf{g}_A,\msf{f}_E$ are $p$-nearly ordinary we have 
\[
\scr{L}_p(\breve{\cal{G}},\breve{\cal{F}})(\text{P}_A,\text{Q}_E)\not=0\quad\implies\quad \dim_{\bb{Q}_p}H_f^1(\bb{Q},(\text{M}_{A,E})_p)=1.
\]
\end{corollary}
\begin{proof}
	The non-vanishing of the $p$-adic $L$-function implies the non-vanishing of the syntomic Abel-Jacobi image by Theorem $\ref{intro GZ formula}$, which in turn forces the non-vanishing of the $p$-adic \'etale Abel-Jacobi image of the Hirzebruch-Zagier cycle (\cite{BDP}, Section 3.4). Finally, Liu's theorem (\cite{YLiu}, Theorem 1.5) gives the dimension of the Block-Kato Selmer group. Note that even though there is a difference between our definition of the HZ-cycle in weight $(2t_L,2)$ and Liu's, both give the same value $\text{AJ}_p(\Delta_{2t_L,2})(\pi_1^*\omega_{\text{P}_A}\cup\pi_2^*\eta_{\text{Q}_E})$.
\end{proof}

While writing this paper we became aware that other people were working on similar projects: Ishikawa constructed twisted triple product $p$-adic $L$-functions over $\bb{Q}$ in his Ph.D. thesis \cite{Ishikawa}, while Blanco and Sols were computing the syntomic Abel-Jacobi image of Hirzebruch-Zagier cycles in \cite{BlancoSols}.

\begin{acknowledgements}
I am grateful to my Ph.D. advisors, Prof. Henri Darmon and Prof. Adrian Iovita, for sharing their passion and expertise through the years and for suggesting the problems considered in this paper. I would like to thank Giovanni Rosso, Nicolas Simard and Jan Vonk for the numerous conversations without which my understanding of mathematics would not be the same. Finally, I appreciated David Loeffler and Daniel Disegni's comments on the first draft of the paper.
\end{acknowledgements}


\section{Automorphic forms}
\label{sec:auto_part} 

\subsection{Adelic Hilbert modular forms}
\label{subsec:ah_mod_form}
Let $F/\bb{Q}$ be a totally real number field and let $I$ be the set of field embeddings of $F$ into $\overline{\bb{Q}}$. We denote by $G$ the algebraic group $\text{Res}_{F/\bb{Q}}\text{GL}_{2,F}$. We choose a square root $i$ of $-1$ which allows us to define the Poincar\'e half-plane $\frak{H}$, we consider the complex manifold $\frak{H}^I$ which is endowed with a transitive action of $G(\bb{R})^+\cong\prod_I\text{GL}_2(\bb{R})^+$ and contains the point $\mathbf{i}=(i,\dots,i)$. For any $K\triangleleft G(\bb{A}^\infty)$ compact open subgroup we denote by $M_{k,w}(K;\bb{C})$ the space of holomorphic Hilbert modular forms of weight $(k,w)\in\bb{Z}[I]^2$ and level $K$. It is defined as the space of functions $\msf{f}:G(\bb{A})\to\bb{C}$ that satisfy the following list of properties: 
\begin{itemize}
\item[$\bullet$] $\msf{f}(\alpha x u)=\msf{f}(x)j_{k,w}(u_\infty,\mathbf{i})^{-1}$ where $\alpha\in G(\bb{Q})$, $u\in K\cdot C_{\infty}^+$ for $C_{\infty}^+$ the stabilizer of $\mathbf{i}$ in $G(\bb{R})^+$ and the automorphy factor is
 $j_{k,w}\bigg(\begin{pmatrix}a& b\\
c&d\end{pmatrix},z\bigg)=(ad-bc)^{-w}(cz+d)^k$ for 
$\begin{pmatrix}a& b\\
c&d\end{pmatrix}
\in G(\bb{R})$, $z\in\mathfrak{H}^I$;
\item[$\bullet$] for every finite adelic point $ x\in G(\bb{A}^\infty)$ the well-defined function $\msf{f}_x:\frak{H}^I\to\bb{C}$ given by $\msf{f}_x(z)=\msf{f}(xu_\infty)j_{k,w}(u_\infty,\mathbf{i})$ is holomorphic, where for each $z\in\mathfrak{H}^I$ we choose $u_\infty\in G(\bb{R})^+$ such that $u_\infty\mathbf{i}=z$.
\item If the totally real field is just the field of rational numbers, $F=\bb{Q}$, we need to impose the extra condition that for all finite adelic point $x\in G(\bb{A}^\infty)$ the function $\lvert \text{Im}(z)^\frac{k}{2}\msf{f}_x(z)\rvert$ is uniformily bounded on $\frak{H}$.
\end{itemize}
We denote by $S_{k,w}(K;\bb{C})$ the $\bb{C}$-linear subspace of $M_{k,w}(K;\bb{C})$ of all the functions that satisfy the additional property:
\begin{itemize}
\item[$\bullet$] for all adelic points $x\in G(\bb{A})$ and for all additive measures on 
$F\backslash\bb{A}_F$ we have
 \[\int_{F\backslash\bb{A}_F}\msf{f}\bigg(\begin{pmatrix}1&a\\0&1\end{pmatrix}x\bigg)da=0.\]
\end{itemize}
\begin{remark}
Fix $(n,v)\in\bb{Z}[I]^2$ with $n,v\ge0$ and set $k=n+2t$, $w=t-v$ where $t=\sum_{\tau\in I}\tau\in\bb{Z}[I]$. To have a non-trivial modular form we need to assume $k-2w=mt$ for some $m\in\bb{Z}$.
\end{remark}

\begin{definition}
	We denote by $G^*$ the algebraic group $\text{Res}_{F/\bb{Q}}\text{GL}_{2,F}\times_{\text{Res}_{F/\bb{Q}}\bb{G}_{m,F}}\bb{G}_m$. By replacing $G$ by $G^*$ in the previous definition, we define $S^*_{k,w}(K;\bb{C})$ to be the space of automorphic forms for $G^*$ of weight $(k,w)\in\bb{Z}[I]^2$ and level $K$, for any $K\triangleleft\ G^*(\bb{Q})$ compact open subgroup.
\end{definition}
A nice feature of these automorphic forms is that we do not need to require $k-2w$ to be parallel to have non-trivial forms and that for all couple of weights $(k,w), (k,w')\in \bb{Z}[I]^2$ there is a natural isomorphism
\begin{equation}\label{prop: change central character}
	\Psi_{w,w'}: S_{k,w}^*(K;\bb{C})\overset{\sim}{\longrightarrow} S_{k,w'}^*(K;\bb{C})\end{equation}
given by $\msf{f}(x)\mapsto\msf{f}(x)\lvert\det(x)\rvert_{\bb{A}_\bb{Q}}^{w'-w}$.

Each irreducible automorphic representation $\pi$ spanned by some form in $S_{k,w}(K;\bb{C})$ has non-unitary central character equal to $\lvert-\rvert_{\bb{A}_F}^{-m}$ up to finite order characters. The twist $\pi^u:=\pi\otimes\lvert-\rvert_{\bb{A}_F}^{\frac{m}{2}}$ is called the unitarization of $\pi$. Note that there is an isomorphism of function spaces (not of $G(\bb{A})$-modules) \begin{equation}\label{eq:unitarization}
\begin{split}
&\pi\overset{\sim}{\longrightarrow}\pi^u \\
&\msf{f}\mapsto \msf{f}^u
\end{split}\qquad\qquad
 \text{where}\qquad\msf{f}^u(x)=\msf{f}(x)\lvert\det(x)\rvert_{\bb{A}_F}^{\frac{m}{2}}.
 \end{equation}
 Choose Haar measures $\text{d}^\times x_v$ on $F_v^\times\backslash \text{GL}_2(F_v)$ for all places $v$ of $F$ so that the product measure $\text{d}^\times x=\prod_v\text{d}^\times x_v$ is the Tamagawa measure on $\bb{A}_F^\times\backslash G_F(\bb{A})$ and the volume of $\cal{O}_{F_v}^\times\backslash\text{GL}_2(\cal{O}_{F_v})$ is equal to $1$ for all but finitely many $v$'s. For any two cuspforms $\msf{f}_1,\msf{f}_2\in S_{k,w}(\frak{N},\bb{C})$ with $k-2w=mt$  we define their Petersson inner product to be
\begin{equation}\label{eq: Petersson inner product}
\langle\msf{f}_1,\msf{f}_2\rangle=\int_{[G_F(\bb{A})]}\msf{f}_1(x)\overline{\msf{f}_2(x)}\lvert\det(x)\rvert_{\bb{A}_F}^m\text{d}^\times x=\langle\msf{f}_1^u,\msf{f}_2^u\rangle
.\end{equation}
where $[G_F(\bb{A})]=\bb{A}_F^\times G_F(\bb{Q})\backslash G_F(\bb{A})$.
\newline
Let $\cal{O}_F$ be the ring of integers of $F$ and $\frak{N}$ one of its ideals. We are interested in the following compact open sugroups of $G(\widehat{\bb{Z}})$:
\begin{itemize}
\item $U_0(\frak{N})=\bigg\{\begin{pmatrix}a&b\\c&d\end{pmatrix}\in G(\widehat{\bb{Z}})\bigg\lvert\ c\in\frak{N}\widehat{\cal{O}_F}\bigg\}$,
\item $V_1(\frak{N})=\bigg\{\begin{pmatrix}a&b\\c&d\end{pmatrix}\in U_0(\frak{N})\bigg\lvert\ d\equiv1 \pmod{\frak{N}\widehat{\cal{O}_F}}\bigg\}$,
\item $U(\frak{N})=\bigg\{\begin{pmatrix}a&b\\c&d\end{pmatrix}\in V_1(\frak{N})\bigg\lvert\ a\equiv1 \pmod{\frak{N}\widehat{\cal{O}_F}}\bigg\}$.
\end{itemize}
One can decompose the ideles of $F$ as 
\[
\bb{A}_F^\times=\coprod_{i=1}^{h_F^+
(\frak{N})}F^\times a_i\det U(\frak{N}) F^\times_{\infty,+}
\]
where $a_i\in\bb{A}_F^{\infty,\times}$ and $h_F^+
(\frak{N})$ is the cardinality of $\text{cl}_F^+(\frak{N}):= F_+^\times\backslash\bb{A}_F^{\infty,\times}/\det U(\frak{N})$, the narrow class group of $U(\frak{N})$. The ideles decomposition induces a decomposition of the adelic points of $G$
\[
G(\bb{A})=\coprod_{i=1}^{h_F^+
(\frak{N})}G(\bb{Q})t_iU(\frak{N})G(\bb{R})^+\qquad \text{for}\qquad t_i=\begin{pmatrix}a_i^{-1}&0\\ 0&1\end{pmatrix}.
\]
\subsubsection{Adelic q-expansion.} The Shimura variety $\mathbf{Sh}_K(G)$, determined by $G$ and a compact open $K$, is not compact, therefore there is a notion of $q$-expansion for Hilbert modular forms. Even more, Shimura found a way to package the $q$-expansions of each connected component of $\mathbf{Sh}_K(G)$ into a unique adelic $q$-expansion.
Fix $\mathsf{d}\in\bb{A}_F^{\infty,\times}$ such that $\mathsf{d}\cal{O}_F=\frak{d}$ is the absolute different ideal of $F$. Let $F^\text{Gal}$ be the Galois closure of $F$ in $\overline{\bb{Q}}$ and write $\mathcal{V}$ for the ring of integers or a valuation ring of a finite extension $F_0$ of $F^\text{Gal}$ such that for every ideal $\frak{a}$ of $\cal{O}_F$, for all $ \tau\in I$, the ideal $\frak{a}^\tau\mathcal{V}$ is principal.
Choose a generator $\{\frak{q}^\tau\}\in\mathcal{V}$ of $\frak{q}^\tau\mathcal{V}$ for each prime ideal $\frak{q}$ of $\cal{O}_F$ and by multiplicativity define $\{\frak{a}^v\}\in\cal{V}$ for each fractional ideal $\frak{a}$ of $F$ and each $v\in\bb{Z}[I]$.
\newline
Given a Hilbert cuspform $\msf{f}\in S_{k,w}(U(\frak{N});\bb{C})$, one can consider for every index $i\in\{1,\dots, h_F^+(\frak{N})\}$, the holomorphic function $\msf{f}_i:\frak{H}^I\to\bb{C}$
\[
\msf{f}_i(z)=y_\infty^{-w}\msf{f}\left(t_i\begin{pmatrix}
y_\infty& x_\infty\\
0&1
\end{pmatrix}\right)=\underset{\xi\in(\frak{a}_i\frak{d}^{-1})_+}{\sum}a(\xi,\msf{f}_i)e_F(\xi z)
\]
for $z=x_\infty+\mathbf{i}y_\infty$, $\frak{a}_i=a_i\cal{O}_F$ and $e_F(\xi z)=\text{exp}\left(2\pi i\sum_{\tau\in I}\tau(\xi)z_\tau\right)$.
Every idele $y$ in $\bb{A}^\times_{F,+}:=\bb{A}_F^{\infty,\times}(F^\times_{\infty,+})$ can be written as $y=\xi a_i^{-1}\mathsf{d}u$ for $\xi\in F_+^\times$ and $u\in\det U(\frak{N})F^\times_{\infty,+}$; the following functions 
\[
\msf{a}(-,\msf{f}):\bb{A}^\times_{F,+}\longrightarrow\bb{C},\qquad\qquad \msf{a}_p(-,\msf{f}):\bb{A}^\times_{F,+}\longrightarrow\overline{\bb{Q}}_p
\]
are defined by $\msf{a}(y,\msf{f}):=a(\xi,\msf{f}_i)\{y^{-v}\}\xi^v\lvert a_i\rvert_{\bb{A}_F}$ and $\msf{a}_p(y,\msf{f}):=a(\xi,\msf{f}_i)y_p^{-v}\xi^v\cal{N}_F(a_i)^{-1}$ if $y\in\widehat{\cal{O}_F}F^\times_{\infty,+}$ and zero otherwise. Here $\cal{N}_F:Z(1)\to\overline{\bb{Q}}_p^\times$ is the map from the group $Z(1)=\bb{A}^\times_F/\overline{F^\times\det U(p^\infty)F^\times_{\infty,+}}$ defined by $y\mapsto y_p^{-t}\lvert y^\infty\rvert_{\bb{A}_F}^{-1}$. Clearly, the function $\mathsf{a}_p(-,f)$ makes sense only if the coefficients $a(\xi,f_i)\in\overline{\bb{Q}}$ are algebraic $\forall \xi,i$. 
 
 \begin{theorem}{(\cite{pHida} Theorem 1.1)}
	Consider the map $e_F:\bb{C}^I\longrightarrow\bb{C}^\times$ defined by $e_F(z)=\text{exp}\left(2\pi i\sum_{\tau\in I}z_\tau\right)$ and the additive character of the ideles $\chi_F:\bb{A}/F\longrightarrow\bb{C}^\times$ which satisfies $\chi_F(x_\infty)=e_F(x_\infty)$, then each cuspform $\msf{f}\in S_{k,w}(U(\frak{N});\bb{C})$ has an adelic $q$-expansion of the form
	 \[
	 \msf{f}\left(\begin{pmatrix}
	 y & x \\ 
	 0 & 1
	 \end{pmatrix}\right)=\lvert y\rvert_\bb{A}\underset{\xi\in F_+}{\sum}\mathsf{a}(\xi y\mathsf{d},\msf{f})\{(\xi y\mathsf{d})^v\}(\xi y_\infty)^{-v}e_F(\mathbf{i}\xi y_\infty)\chi_F(\xi x) 
	 \]
	 where $\mathsf{a}(-,\msf{f}):\bb{A}^\times_{F,+}\longrightarrow \bb{C}$ vanishes outside $\widehat{\cal{O}_F}F^\times_{\infty,+}$ and depends only on the coset $y^\infty\det U(\frak{N})$. Furthermore, by formally replacing $e_F(\mathbf{i}\xi y_\infty)\chi_F(\xi x)$ by $q^\xi$ and $(\xi y_\infty)^{-v}$ by $(\xi\mathsf{d}_p y_p)^{-v}$, we have the $q$-expansion
	 \[
	 \mathsf{f}=\cal{N}(y)^{-1}\underset{\xi\in F_+}{\sum}\mathsf{a}_p(\xi y\mathsf{d},\msf{f})q^{\xi}
	 \]
where 	 $\mathsf{a}_p(-,\msf{f}):\bb{A}^\times_{F,+}\longrightarrow \overline{\bb{Q}}_p$ depends only on the coset $y^\infty\det U(\frak{N}p^\infty)$ for $\det U(\frak{N}p^\infty)=\{u\in\det U(\frak{N})\lvert\ u_p=1\}$.
	 \end{theorem}

\subsubsection{Nearly holomorphic cuspforms}
For any $K\triangleleft G(\bb{A}^\infty)$ compact open subgroup we denote by $N_{k,w,m}(K;\bb{C})$ the space of nearly holomorphic cuspforms of weight $(k,w)\in\bb{Z}[I]^2$ and order less or equal to $m\in\bb{N}[I]$ with respect to $K$. It is the space of functions $\msf{f}:G(\bb{A})\to\bb{C}$ that satisfy the following list of properties: 
\begin{itemize}
\item[$\bullet$] $\msf{f}(\alpha x u)=\msf{f}(x)j_{k,w}(u_\infty,\mathbf{i})^{-1}$ where $\alpha\in G(\bb{Q})$, $u\in K\cdot C_{\infty}^+$;
\item[$\bullet$] for each $ x\in G(\bb{A}^\infty)$ the well-defined function $\msf{f}_x(z)=\msf{f}(xu_\infty)j_{k,w}(u_\infty,\mathbf{i})$ can be written as 
\[
\msf{f}_x(z)=\sum_{\xi\in L(x)_+}a(\xi,\msf{f}_x)((4\pi y)^{-1})e_F(\xi z)
\] for polynomials $a(\xi,\msf{f}_x)(Y)$ in the variables $(Y_\tau)_{\tau\in I}$ of degree less than $m_\tau$ in $Y_\tau$ for each $\tau\in I$ and for $L(x)$ a lattice of $F$.
\end{itemize}
As before $\msf{f}_i$ stands for $\msf{f}_{t_i}$ and we consider adelic Fourier coefficients 
\[
\msf{a}(y,\msf{f})(Y)=\{y^{-v}\}\xi^v\lvert a_i\rvert_{\bb{A}_F}a(\xi,\msf{f}_i)(Y),\qquad \msf{a}_p(y,\msf{f})(Y)=y_p^{-v}\xi^v\cal{N}_F(a_i)^{-1}a(\xi,\msf{f}_i)(Y)
\]
if $y=\xi a_i^{-1}\msf{d}u\in\widehat{\cal{O}_F}F^\times_{\infty,+}$ and zero otherwise. Then the adelic Fourier expansion of a nearly holomorphic cuspform $\msf{f}$ is given by
\[
	 \msf{f}\left(\begin{pmatrix}
	 y & x \\ 
	 0 & 1
	 \end{pmatrix}\right)=\lvert y\rvert_{\bb{A}_F}\underset{\xi\in F_+}{\sum}\mathsf{a}(\xi y\mathsf{d},\msf{f})(Y)\{(\xi y\mathsf{d})^v\}(\xi y_\infty)^{-v}e_F(\mathbf{i}\xi y_\infty)\chi_F(\xi x) 
	 \]
for $Y=(4\pi y_\infty)^{-1}$. If $A$ is a subring of $\bb{C}$ and the compact open subgroup $K\triangleleft G(\bb{A}^\infty)$ satisfies $U(\frak{N})\subset K\subset U_0(\frak{N})$ for some $\cal{O}_F$-ideal $\frak{N}$, one can consider
\[
N_{k,w,m}(K;A)=\{\msf{f}\in N_{k,w,m}(K;\bb{C})\lvert\ \msf{a}(y,\msf{f})\in A\ \forall y\in\bb{A}_{F,+}^\times\}.
\]
Moreover, there are  Mass-Shimura differential operators for $r\in\bb{N}[I]$, $k\in\bb{Z}[I]$ defined as
\[
\delta^r_k=\prod_{\tau\in I}(\delta^\tau_{k_\tau+2r_\tau-2}\circ\dots\circ\delta^\tau_{k_\tau}),\qquad\text{where}\qquad\delta_\lambda^\tau=\frac{1}{2\pi i}\left(\frac{\lambda}{2iy_\tau}+\frac{\partial}{\partial z_\tau}\right).
\]
They act on a nearly holomorphic cuspform $\msf{f}\in N_{k,w,m}(K;\bb{C})$ via the expression $\msf{a}(y,\delta^r_k\msf{f})(Y)=\{y^{-v+r}\}\xi^{v-r}\lvert a_i\rvert_{\bb{A}_F}a(\xi, \delta^r_k\msf{f}_i)(Y)$.
Suppose that $\bb{Q}\subset A$, then Hida showed (\cite{pHida} Proposition 1.2) the differential operator $\delta^r_k$ maps $N_{k,w,m}(K;A)$ to $N_{k+2r,w+r,m+r}(K;A)$ and if $k>2m$ there is holomorphic projector $
\Pi^\text{hol}:N_{k,w,m}(K;A)\longrightarrow S_{k,w}(K;A)$.


\subsubsection{Hecke theory}
Consider a compact open subgroup $K\triangleleft G(\bb{A}^\infty)$ of the finite adelic points of $G$ that satisfies $U(\frak{N})\subset K\subset U_0(\frak{N})$. Suppose that $\cal{V}$ is the valuation ring corresponding to the fixed embedding $\iota_p:F^\text{Gal}\hookrightarrow\overline{\bb{Q}}_p$, so that we may assume $\{y^v\}=1$ whenever the ideal $y\cal{O}_F$ generated by $y$ is prime to $p\cal{O}_F$.
Let $\varpi$ be a uniformizer of the completion $\cal{O}_{F,\frak{q}}$ of $\cal{O}_F$ at a prime $\frak{q}$ and let $m\in\bb{N}$ be a positive integer. We are interested in Hecke operators defined by the following double cosets 
\[
T_0(\varpi)=\{\varpi^{-v}\}\left[U(\frak{N})\begin{pmatrix}
\varpi & 0 \\ 
0 & 1
\end{pmatrix}U(\frak{N})\right] \qquad \text{if}\ \frak{q}\nmid\frak{N},
\]
\[
U_0(\varpi^m)=\{\varpi^{-mv}\}\left[U(\frak{N})\begin{pmatrix}
\varpi^m & 0 \\ 
0 & 1
\end{pmatrix}U(\frak{N})\right] \qquad \text{if}\ \frak{q}\mid\frak{N},
\]
and for $a,b\in\cal{O}_{F,\frak{N}}^\times:=\prod_{\frak{q}\mid\frak{N}}\cal{O}_{F,\frak{q}}^\times$ the double cosets
\[
T(a,b)=\left[U(\frak{N})\begin{pmatrix}
a & 0 \\ 
0 & b
\end{pmatrix} U(\frak{N})\right].
\]
If the prime $\frak{q}$ is coprime to the level, then the Hecke operator $T_0(\varpi)$ acting on modular forms is independent of the choice of the uniformizer $\varpi$ and we simply denote it $T_0(\frak{q})$. For any finite adelic point $z\in Z_G(\bb{A}^\infty)$ of the center of $G$ we define the diamond operator associated to it by $\msf{f}_{\lvert\langle z\rangle}(x)=\msf{f}(xz)$. For a prime ideal $\frak{q}$ such that $\text{GL}_2(\cal{O}_{F,\frak{q}})\subset K$, we write $\langle\frak{q}\rangle$ for the operator $\langle \varpi\rangle$, where $\varpi$ is a uniformizer of $\cal{O}_{F,\frak{q}}$. The action of the operators on adelic q-expansion is given by the following formulas. If $\frak{q}\nmid\frak{N}$ one can compute 
\[
\mathsf{a}_p(y, \msf{f}_{\lvert T_0(\frak{q})})=\mathsf{a}_p(y\varpi,\msf{f})\{\varpi^{-v}\}\varpi_p^v+\text{N}_{F/\bb{Q}}(\frak{q})\{\frak{q}^{-2v}\}\mathsf{a}_p(y\varpi^{-1},\msf{f}_{\lvert\langle\frak{q}\rangle})\{\varpi^v\}\varpi_p^{-v}
\] and 
\[
\mathsf{a}(y, \msf{f}_{\lvert T_0(\frak{q})})=\mathsf{a}(y\varpi,\msf{f})+\text{N}_{F/\bb{Q}}(\frak{q})\{\frak{q}^{-2v}\}\mathsf{a}(y\varpi^{-1}, \msf{f}_{\lvert\langle\frak{q}\rangle}).
\]
If $\frak{q}\mid\frak{N}$ one can compute 
\[
\mathsf{a}_p(y, \msf{f}_{\lvert U_0(\varpi^m)})=\mathsf{a}_p(y\varpi^m,\msf{f})\{\varpi^{-mv}\}\varpi_p^{mv}
\]
and 
\[
\mathsf{a}(y, \msf{f}_{\lvert U_0(\varpi^m)})=\mathsf{a}(y\varpi^m,\msf{f}).
\]
Finally, for $a,b\in\cal{O}_{F,\frak{N}}^\times$ one finds $\msf{a}_p(y,\msf{f}_{\lvert T(a,1)})=\msf{a}_p(ya,\msf{f})a_p^v$.

The Hecke algebra $\mathsf{h}_{k,w}(K;\cal{V})$ is defined to be the $\cal{V}$-subalgebra of $\text{End}_\bb{C}\left(S_{k,w}(K;\bb{C})\right)$ generated by the Hecke operators $T_0(\frak{q})$'s for primes outside the level $\frak{q}\nmid\frak{N}$, $U_0(\varpi^m)$'s for primes dividing the level $\frak{q}\mid\frak{N}$, and $T(a,b)$'s for $a,b\in\cal{O}^\times_{F,\frak{N}}$. For each $\cal{V}$-algebra $A$ contained in $\bb{C}$ one defines $\mathsf{h}_{k,w}(K;A)=\mathsf{h}_{k,w}(K;\cal{V})\otimes_\cal{V}A$ and denotes by $S_{k,w}(K;A)$ the set $\big\{\msf{f}\in S_{k,w}(K;\bb{C})\big\lvert\ \mathsf{a}(y,\msf{f})\in A\ \forall y\in\bb{A}^\times_{F,+}\big\}$. The objects just defined are very well-behaved. 

\begin{theorem}{(\cite{pHida} Theorem 2.2)}
For any finite field extension $L/F^\text{Gal}$ and any $\cal{V}$-subalgebra $A$ of $L$, there is a natural isomorphism $S_{k,w}(K;L)\cong S_{k,w}(K;A)\otimes_AL$. Moreover, if $A$ an integrally closed domain containing $\cal{V}$, finite flat over either $\cal{V}$ or $\bb{Z}_p$, then $S_{k,w}(K;A)$ is stable under $\mathsf{h}_{k,w}(K;A)$ and the pairing $(,):S_{k,w}(K;A)\times \mathsf{h}_{k,w}(K;A)\to A$ given by $(f,h)=\mathsf{a}(1,f_{\lvert h})$ induces isomorphisms of $A$-modules
\[
\mathsf{h}_{k,w}(K;A)\cong S_{k,w}(K;A)^*\ \ \text{and}\ \ \ S_{k,w}(K;A)\cong \mathsf{h}_{k,w}(K;A)^*,
\]
where $(-)^*$ denotes the $A$-linear dual $\text{Hom}_A\left(-,A\right)$.
\end{theorem}

A cuspform that is an eigenvector for all the Hecke operators is called an eigenform and it is normalized when $\mathsf{a}(1,\msf{f})=1$. Shimura proved (\cite{Shimura} Proposition 2.2) that the eigenvalues for the Hecke operators are algebraic numbers,  hence a normalized eigenform $\msf{f}\in S_{k,w}(K;\bb{C})$ is an element of $ S_{k,w}(K;\overline{\bb{Q}})$ since the $T_0(y)$-eigenvalue is $\mathsf{a}(y,\msf{f})$ for every idele $y$ where $T_0(y)=\{y^{-v}\}T(y)$ for $T(y)$ as in $(\ref{general Hecke operator})$.

\begin{definition}
Let $\frak{p}\mid p$ be a prime of $\cal{O}_F$ coprime to the level $K$. A normalized eigenform $\msf{f}\in S_{k,w}(K;\overline{\bb{Q}})$ is nearly ordinary at $\frak{p}$ if the $T_0(\frak{p})$-eigenvalue is a $p$-adic unit with respect to the specified embedding $\iota_p:\overline{\bb{Q}}\hookrightarrow\overline{\bb{Q}}_p$.
If $\msf{f}$ is nearly ordinary at $\frak{p}$ for all $\frak{p}\mid p$ we say that $\msf{f}$ is $p$-nearly ordinary.
\end{definition}

\begin{definition}
For every idele $m\in\bb{A}_F^\times$ there is an operator $V(m)$ on cuspforms of weight $(k,w)$ defined by 
\[
\msf{f}_{\lvert V(m)}(x)=\text{N}_{F/\bb{Q}}(m\cal{O}_F)\msf{f}\left(x\begin{pmatrix}
m^{-1}&0\\
0&1
\end{pmatrix}\right)
\]
 that acts on $p$-adic $q$-expansions as $\mathsf{a}_p(y,\msf{f}_{\lvert V(m)})=m_p^{-v}\mathsf{a}_p(ym^{-1},\msf{f})$ (this operator is denoted by $[m]$ in \cite{pHida} Section $7\text{B}$). Its normalization $[m]=\{m^v\}V(m)$ acts on $q$-expansions by $\mathsf{a}(y,\msf{f}_{\lvert[m]})=\mathsf{a}(ym^{-1},\msf{f})$. 
\end{definition}
\begin{lemma}\label{lemma: central character and V(p)}
	For all couple of weights $(k,u), (k,u')\in \bb{Z}[I]^2$ we the equality 
	\[
	V(p)\circ\Psi_{u,u'}=p^{u'-u} \Psi_{u,u'}\circ V(p)
	\]
	of maps  from $S^*_{k,u}(K,\bb{C})$ to $S^*_{k,u'}(K,\bb{C})$.
\end{lemma}
\begin{proof}
	Let $\msf{f}\in S^*_{k,u}(K,\bb{C})$, then 
	\[\begin{split}
	\Psi_{u,u'}(\msf{f})_{\lvert V(p)}(x)&=\text{N}_{F/\bb{Q}}(p)^{-1}\msf{f}\left(x\begin{pmatrix}
	p_p^{-1} &0\\
	0& 1
	\end{pmatrix}\right)\left\lvert\det\left(x\begin{pmatrix}
	p_p^{-1} &0\\
	0& 1
	\end{pmatrix}\right)\right\rvert_{\bb{A}_\bb{Q}}^{u'-u}\\
	&=p^{u'-u}\Psi_{u,u'}(\msf{f}_{\lvert V(p)})(x).
	\end{split}\]
\end{proof}
Let $\msf{f}\in S_{k,w}(K,\overline{\bb{Q}})$ be a normalized eigenform of level prime to $\frak{p}$. Set $\langle\frak{p}\rangle_0:=\{\varpi_\frak{p}^{-2v}\}\langle\frak{p}\rangle$, then the $\langle\frak{p}\rangle_0$-eigenvalue of $\msf{f}$ is $\psi_{\msf{f},0}(\frak{p})=\{\varpi_\frak{p}^{-2v}\}\psi_\msf{f}(\frak{p})$ and we can write the Hecke polynomial for $T_0(\frak{p})$ as
\[
1-\mathsf{a}(\frak{p},\msf{f})X+\text{N}_{F/\bb{Q}}(\frak{p})\psi_{\msf{f},0}(\frak{p})X^2=(1-\alpha_{0,\frak{p}}X)(1-\beta_{0,\frak{p}}X).
\]
If $\msf{f}$ is nearly ordinary at $\frak{p}$, $\mathsf{a}(\frak{p},\msf{f})$ is a $p$-adic unit and we can assume that $\alpha_{0,\frak{p}}$ is a $p$-adic unit too. The nearly ordinary $\frak{p}$-stabilization of $\msf{f}$ is the cuspform $\msf{f}^{(\frak{p})}=(1-\beta_{0,\frak{p}}[\varpi_\frak{p}])\msf{f}$ that has the same Hecke eigenvalues of $\msf{f}$ away from $\frak{p}$ and whose $U_0(\varpi_\frak{p})$-eigenvalue is $\alpha_{0,\frak{p}}$. For $\cal{Q}$ a finite set of prime $\cal{O}_F$-ideals, the $\cal{Q}$-depletion of a cuspform $\msf{f}$ is the cuspform $
\msf{f}^{[\cal{Q}]}=\prod_{\frak{p}\in\cal{P}}\left(1-V(\varpi_\frak{p})\circ U(\varpi_\frak{p})\right)\msf{f}$ whose Fourier coefficient $\mathsf{a}_p(y,\msf{f}^{[\cal{Q}]})$ equals $\mathsf{a}_p(y,\msf{f})$ if $y_\cal{Q}\in \cal{O}_{F,\cal{Q}}^\times$ and $0$ otherwise.

\subsection{Hida families}
\label{subsec: HidaFamilies}
Fix a prime $p$ and a valuation ring $O$ in $\overline{\bb{Q}}_p$ finite flat over $\bb{Z}_p$ containing $\iota_p(\cal{V})$. We consider compact open subgroups that satisfy $V_1(\frak{N})\subseteq K\subseteq U_0(\frak{N})$ for an ideal $\frak{N}$ of $\cal{O}_F$ prime to $p$. Let $K(p^\alpha)=K\cap U(p^\alpha)$ and for $A\in\{O,\text{Frac}(O),\bb{C}_p\}$ consider the space of $p$-adic cuspforms
\[
S_{k,w}(K(p^\infty); A)=\underset{\alpha, \to}{\lim}\ S_{k,w}(K(p^\alpha);A)
\]
on which the $p$-adic Hecke algebra
\[
\msf{h}_{k,w}(K(p^\infty); A)=\underset{\leftarrow, \alpha}{\lim}\ \msf{h}_{k,w}(K(p^\alpha);A)
\]
naturally acts. Every idele $y\in\widehat{\cal{O}_F}\cap\bb{A}_F^\times$ can be written uniquely as $y=a\prod_\frak {q}\varpi_\frak{q}^{e(\frak{q})}u$ with $u\in\det U(\frak{N})$, $a\in\cal{O}_{F,\frak{N}}^\times$. Write $\frak{n}$ for the ideal $\big(\prod_{\frak{q}\nmid\frak{N}}\varpi_\frak{q}^{e(\frak{q})}\big)\cal{O}_F$ and then define
\begin{equation}\label{general Hecke operator}
T(y)= T(a,1)T(\frak{n})\prod_{\frak{q}\mid\frak{N}}U(\varpi_\frak{q}^{e(\frak{q})}).
\end{equation}
The Hecke operators defined by $\mathbf{T}(y)=\underset{\leftarrow}{\lim}\ T(y)y_p^{-v}$ play an important role in the theory.
There is a $p$-adic norm on the space of $p$-adic cuspforms $S_{k,w}(K(p^\infty); A)$  defined by $\lvert\msf{f}\rvert_p=\text{sup}_y\{\lvert \msf{a}_p(y,\msf{f})\rvert_p\}$; the resulting completed space is denoted by $\overline{S}_{k,w}(K(p^\infty); A)$ and it has a natural perfect $O$-pairing with the $p$-adic Hecke algebra (\cite{pHida} Theorem 3.1).
Each element $\msf{f}\in\overline{S}_{k,w}(K(p^\infty); A)$ induces a continuous function $\msf{f}:\frak{J}\longrightarrow A$, defined by $y\mapsto\msf{a}_p(y,\msf{f})$, on the topological semigroup
\[
\frak{J}=\widehat{\cal{O}_F}F^\times_{\infty,+}/\det U(p^\infty)F^\times_{\infty,+},
\]
isomorphic to $\cal{O}_{F,p}^\times\times \scr{I}_F$ for $\scr{I}_F$ the free semigroup of integral ideals of $F$. Hence, there is a continuous embedding $\overline{S}_{k,w}(K(p^\infty); A)\hookrightarrow\cal{C}(\frak{J};A)$ of the completed space of $p$-adic cuspforms into the continuous functions from $\frak{J}$ to $A$. The image of the embedding $\overline{\mathbf{S}}(K)$ is independent of the weight $(k,w)$ since there exists a canonical algebra isomorphism $
\mathsf{h}_{k,w}\left(K(p^\infty);O\right)\cong \mathsf{h}_{2t,t}\left(K(p^\infty);O\right)$
which takes $\mathbf{T}(y)$ to $\mathbf{T}(y)$ (\cite{nearlyHida} Theorem 2.3).
From now on, $\overline{\mathbf{S}}(\frak{N})$ and $\mathbf{h}(\frak{N};O)$ stand for  $\overline{\mathbf{S}}(V_1(\frak{N}))$ and $\mathbf{h}(V_1(\frak{N});O)$ respectively.

One can decompose the compact ring $\mathbf{h}(K;O)$ as a direct sum of algebras $\mathbf{h}(K;O)=\mathbf{h}^{\text{n.o.}}(K;O)\oplus\mathbf{h}^{\text{ss}}(K;O)$ in such a way that $\mathbf{T}(p)$ is a unit in $\mathbf{h}^{\text{n.o.}}(K;O)$ and it is topologically nilpotent in $\mathbf{h}^{\text{ss}}(K;O)$. Furthermore, the idempotent $e_{\text{n.o.}}$ of the nearly ordinary part $\mathbf{h}^{\text{n.o.}}(K;O)$ has the familiar expression $e_{\text{n.o.}}=\underset{n\to\infty}{\lim}\ \mathbf{T}(p)^{n!}$. Let $\overline{\mathbf{S}}^{\text{n.o.}}(K)=e_{\text{n.o.}}\overline{\mathbf{S}}(K)$ be the space of nearly ordinary $p$-adic cuspforms. We consider the topological group
\[
\bb{G}(\frak{N})=Z(\frak{N})\times\cal{O}_{F,p}^\times=\bb{A}^\times_F/\overline{F^\times\det U(\frak{N}p^\infty)F^\times_{\infty,+}}\times\cal{O}_{F,p}^\times
\]
equipped with a continuous group homomorphism $\bb{G}(\frak{N})\longrightarrow \mathbf{h}^{\text{n.o.}}(K;O)^\times$ given by $\langle z,a\rangle\mapsto T(a^{-1},1)\langle z\rangle$. Fix a decomposition $\bb{G}(\frak{N})\cong\bb{G}_\text{tor}(\frak{N})\times\mathbf{W}$ consisting of the torsion subgroup $\bb{G}_\text{tor}(\frak{N})$ and some $\bb{Z}_p$-torsion free $\mathbf{W}$. Then $\mathbf{W}\cong\bb{Z}_p^r$ for $r=[F:\bb{Q}]+1+\delta$ where $\delta$ is Leopoldt's defect for $F$ and we denote by $O[[\mathbf{W}]]\cong O[[X_1,\dots,X_r]]$ the completed group ring.
\begin{theorem}{(\cite{nearlyHida} Theorem 2.4)}
The universal nearly ordinary Hecke algebra $\mathbf{h}^{\text{n.o.}}(K;O)$ is finite and torsion-free over $\mathbf{A}=O[[\mathbf{W}]]$.
\end{theorem}

\begin{definition}\label{def: I-adic cuspforms}
Let $K$ be a compact open subgroup which satisfies $V_1(\frak{N})\subseteq K\subseteq U_0(\frak{N})$ for an ideal $\frak{N}$ of $\cal{O}_F$ prime to $p$. For an $\mathbf{A}$-algebra $\mathbf{I}$ we define the space of nearly ordinary $\mathbf{I}$-adic cuspforms of tame level $K$ to be
$\overline{\mathbf{S}}^{\text{n.o.}}(K;\mathbf{I})=\text{Hom}_{\mathbf{A}-\text{mod}}(\mathbf{h}^{\text{n.o.}}(K;O),\mathbf{I})$. We call Hida families those homomorphisms that are homomorphisms of $\mathbf{A}$-algebras .
\end{definition}
When $\mathbf{I}$ is an $\mathbf{A}$-algebra domain, $\overline{\mathbf{S}}^{\text{n.o.}}(K;\mathbf{I})$ is a finite torsion-free $\mathbf{I}$-module. Indeed it is torsion-free because $\mathbf{I}$ is a domain, and $\overline{\mathbf{S}}^{\text{n.o.}}(K;\mathbf{A})$ is a finite $\mathbf{A}$-module being the dual of a finite module over a noetherian ring.

We want to motivate the definition of nearly ordinary $\mathbf{I}$-adic cuspforms. For each pair of characters $\psi: \text{cl}^+_F(\frak{N}p^\alpha)\longrightarrow O^\times$, $\psi': \left(\cal{O}_F/p^\alpha\right)^\times\longrightarrow O^\times$ there is a space of cuspforms
\[
S_{k,w}(\frak{N}p^\alpha,\psi,\psi';O)=\left\{\msf{f}\in S_{k,w}(U(\frak{N}p^\alpha);O)\bigg\lvert\ \forall (z,a)\in\bb{G}(\frak{N})\ \msf{f}_{\lvert\langle z,a\rangle}=\psi(z)\psi'(a)\cal{N}(z)^{[n+2v]}a^v\msf{f} \right\},
\]
and given a pair of characters $(\psi,\psi')$ and elements $n,v\in\bb{N}[I]^2$, one can define the algebra homomorphism $\text{P}_{n,v,\psi,\psi'}:O[[\bb{G}(\frak{N})]]\to O$ which corresponds to the character $\bb{G}(\frak{N})\to O^\times$ determined by the rule $(z,a)\mapsto\psi(z)\psi'(a)\cal{N}(z)^{[n+2v]}a^v$.
Let's fix $\overline{\mathbf{L}}$ an algebraic closure of the fraction field $\mathbf{L}$ of $\mathbf{A}$, and an embedding of $\overline{\bb{Q}}_p$ into $\overline{\mathbf{L}}$. Suppose  $\cal{F}:\mathbf{h}^{\text{n.o.}}(K;O)\to \overline{\mathbf{L}}$ is an $\mathbf{A}$-linear map; since the universal nearly ordinary Hecke algebra is finite over $\mathbf{A}$, the image of $\cal{F}$ is contained in the integral closure $\mathbf{I}$ of  $\mathbf{A}$ in a finite extension $\mathbf{K}$ of $\mathbf{L}$. Let $\cal{A}(\mathbf{I})$ be the set of \emph{arithmetic points}, i.e., the subset of $\text{Hom}_{O-\text{alg}}(\mathbf{I},\overline{\bb{Q}}_p)$ consisting of homomorphisms that coincide with $\text{P}_{n,v,\psi,\psi'}$ on $\mathbf{A}$ for some $n,v\geq0$ and characters $\psi,\psi'$. Then for all $P\in\cal{A}(\mathbf{I})$ the composite $\cal{F}_P=P\circ\cal{F}$  induces a $\overline{\bb{Q}}_p$-linear map for suitable $\alpha>0$,
\[
\cal{F}_P:\msf{h}^{\text{n.o.}}_{k,w}(\frak{N}p^\alpha;\overline{\bb{Q}}_p)\longrightarrow\overline{\bb{Q}}_p
\]
and  $k=n(P)+2$, $w=t-v(P)$ (\cite{nearlyHida} Theorem 2.4). Therefore, the duality between Hecke algebra and cuspforms produces a unique $p$-adic cuspform $\msf{f}_P\in S^{\text{n.o.}}_{k,w}(\frak{N}p^\alpha;\overline{\bb{Q}}_p)$ that satisfies $\msf{a}_p(y,\msf{f}_P)=\cal{F}_P(\mathbf{T}(y))$ for all integral ideles $y$. In other words, any nearly ordinary $\mathbf{I}$-adic cuspform $\cal{F}\in\overline{\mathbf{S}}^{\text{n.o.}}(K;\mathbf{I})$ gives rise to a family of cuspforms parametrized by $\cal{A}(\mathbf{I})$. Furthermore, if $\cal{F}$ is Hida family, each specialization $\msf{f}_P$ is an eigenform and so classical i.e. an element of $S^{\text{n.o.}}_{k,w}(\frak{N}p^\alpha,\psi,\psi';\overline{\bb{Q}})$.
Moreover, if $\msf{f}\in S_{k,w}(\frak{N}p^\alpha;\overline{\bb{Q}})$ is an eigenform for all Hecke operators  such that its $U_0(p)$-eigenvalue is a $p$-adic unit with respect to the fixed $p$-adic embedding $\iota_p$, then there exists a finite integrally closed extension $\mathbf{I}$ of $\mathbf{A}$ and a nearly ordinary $\mathbf{I}$-adic cuspform of tame level $\frak{N}$ passing through $\msf{f}$ (\cite{nearlyHida} Theorem 2.4).
\begin{remark}
A Hida family $\cal{F}\in\overline{\mathbf{S}}^\text{n.o.}(\frak{N};\mathbf{I})$ defines a projector $e_\cal{F}$ on $\overline{\mathbf{S}}^\text{n.o.}(\frak{N};\mathbf{I}\otimes_{\mathbf{A}}\mathbf{L})$. Indeed, the universal Hecke algebra is reduced being a projective limit of reduced algebras (\cite{Miyake} Theorem A), thus $\mathbf{h}^\text{n.o.}(\frak{N};O)\otimes_{\mathbf{A}}\mathbf{L}$ is a reduced Artinian $\mathbf{L}$-algebra, i.e., a product of finite field extensions of $\mathbf{L}$ indexed by the Hida families of tame level $\frak{N}$.
\end{remark}


\subsection{Diagonal restriction}
\label{subsec: Diagonal restriction}
We modify some of the notations for distinguishing between spaces of Hilbert modular forms over different totally real number fields. For a totally real number field $F$ we denote by $S_{k,w}(U(\frak{N}),F;\bb{C})$ the space of cuspforms over $F$ of weight $(k,w)\in\bb{Z}[I_F]^2$ and level $U(\frak{N})$. We set $t_F=\sum_{\tau\in I_F}\tau$ and we denote by $\msf{d}_F$ an idele that generates the absolute different of $F$, $\msf{d}_F\cal{O}_F=\frak{d}_F$. 
If $L/F$ is an extension of totally real fields, there is a restriction map $I_L\to I_F$ which induces a group homomorphism $\bb{Z}[I_L]\to\bb{Z}[I_F]$ denoted by $k\mapsto k_{\lvert F}$ and has the property $(t_L)_{\lvert F}=[L:F]\cdot t_F$. Let $\frak{N}$ an ideal of $\cal{O}_F$, the natural inclusion $\zeta:\text{GL}_2(\bb{A}_F)\hookrightarrow\text{GL}_2(\bb{A}_L)$ defines by composition a \emph{diagonal restriction} map $
\zeta^*:S_{k,w}(U(\frak{N}),L;\bb{C})\longrightarrow S_{k_{\lvert F},w_{\lvert F}}(U(\frak{N}),F;\bb{C})$.
Write 
\[
\bb{A}^\times_F=\coprod_{i=1}^{h^+_F(\frak{N})}F^\times a'_i\msf{d}_F^{-1}\det U(\frak{N})F^\times_{\infty,+},\qquad\qquad\bb{A}^\times_L=\coprod_{j=1}^{h^+_L(\frak{N})}L^\times b'_j\msf{d}_L^{-1}\det U(\frak{N})L^\times_{\infty,+}
\]
for $a'_i\in\widehat{\cal{O}_F}\cap\bb{A}_F^{\times}$, $b'_j\in\widehat{\cal{O}_L}\cap\bb{A}_L^{\times}$ and set $a_i=a'_i\msf{d}_F^{-1}$ and $b_j=b'_j\msf{d}_L^{-1}$ for all $i,j$. If $p$ is unramified in $L$ then we can assume that $a_{i,p}=1$, $b_{j,p}=1$ for all $i,j$. The inclusion $\bb{A}_F^\times\hookrightarrow\bb{A}^\times_L$ induces a group homomorphism $\zeta: \text{cl}^+_F(\frak{N})\rightarrow \text{cl}^+_L(\frak{N})$ that determines a map of sets $\zeta:\{1,\dots,h^+_F(\frak{N})\}\to\{1,\dots,h^+_L(\frak{N})\}$ by the rule $\zeta[a_i]=[b_{\zeta(i)}]$. We name $\frak{a}_i$ the fractional ideal of $F$ generated by $a_i$.

\begin{lemma}\label{lemma: autom diag}
 The diagonal restriction of a cuspform $\msf{g}\in S_{k,w}(U(\frak{N}),L;\bb{C})$ is a cuspform $\zeta^*\msf{g}\in S_{k_{\lvert F},w_{\lvert F}}(U(\frak{N}),F;\bb{C})$ whose Fourier coefficients are given by the formula
\[
a\left(\eta,\left(\zeta^*\msf{g}\right)_i\right)=\underset{\xi\in(\frak{a}_{i}\cal{O}_L)_+,\ \text{tr}_{L/F}(\xi)=\eta}{\sum} a(\xi,\msf{g}_{\zeta(i)})
\]
for every $\eta\in(\frak{a}_i)_+$ and every $i=1,\dots, h^+_F(\frak{N})$.
\end{lemma}
\begin{proof}
The inclusion $\zeta:\text{GL}_2(\bb{A}_F)\hookrightarrow\text{GL}_2(\bb{A}_L)$ of adelic points induces the map $\zeta:\frak{H}^{I_F}\to\frak{H}^{I_L}$ of symmetric domains whose $\nu$-component, $\nu\in I_L$, is $\zeta(z)_\nu=z_{\nu_{\lvert F}}$. We compute
\[
\left(\zeta^*\msf{g}\right)_i(z)=\underset{\xi\in (\frak{a}_{i}\cal{O}_L)_+}{\sum} 		a(\xi,\msf{g}_{\zeta(i)})e_L(\xi\zeta(z))=\underset{\eta\in (\frak{a}_i)_+}{\sum}\left(\underset{\text{tr}_{L/F}(\xi)=\eta}{\sum} a(\xi,\msf{g}_{\zeta(i)})\right) e_F(\eta z).
\]
since $\text{tr}_{L/F}(\frak{a}_i\cal{O}_L)=\frak{a}_i\text{tr}_{L/F}(\cal{O}_L)\subset\frak{a}_i$ and totally positive elements are mapped to totally positive elements by the trace map.
\end{proof}

\begin{lemma}\label{lemma: p-adic diag}
Let $p$ be an unramified prime in $L$. For $\msf{g}\in S_{k,w}(U(\frak{N}),L;\overline{\bb{Q}})$ a cuspform, we can compute the $p$-adic Fourier coefficients of its diagonal restriction $\zeta^*\msf{g}\in S_{k_{\lvert F},w_{\lvert F}}(U(\frak{N}),F;\overline{\bb{Q}})$ as follows. Let $y=\eta a_i^{-1}u\in\widehat{\cal{O}_F} F^\times_{\infty,+}$ then
\[
\mathsf{a}_p(y,\zeta^*\msf{g})=C(a_i,u)\underset{\xi\in(\frak{a}_i\cal{O}_L)_+,\ \text{tr}_{L/F}(\xi)=\eta}{\sum} \mathsf{a}_p(y_\xi,\msf{g})
\]
for the ideles $y_\xi=\xi a_i^{-1}u\in\widehat{\cal{O}_L} L^\times_{\infty,+}$ indexed by $\xi\in(\frak{a}_i\cal{O}_L)_+$ with $\text{tr}_{L/F}(\xi)=\eta$, and where the constant $C(a_i,u)$ is equal to $C(a_i,u)=\cal{N}_L(a_i\msf{d}_L)\cal{N}_F(a_i\msf{d}_F)^{-1}u_p^{([L:F]-1)t_F}$.
\end{lemma}
\begin{proof}
The proof is a direct computation. 
\[\begin{split}
\mathsf{a}_p(y,\zeta^*\msf{g})&= a(\eta,(\zeta^*\msf{g})_i)y_p^{w_{\lvert F}-t_F}\eta^{t_F-w_{\lvert F}}\cal{N}_F(a_i\msf{d}_F)^{-1}\\
	&=\underset{\xi\in(\frak{a}_i\cal{O}_L)_+,\ \text{tr}_{L/F}(\xi)=\eta}{\sum} a(\xi,\msf{g}_{\zeta(i)})y_p^{w_{\lvert F}-t_F}\eta^{t_F-w_{\lvert F}}\cal{N}_F(a_i\msf{d}_F)^{-1}\\
	&=\underset{\text{tr}_{L/F}(\xi)=\eta}{\sum} \mathsf{a}_p(y_\xi,\msf{g})(y_\xi)_p^{t_L-w}\xi^{w-t_L}\cal{N}_L(a_i\msf{d}_L)y_p^{w_{\lvert F}-t_F}\eta^{t_F-w_{\lvert F}}\cal{N}_F(a_i\msf{d}_F)^{-1}\\
	&=\underset{\text{tr}_{L/F}(\xi)=\eta}{\sum} \mathsf{a}_p(y_\xi,\msf{g})(\xi a_i^{-1}u)_p^{t_L-w}\xi^{w-t_L}(\eta a_i^{-1}u)_p^{w_{\lvert F}-t_F}\eta^{t_F-w_{\lvert F}}\cal{N}_L(a_i\msf{d}_L)\cal{N}_F(a_i\msf{d}_F)^{-1}\\
	&=\cal{N}_L(a_i\msf{d}_L)\cal{N}_F(a_i\msf{d}_F)^{-1}u_p^{([L:F]-1)t_F}\underset{\text{tr}_{L/F}(\xi)=\eta}{\sum} \mathsf{a}_p(y_\xi,\msf{g}).
\end{split}\]
\end{proof}

\begin{proposition}\label{prop: V operator}
Let $p$ be an unramified prime in $L$ and for $\wp\mid p$ a prime $\cal{O}_F$-ideal $\varpi_\wp$ be a uniformizer of $\cal{O}_{F,\wp}$. Any cuspform $\msf{g}\in S_{k,w}(\frak{N}p,L;\overline{\bb{Q}}_p)$ with $\frak{N}$ any $\cal{O}_F$-ideal satisfies
\[
\zeta^*(\msf{g}_{\lvert V(\varpi_\wp)})=\text{N}_{F/\bb{Q}}(\wp)^{1-[L:F]}(\zeta^*\msf{g})_{\lvert V(\varpi_\wp)}.
\]
\end{proposition}
\begin{proof}
Let $y=\eta a_i^{-1}u\in\widehat{\cal{O}_F} F^\times_{\infty,+}$ an idele of $F$, note that $(a_i^{-1}u)_p$ is unit under our assumptions. Therefore $y\varpi_\wp^{-1}\in\widehat{\cal{O}_F} F^\times_{\infty,+}$ if and only if $\eta\in\wp\cal{O}_F$ and that similarly for any $\xi\in L_+^\times$  we have that $y_\xi\varpi_\wp^{-1}\in\widehat{\cal{O}_L}L^\times_{\infty,+}$ if and only if $\xi\in\wp\cal{O}_L$. The class $[a_i][\varpi_\wp]\in\text{cl}^+_F(\frak{N}p)$ is represented by some $a_j$ so we can write $a_i\varpi_\wp=\varrho a_j\nu$ for $\varrho\in F_+^\times$, $\nu\in\det U_F(\frak{N}p)$. On the one hand, we compute that 
\[\begin{split}
\mathsf{a}_p(y,\zeta^*(\msf{g}_{\lvert V(\varpi_\wp)})&=C(a_i,u)\underset{\xi\in(\frak{a}_i\cal{O}_L)_+,\ \text{tr}_{L/F}(\xi)=\eta}{\sum} \mathsf{a}_p(y_\xi,\msf{g}_{\lvert V(\varpi_\wp)})\\
	&=C(a_i,u)\underset{\text{tr}_{L/F}(\xi)=\eta}{\sum} \varpi_\wp^{-v}\mathsf{a}_p(y_\xi \varpi_\wp^{-1},\msf{g})\\
	&=C(a_i,u)\underset{\xi\in(\frak{a}_i\cal{O}_L)_+\cap\wp\cal{O}_L,\ \text{tr}_{L/F}(\xi)=\eta}{\sum} \varpi_\wp^{-v}\mathsf{a}_p(\xi\varrho^{-1}a_j^{-1}\nu^{-1}u,\msf{g}).
\end{split}\]
On the other hand, if $\eta\in \wp\cal{O}_F$ we have $y\varpi_\wp^{-1}=\eta\varrho^{-1}a_j^{-1}\nu^{-1}u$ and
\[\begin{split}
\mathsf{a}_p(y,\left(\zeta^*\msf{g}\right)_{\lvert V(\varpi_\wp)})&=\varpi_\wp^{w_{\lvert F}-t_F}\mathsf{a}_p(\eta\varrho^{-1}a_j^{-1}\nu^{-1}u,\zeta^*\msf{g})\\
	&=\varpi_\wp^{w_{\lvert F}-t_F}C(a_j,\nu^{-1}u)\underset{\epsilon\in(\frak{a}_j\cal{O}_L)_+;\ \text{tr}_{L/F}(\epsilon)=\eta\varrho^{-1}}{\sum} \msf{a}_p(\epsilon a_j^{-1}\nu^{-1}u ,\msf{g}),
\end{split}\]
while $\mathsf{a}_p(y,\left(\zeta^*\msf{g}\right)_{\lvert V(\wp)})=0$ if $\eta\notin \wp\cal{O}_F$. 

When $\eta\in\wp\cal{O}_F$ there is a bijection between $\{\xi\in (\frak{a}_i\cal{O}_L)_+\cap\wp\cal{O}_L\lvert \ \text{tr}_{L/F}(\xi)=\eta\}$ and $\{\epsilon\in (\frak{a}_j\cal{O}_L)_+\lvert\  \text{tr}_{L/F}(\epsilon)=\eta\varrho^{-1}\}$ given by $\xi\mapsto \xi\varrho^{-1}$, therefore we find that
\[\begin{split}
\mathsf{a}_p(y,\zeta^*(\msf{g}_{\lvert V(\varpi_\wp)})&=\frac{C(a_i,u)}{C(a_j,\nu^{-1}u)}\varpi_\wp^{(1-[L:F])t_F}\mathsf{a}_p(y,\left(\zeta^*\msf{g}\right)_{\lvert V(\varpi_\wp)})\\
	&=\big(\cal{N}_F(\nu)\cal{N}_F(\varpi_\wp)^{-1}\big(\frac{\nu_p}{\varpi_\wp}\big)^{t_F}\big)^{[L:F]-1}\mathsf{a}_p(y,\left(\zeta^*\msf{g}\right)_{\lvert V(\varpi_\wp)})\\
	&=\text{N}_{F/\bb{Q}}(\wp)^{1-[L:F]}\mathsf{a}_p(y,\left(\zeta^*\msf{g}\right)_{\lvert V(\varpi_\wp)}).
\end{split}
\]
\end{proof}

We claim that diamond operators and operators $T(a,1)$ for $a\in\cal{O}_{F,p}^\times$ commute with diagonal restriction. Indeed, the claim is clear for diamond operators and following the proof of Proposition $\ref{prop: V operator}$ with $V(\varpi_\wp)$ replaced by $T(a,1)$ and by noticing that $a\cal{O}_F=\cal{O}_F$ we see that for any cuspform $\msf{g}$ we have 
\begin{equation}\label{eq: diamonds}
\left(\zeta^*\msf{g}\right)_{\lvert\langle z\rangle}=\zeta^*\left(\msf{g}_{\lvert\langle z\rangle}\right)\qquad\text{and}\qquad \left(\zeta^*\msf{g}\right)_{\lvert T(a,1)}=\zeta^*\left(\msf{g}_{\lvert T(a,1)}\right).
\end{equation}

\begin{definition}\label{def: universal diag}
Let $L/F$ be an extension of totally real number fields unramified at $p$ and let $\frak{N}$ be an $\cal{O}_F$-ideal prime to $p$. Assume that the valuation ring $O$ contains all the completions of $\cal{O}_F$ at primes above $p$. The diagonal restriction of cuspforms induces by duality an $O$-modules map between universal Hecke algebras
\[
\zeta:\mathbf{h}_F(\frak{N};O)\longrightarrow\mathbf{h}_L(\frak{N};O).
\] 
The image $\zeta(\mathbf{T}(y))$ is determined by the equality 
\[
\msf{a}_p\big(1,\msf{g}_{\lvert\zeta(\mathbf{T}(y))}\big)=\msf{a}_p\big(1,(\zeta^*\msf{g})_{\lvert\mathbf{T}(y)}\big)\qquad \forall \msf{g}\in\overline{\mathbf{S}}_L(\frak{N}).
\]
If we endow $\mathbf{A}_L$ with the natural $\mathbf{A}_F$-algebra structure, then one can check that $\zeta$ is also $\mathbf{A}_F$-linear using $(\ref{eq: diamonds})$.
\end{definition}

\subsubsection{Differential operators}
For each $\nu\in I_L$ there is an operator on $p$-adic cuspforms $d_\nu:\overline{\mathbf{S}}(\frak{N},L)\to \overline{\mathbf{S}}(\frak{N},L)$ defined by $\msf{a}_p(y,d_\nu\msf{g})=y_p^\nu\msf{a}_p(y,\msf{g})$, the definition can be extended to all  $r\in\bb{N}[I_L]$ by setting $d^r=\prod_{\nu\in I_L}d_\nu^{r_\nu}$ (\cite{pHida} Section $6\text{G}$). 
\begin{lemma}\label{lemma: auto comparison}
Let $r\in\bb{N}[I_L]$ and let $\msf{g}\in S_{\ell,x}(\frak{N},L;\overline{\bb{Q}})$ be a cuspform,  then 
\[
e_\text{n.o.}\zeta^*(d^r\msf{g})=e_\text{n.o.}\Pi^\text{hol}\zeta^*(\delta^r_\ell\msf{g}).
\]
\end{lemma}
\begin{proof}
For each nearly holomorphic cuspform $\msf{f}\in N_{k,w,m}(\frak{N}p^\alpha;A)$ one can define a $p$-adic cuspform by $
c(\msf{f})=\cal{N}(y)^{-1}\sum_{\xi\in F^\times_+}\msf{a}_p(\xi y\msf{d},\msf{f})(0)q^\xi\in\overline{\mathbf{S}}(\frak{N})$.
It's easy to see that $c(\zeta^*\msf{g})=\zeta^*c(\msf{g})$ using Lemma $\ref{lemma: p-adic diag}$. Moreover $c(\delta_k^r\msf{g})=d^rc(\msf{g})$ as 
\[\begin{split}
\msf{a}_p(y,c(\delta_k^r\msf{g}))=\msf{a}_p(y,\delta_k^r\msf{g})(0)&=y_p^{-v+r}\cal{N}(a_i)^{-1}\xi^{v-r}a(\xi,\delta_k^r\msf{g}_i)(0)\\
	&=y_p^{-v+r}\cal{N}(a_i)^{-1}\xi^{v}a(\xi,\msf{g}_i)(0)=\msf{a}_p(y,d^rc(\msf{g})).
\end{split}\]
Therefore to conclude we use (\cite{pHida} Proposition 7.3) to get the equality $e_\text{n.o.}\Pi^\text{hol}\zeta^*(\delta^r_\ell\msf{g})=e_\text{n.o.}c(\zeta^*(\delta^r_\ell\msf{g}))=e_\text{n.o.}\zeta^*(d^rc(\msf{g}))$. 
\end{proof}



\section{Twisted triple product $L$-functions}
\subsection{Complex L-functions}
\label{subsubsec: complex L-functions}
 Let $L/F$ be a quadratic extension of totally real number fields, $\frak{Q}\triangleleft\cal{O}_L$ and $\frak{N}\triangleleft\cal{O}_F$ ideals. The primitive eigenforms $\msf{g}\in S_{\ell,x}(\frak{Q}, L; \overline{\bb{Q}})$ and $\msf{f}\in S_{k,w}(\frak{N}, F; \overline{\bb{Q}})$ generate irreducible cuspidal automorphic representations $\pi, \sigma$ of $G_L(\bb{A}), G_F(\bb{A})$ respectively. Let $\pi^u=\pi\otimes\lvert\bfcdot\rvert_{\bb{A}_L}^{n/2}$, $\sigma^u=\sigma\otimes\lvert\bfcdot\rvert_{\bb{A}_F}^{m/2}$ their unitarizations, where $n,m$ are the integers satisfying $n\cdot t_L=\ell-2x$, $m\cdot t_F=k-2w$, and let $\msf{g}^\circ$, $\msf{f}^\circ$ be the primitive forms associated to them. One can define a unitary representation of $\text{GL}_2(\bb{A}_{L\times F})$ by $\Pi=\pi^u\otimes\sigma^u$. Let $\rho:\Gamma_F\to S_3$ be the homomorphism mapping the absolute Galois group of $F$ to the symmetric group over $3$ elements associated with the \'etale cubic algebra $(L\times F)/F$. The $L$-group $^L(G_{L\times F})$ is given by the semi-direct product $\widehat{G}\rtimes\Gamma_F$ where $\Gamma_F$ acts on $\widehat{G}=\text{GL}_2(\bb{C})^{\times3}$ through $\rho$. 

\begin{definition}
Assume the central character $\omega_\Pi$ of $\Pi$ satisfies ${\omega_\Pi}_{\lvert\bb{A}_F^\times}\equiv 1$, then the twisted triple product $L$-function associated with the unitary autormophic representation $\Pi$ is given by the Euler product
\[
L(s,\Pi,r)=\prod_v L_v(s,\Pi_v,r)^{-1}
\]
where $\Pi_v$ is the local representation at the place $v$ of $F$ appearing in the restricted tensor product decomposition $\Pi=\bigotimes_v'\Pi_v$ and representation $r$ gives the action of the $L$-group of $G_{L\times F}$ on $\bb{C}^2\otimes\bb{C}^2\otimes\bb{C}^2$ which restricts to the natural $8$-dimensional representation of $\widehat{G}$ and for which $\Gamma_F$ acts via $\rho$ permuting the vectors.
\end{definition}

The complex $L$-function $L(s,\Pi,r)$ was studied in (\cite{PS-R} Theorem 5.1 , 5.2 , 5.3), where it was proved it has meromorphic continuation to $\bb{C}$ with possible poles at $0,\frac{1}{4},\frac{3}{4},1$ and functional equation
\[
L(s,\Pi,r)=\epsilon(s,\Pi,r)L(1-s,\Pi,r).
\]
The functional equation has $\frac{1}{2}$ as central point at which the $L$-function is holomorphic. The Euler factors at the unramified primes can be easily described. Suppose $v$ is a prime of $F$ unramified in $L$ for which $\Pi_v$ is an unramified principal series, that is $v\nmid\frak{N}\cdot\text{N}_{L/F}(\frak{Q})\cdot d_{L/F}$. We distinguish the two following cases.

\paragraph{Split case.} Suppose that $v=\scr{V}\cdot\overline{\scr{V}}$ splits in $L$, write $\varpi_v$ for a uniformizer of $F_v$ and $q_v$ for the size of its residue field. The representation $\Pi_v$ of $\text{GL}_2(F_v)^{\times3}$ can be written as
$\Pi_v=\pi(\chi_{1,\scr{V}},\chi_{2,\scr{V}})\otimes \pi(\chi_{1,\overline{\scr{V}}},\chi_{2,\overline{\scr{V}}})\otimes \pi(\psi_{1,v},\psi_{2,v})$
using normalized parabolic induction. One can attach to $\Pi_v$ a conjugacy class $C_v$ in the $L$-group of $G_{L\times F}$ via the Satake isomorphism
\[
C_v=\left[\begin{pmatrix}\chi_{1,\scr{V}}(\varpi_v)& 0\\
0& \chi_{2,\scr{V}}(\varpi_v)\end{pmatrix},
\begin{pmatrix}\chi_{1,\overline{\scr{V}}}(\varpi_v)& 0\\
0& \chi_{2,\overline{\scr{V}}}(\varpi_v)\end{pmatrix},
\begin{pmatrix}\psi_{1,v}(\varpi_v)& 0\\
0& \psi_{2,v}(\varpi_v)\end{pmatrix}; 1\right].
\]
Then the local $L$-factor is given by the formula $L_v(s,\Pi_v,r)=\det(\bb{I}_8-\text{r}(C_v)q_v^{-s})$, which can be compute to be equal to
\begin{equation}\label{eq: L-factor split}
L_v(s,\Pi_v,r)=\prod_{i,j,k}(1-\chi_{i,\scr{V}}(\varpi_v)\chi_{j,\overline{\scr{V}}}(\varpi_v)\psi_{k,v}(\varpi_v)q_v^{-s}).
\end{equation}
\paragraph{Inert case.} Suppose that $v$ is inert in $L$. The representation $\Pi_v$ of $\text{GL}_2(L_v)\times\text{GL}_2(F_v)$ can be written as $
\Pi_v=\pi(\chi_{1,v},\chi_{2,v})\otimes \pi(\psi_{1,v},\psi_{2,v})$.
The Satake isomorphism attach to it a conjugacy class $C_v$ in the $L$-group of $G_{L\times F}$
\[
C_v=\left[\begin{pmatrix}\chi_{1,v}(\varpi_v)& 0\\
0& \chi_{2,v}(\varpi_v)\end{pmatrix},
\begin{pmatrix}1& 0\\
0& 1\end{pmatrix},
\begin{pmatrix}\psi_{1,v}(\varpi_v)& 0\\
0& \psi_{2,v}(\varpi_v)\end{pmatrix}; \text{Fr}_v\right].
\]
 The value of the representation $r$ on this conjugacy class is
\[
r(C_v)=\begin{pmatrix}
\chi_{1,v}(\varpi_v) &&&\\
&&\chi_{1,v}(\varpi_v)&\\
&\chi_{2,v}(\varpi_v) &&\\
&&& \chi_{2,v}(\varpi_v)
\end{pmatrix}\otimes\begin{pmatrix}
\psi_{1,v}(\varpi_v)&\\
& \psi_{2,v}(\varpi_v)
\end{pmatrix}
\]
and the local $L$-factor can be computed to be equal to
\begin{equation}\label{eq: L-factor inert}
L_v(s,\Pi_v,r)=\ \det(\bb{I}_8-\text{r}(C_v)q_v^{-s})= P(q_v^s)\times\prod_{i,j}(1-\chi_{i,v}(\varpi_v)\psi_{j,v}(\varpi_v)q_v^{-s})
\end{equation}
where $P(q_v^s)=\big[1-\chi_{1,v}\chi_{2,v}(\varpi_v)(\psi_{1,v}(\varpi_v)^2+\psi_{2,v}(\varpi_v)^2)q_v^{-2s}+(\chi_{1,v}\chi_{2,v}(\varpi_v)\psi_{1,v}\psi_{2,v}(\varpi_v))^2q_v^{-4s}\big]$.

\begin{remark}\label{relation}
The relation between Satake parameters of $\pi^u$, $\sigma^u$ and Hecke eigenvalues of the primitive eigenforms $\msf{g}^\circ,\msf{f}^\circ$ can be given explicitly as follows. 
Suppose $v\nmid\frak{Q}$ and $v=\scr{V}\overline{\scr{V}}$ splits in $L$, then
\begin{equation}\label{eq:Satake2}
{\msf{g}^\circ}_{\lvert T(\scr{V})}=q_v^{1/2}\left(\chi_{1,\scr{V}}(\varpi_v)+\chi_{2,\scr{V}}(\varpi_v)\right)\msf{g}^\circ,\qquad {\msf{g}^\circ}_{\lvert T(\overline{\scr{V}})}=q_v^{1/2}\left(\chi_{1,\overline{\scr{V}}}(\varpi_v)+\chi_{2,\overline{\scr{V}}}(\varpi_v)\right)\msf{g}^\circ.
\end{equation}
Moreover, if $v\nmid\frak{Q}$ and $v$ is inert in $L$ then
\begin{equation}\label{eq:Satake3}
{\msf{g}^\circ}_{\lvert T(v)}=q_v\left(\chi_{1,v}(\varpi_v)+\chi_{2,v}(\varpi_v)\right)\msf{g}^\circ.
\end{equation}
Finally, if $v\nmid\frak{N}$ a finite place of $F$ we have
\begin{equation}\label{eq:Satake1}
{\msf{f}^\circ}_{\lvert T(v)}=q_v^{1/2}\left(\psi_{1,v}(\varpi_v)+\psi_{2,v}(\varpi_v)\right)\msf{f}^\circ.
\end{equation}
\end{remark}


\subsection{Central $L$-values and period integrals}
\label{subsec: Central values and period integrals}
Let $D_{/F}$ be a quaternion algebra. We denote be $\Pi^D$ the irreducible unitary cuspidal automorphic representation of $D^\times(\bb{A}_E)$ associated with $\Pi$ by the Jacquet-Langlands correpondence. 
For a vector $\phi\in\Pi^D$ one defines its period integral as
\[
\text{I}^D(\phi)=\int_{[D^\times(\bb{A}_F)]}\phi(x)\text{d}^\times x
\]
where $[D^\times(\bb{A}_F)]=\bb{A}_F^\times D^\times(F)\backslash D^\times(\bb{A}_F)$. To simplify the notation we write $\text{I}(\phi)$ to denote the period integral for the quaternion algebra $\text{M}_2(F)$.

\begin{theorem}\label{thm: Itch}
Let $\eta:\bb{A}_F^\times\to\bb{C}^\times$ be the quadratic character attached to $L/F$ by class field theory. Then the following are equivalent:
\begin{itemize}

\item[(1)] The central $L$-value $L(\frac{1}{2},\Pi,r)$ does not vanish, and for every place $v$ of $F$ the local $\epsilon$-factor satisfies $\epsilon_v(\frac{1}{2},\Pi_v,r)\cdot\eta_v(-1)=1$.
\item[(2)] There exists a vector $\phi\in\Pi$, called test vector, whose period integral $\text{I}(\phi)$ does not vanish. Moreover, we can assume the vector $\phi$ is defined over a Galois number field $E$ containing $L,F$ as well as the Hecke eigenvalues of $\msf{g}^\circ$ and $\msf{f}^\circ$ .

\end{itemize}
\end{theorem}
\begin{proof}
$(1)\implies(2)$ By Jacquet conjecture, as proved in (\cite{P-SP} Theorem 1.1), the non-vanishing of the central value implies that there exists a quaternion algebra $D{/F}$ and a vector $\phi\in\Pi^D$ such that its period integral is non-zero, i.e., $\text{I}^D(\phi)\not=0$. We want to show that the assumption on local $\epsilon$-factors forces the quaternion algebra to be split everywhere. Ichino's formula (\cite{I} Theorem 1.1) gives an equality, up to non-zero constants,
\[
\text{I}^D\cdot \widetilde{\text{I}^D}\overset{.}{=} L\big(\frac{1}{2},\Pi,r\big)\cdot\prod_v \text{I}^D_v
\]
of linear forms in $\text{Hom}_{D^\times(\bb{A}_F)\times D^\times(\bb{A}_F)}\left(\Pi^D\otimes\widetilde{\Pi^D},\bb{C}\right)$ where the $\text{I}^D_v$'s are local linear forms in $\text{Hom}_{D^\times(\bb{A}_{F_v})\times D^\times(\bb{A}_{F_v})}\left(\Pi_v^D\otimes(\widetilde{\Pi^D})_v,\bb{C}\right)$. 
Suppose $v$ is a place of $F$ at which the quaternion algebra $D$ ramifies, i.e. $v\mid \text{disc}D$. Requiring the value of the expression $\epsilon_v(\frac{1}{2},\Pi_v,r)\cdot\eta_\ell(-1)$ to be equal to $1$ forces the local Hom-space $\text{Hom}_{D^\times(\bb{A}_{F_v})\times D^\times(\bb{A}_{F_v})}\left(\Pi_v^D\otimes(\widetilde{\Pi^D})_v,\bb{C}\right)$ to be trivial  (\cite{Gan} Theorem 1.2); in particular it forces the local linear form $\text{I}^D_v$ to be trivial. This produces a contradiction because the LHS of Ichino's formula is non-trivial. Indeed, choosing the complex conjugate $\bar{\phi}\in\overline{\Pi^D}\cong\widetilde{\Pi^D}$ of the test vector $\phi$ we compute that 
\[
\text{I}^D\cdot \widetilde{\text{I}^D}(\phi\otimes\bar{\phi})=\left\lvert \text{I}^D(\phi)\right\rvert^2\not=0.
\]
 Hence, the discriminant of $D$ has to be trivial, i.e., $D=M_2(F)$.
\newline
$(2)\implies(1)$ The existence of a test vector $\phi\in\Pi$ implies the non-vanishing of the central value $L(\frac{1}{2},\Pi,r)$ by Jacquet conjecture. Moreover, Ichino's formula provides us with non-trivial local linear forms, the $\text{I}_v$'s, in the local Hom-spaces $\text{Hom}_{\text{GL}_2(\bb{A}_{F_v})\times \text{GL}_2(\bb{A}_{F_v})}\left(\Pi_v\otimes(\widetilde{\Pi})_v,\bb{C}\right)$ which force the equality $\epsilon_v(\frac{1}{2},\Pi_v,r)\cdot\eta_v(-1)=1$ for every place $v$ of $F$ (\cite{Gan} Theorem 1.1).
\end{proof}

\begin{remark}
We can give sufficient conditions on the eigenforms $\msf{g}\in S_{\ell,x}(\frak{Q}, L; \overline{\bb{Q}})$ and $\msf{f}\in S_{k,w}(\frak{N}, F; \overline{\bb{Q}})$ such that the local $\epsilon$-factors of the automorphic representation $\Pi$ satisfy the hypothesis of Theorem $\ref{thm: Itch}$. The local $\epsilon$-factor at the archimedean places of $F$ satisfy the hypothesis of the theorem if the weights of $\msf{g}$ and $\msf{f}$ are unbalanced. Moreover, the same is true for the $\epsilon$-factors at the finite places if we assume that $\text{N}_{L/F}(\frak{Q})\cdot d_{L/F}$ and $\frak{N}$ are coprime and that every finite prime $v$ dividing the level $\frak{N}$ of $\msf{f}$ splits in $L$ (\cite{epsilonprasad} Theorems $\text{B},\text{D}$ and Remark 4.1.1).
\end{remark}


\subsection{$p$-adic $L$-functions}
\label{subsec:p-adic L-functions}
Let $L/F$ be a quadratic extension of totally real number fields and let $\msf{g}\in S_{\ell,x}(\frak{Q}, L; \overline{\bb{Q}})$, $\msf{f}\in S_{k,w}(\frak{N}, F; \overline{\bb{Q}})$ be primitive eigenforms of unbalanced weights. We assume the central character $\omega_\Pi$ of $\Pi$ to be trivial when restricted to $\bb{A}_F^\times$, that the central $L$-value $L(\frac{1}{2},\Pi,r)$ does not vanish, and that for every place $v$ of $F$ we have the condition $\epsilon_v(\frac{1}{2},\Pi_v,r)\eta_v(-1)=1$ on local $\epsilon$-factors satisfied. Then there exists a vector $\phi\in\Pi$ such that the period integral $\text{I}(\phi)$ is non-zero ( Theorem $\ref{thm: Itch}$). Let $J$ be the element
\[
J=\begin{pmatrix}
-1 & 0 \\ 
0 & 1
\end{pmatrix}^{I_F}\in\text{GL}_2(\bb{R})^{I_F};
\]
for any $\msf{h}\in\sigma^u$ we define $\msf{h}^J\in\sigma^u$ to be the vector obtained by right translation $\msf{h}^J(g)=\msf{h}(gJ)$. If $\msf{h}$ has weight $k\in\bb{Z}[I_F]$ then $\msf{h}^J(h)$ has weight $-k$.
\begin{lemma}\label{lemma:choice test vector}
Let $r\in\bb{N}[I_L]$ be such that $k=(\ell+2r)_{\lvert F}$ and $w=(x+r)_{\lvert F}$. Then the test vector $\phi$ can be chosen to be of the form $\phi=(\delta^r\breve{\msf{g}})^u\otimes (\breve{\msf{f}}^J)^u$ for $\breve{\msf{g}}\in S_{\ell,x}(M,L;E)$ and $\breve{\msf{f}}\in S_{k,w}(M,F;E)$. The $\cal{O}_F$-ideal $M$ has only prime factors among those of $\frak{N}\cdot\text{N}_{L/F}(\frak{Q})\cdot d_{L/F}$, and $\breve{\msf{g}}, \breve{\msf{f}}$ are eigenforms for all Hecke operators outside $\frak{N}\cdot\text{N}_{L/F}(\frak{Q})\cdot d_{L/F}$ with the same Hecke eigenvalues of $\msf{g}$ and $\msf{f}$ respectively. 
\end{lemma}
\begin{proof}
By linearity of the period integral we can assume $\phi$ to be a simple tensor of the form $\phi=\delta^r\vartheta\otimes\nu^J\in\Pi$. Indeed, we can choose the holomorphic lowest weight vectors for $\vartheta$ and $\nu$ at the archimedean places because the archimedean linear functional appearing in Ichino's formula is non-zero if and only if the sum of the weights of the local vectors is zero. We claim that for every finite place $v$ of $F$ not dividing $\frak{N}\cdot\text{N}_{L/F}(\frak{Q})\cdot d_{L/F}$, the newvector in $\Pi_v$ is a choice for which Ichino's local linear functional doesn't vanish. If $v$ is a place that splits in $L$, the claim follows from (\cite{prasadtrilinear} Theorem 5.10), the proof given by Prasad can be adapted to the inert case as follows. As in (\cite{epsilonprasad} Section 4) we can assume that $(\pi^u)_v$ is the principal series $V_\chi$ for  the character of the Borel $\chi(ab0d)=\alpha(a)/\beta(d)$ induced by unramified characters $\alpha,\beta:L_v^\times\to\bb{C}^\times$, then the representation $V_\chi$ can be realized in the functions of $\bb{P}^1_{L_v}$. The projective line $\bb{P}^1_{L_v}$ can be decomposed into an open and a closed orbit for the action of $\text{GL}_2(F_v)$, 
\[\bb{P}^1_{L_v}=\left(\bb{P}^1_{L_v}\setminus\bb{P}^1_{F_v}\right)\coprod \bb{P}^1_{F_v},
\]
which produces an exact sequence of $\text{GL}_2(F_v)$-modules
\begin{equation}\label{ses of representations}\xymatrix{
0\ar[r]& \text{ind}_{L_v^\times}^{\text{GL}_2(F_v)}(\chi')\ar[r]&V_\chi\ar[r]&\text{Ind}_{\text{B}(F_v)}^{\text{GL}_2(F_v)}(\chi\delta_{F_v}^{1/2})\ar[r]&0
}\end{equation}
for $\chi':L_v^\times\to\bb{C}^\times$ the character defined by $\chi'(x)=\alpha(x)\beta(\overline{x})$. If $\text{Ind}_{\text{B}(F_v)}^{\text{GL}_2(F_v)}(\chi\delta_{F_v}^{1/2})$ is isomorphic to the contragradient representation $\widetilde{(\sigma^u)_v}$  then we are done. Otherwise, we assume that the image of the newvector $1_{\bb{P}^1_{L_v}}$ is zero under the local linear functional $\Upsilon:V_\chi\to\widetilde{(\sigma^u)_v}$ appearing in Ichino's formula and we find a contradiction. Let $T_v$ be the Hecke operator given by the double coset 
\[
T_v=\left[\text{GL}_2(\cal{O}_{F_v})\begin{pmatrix}
\varpi_v&0\\
0&1
\end{pmatrix}\text{GL}_2(\cal{O}_{F_v})\right],
\]
then the function
\begin{equation}\label{function}
\frac{1}{(q_v+1)\chi\delta^{1/2}(\varpi_v)}\left(T_v(1_{\bb{P}^1_{L_v}})-q_v\chi\delta^{1/2}(\varpi_v)-\chi\delta^{1/2}(1/\varpi_v)\right)
\end{equation}
is the constant function $1$ on the $\text{GL}_2(\cal{O}_{F_v})$-orbit of $\bb{P}^1_{L_v}$ consisting of those points that reduce to a point in $\bb{P}^1(\cal{O}_{L_v}/\varpi_v)\setminus \bb{P}^1(\cal{O}_{F_v}/\varpi_v)$, and the constant function zero everywhere else. Therefore, the function $(\ref{function})$ is an element of $\text{ind}_{L_v^\times}^{\text{GL}_2(F_v)}(\chi')$ because of the short exact sequence $(\ref{ses of representations})$. The contraddiction is that the function $(\ref{function})$ is sent to zero by $\Upsilon$ by $\text{GL}_2(F_v)$-equivariance, but at the same time that is not possible because we can explicitly describe the elements of $\text{Hom}_{\text{GL}_2(F_v)}\left(\text{ind}_{L_v^\times}^{\text{GL}_2(F_v)}(\chi'),\widetilde{(\sigma^u)_v} \right)$ in terms of integration over $\text{GL}_2(\cal{O}_{F_v})$-orbits of $\bb{P}^1_{L_v}$.

To conclude we note that spherical vectors are mapped to spherical vectors by the isomorphism  $\pi\otimes\sigma\overset{\sim}{\to}\pi^u\otimes\sigma^u$ as in $(\ref{eq:unitarization})$. Therefore we can write $\phi=\delta^r\vartheta\otimes\nu^J$ as $(\delta^r\breve{\msf{g}})^u\otimes(\breve{\msf{f}}^J)^u$, for $\msf{g}\in\pi$ and $\msf{f}\in\sigma$ as claimed. 
\end{proof}
Therefore we can rewrite the period integral $\text{I}(\phi)$ as a Petersson inner product
\begin{equation}\label{eq: period-Petersson}
\text{I}(\phi)=\int_{[\text{GL}_2(\bb{A}_F)]}(\delta^r\breve{\msf{g}})^u\otimes(\breve{\msf{f}}^\frak{J})^u\text{d}^\times x=\left\langle\zeta^*\left(\delta^r\breve{\msf{g}}\right),\breve{\msf{f}}^*\right\rangle
\end{equation}
where $\breve{\msf{f}}^*=\overline{\big(\breve{\msf{f}}^\frak{J}\big)}$ is the cuspform in $S_{k,w}(M,F;E)$ whose Fourier coefficients are complex conjugates of those of $\breve{\msf{f}}$. We conclude the section with a proposition showing that a good trascendental period for the central $L$-value of the twisted triple product $L$-function is the Petersson norm of the eigenform $\msf{f}^*$.

\begin{proposition}\label{prop: rationality}
Let $E$ be a number field and let $\breve{\msf{f}}\in S_{k,w}(M;E)$ be a vector in an irreducible cuspidal automorphic representation $\sigma$ spanned by a primitive cuspform $\msf{f}\in S_{k,w}(\frak{N};E)$. Then for any $\phi\in S_{k,w}(M;E)$ the Petersson inner product $\langle\phi,\breve{\msf{f}}\rangle$ is a $E$-rational multiple of $\langle\msf{f},\msf{f}\rangle$.
\end{proposition}
\begin{proof}
The Petersson inner product $\langle\breve{\msf{f}},\phi\rangle$ depends only on the projection $e_\msf{f}\phi$ of $\phi$ to $\sigma$. The $E$-vector space $e_\msf{f}S_{k,w}(M;E)$ is spanned by the cuspforms  
\[\left\{\msf{f}_{\frak{a}}\lvert\ \msf{f}_{\frak{a}}(x)=\msf{f}(xs_\frak{a})\right\}_{\frak{a}\mid\frac{M}{\frak{N}}},
\qquad s_\frak{a}=\begin{pmatrix}
1&0\\
0& a
\end{pmatrix},\ \ a\cal{O}_F=\frak{a}
\] for all ideals $\frak{a}$ dividing $M/\frak{N}$ (\cite{Shimura} Proposition 2.3). Thus, it suffices to prove the statement for $\msf{f}_{\frak{a}_1}$ and $\msf{f}_{\frak{a}_2}$ when $\frak{a}_1,\frak{a}_2\mid M/\frak{N}$. We prove the claim by induction on the prime divisors of $\frak{a}_1,\frak{a}_2$. If  $\frak{a}_1\frak{a}_2=\cal{O}_F$ then the claim is clear. Suppose there is a prime ideal $\frak{p}$ that divides both $\frak{a}_1$ and $\frak{a}_2$, then $\langle\msf{f}_{\frak{a}_1},\msf{f}_{\frak{a}_2}\rangle=\langle\msf{f}_{\frak{a}_1/\frak{p}},\msf{f}_{\frak{a}_2/\frak{p}}\rangle$ because the Haar measure is invariant under translation. Thus, without loss of generality, we can assume $\frak{p}$ divides $\frak{a}_2$ but not $\frak{a}_1$. We compute the equalities \[
\msf{a}(\varpi_\frak{p},\msf{f})\langle\msf{f}_{\frak{a}_1},\msf{f}_{\frak{a}_2}\rangle=\begin{cases}
\langle T_0(\varpi_\frak{p})\msf{f}_{\frak{a}_1},\msf{f}_{\frak{a}_2}\rangle=(q_\frak{p}+1)\langle\msf{f}_{\frak{a}_1},\msf{f}_{\frak{a}_2/\frak{p}}\rangle \qquad &\text{if}\ \frak{p}\nmid\frak{N}\\
\langle U_0(\varpi_\frak{p})\msf{f}_{\frak{a}_1},\msf{f}_{\frak{a}_2}\rangle=q_\frak{p}\langle\msf{f}_{\frak{a}_1},\msf{f}_{\frak{a}_2/\frak{p}}\rangle\qquad &\text{if}\ \frak{p}\mid\frak{N},
\end{cases}\]
that show \[
\langle\msf{f}_{\frak{a}_1},\msf{f}_{\frak{a}_2}\rangle=\begin{cases}
\frac{\msf{a}(\varpi_\frak{p},\msf{f})}{q_\frak{p}+1}\langle\msf{f}_{\frak{a}_1},\msf{f}_{\frak{a}_2/\frak{p}}\rangle\qquad \text{if}\ \frak{p}\nmid\frak{N}\\
\\
\frac{\msf{a}(\varpi_\frak{p},\msf{f})}{q_\frak{p}}\langle\msf{f}_{\frak{a}_1},\msf{f}_{\frak{a}_2/\frak{p}}\rangle\qquad \text{if}\ \frak{p}\mid\frak{N},
\end{cases}\]
concluding the inductive step.
\end{proof}
\subsubsection{Construction.}
 We choose a rational prime $p$ unramified in $L$, coprime to the levels $\frak{Q},\frak{N}$, and $\Theta_{L/F}$ an element in $\prod_{\tau\in I_F}\{\mu\in I_L\lvert\ \mu_{\lvert F}=\tau\}$. The fixed $p$-adic embedding $\iota_p:\overline{\bb{Q}}\hookrightarrow\overline{\bb{Q}}_p$ allows us to associate to $\Theta_{L/F}$ a set $\cal{P}$ of prime $\cal{O}_L$-ideals above $p$. Moreover, we suppose $\msf{g}$ and $\msf{f}$ to be $p$-nearly ordinary and we denote by $\cal{G}$ (resp. $\cal{F}$) the Hida family with coefficients in $\mathbf{I}_\cal{G}$ (resp. $\mathbf{I}_\cal{F}$) passing through the nearly ordinary $p$-stabilization of $\msf{g}$ (resp. $\msf{f}$).

We set $\mathbf{K}_{\cal{G}}=\mathbf{I}_{\cal{G}}\otimes \bb{Q}$ (resp. $\mathbf{K}_{\cal{F}}=\mathbf{I}_{\cal{F}}\otimes\bb{Q}$) and we define a $\mathbf{K}_{\cal{G}}$-adic cuspform $\breve{\cal{G}}$ (resp. $\mathbf{K}_{\cal{F}}$-adic cuspform $\breve{\cal{F}}$) passing through the nearly ordinary $p$-stabilization of the test vectors $\breve{\msf{g}},\breve{\msf{f}}$ as in \cite{DR} Section 2.6.  Writing $\mathbf{R}_{\cal{G}}=O[[\cal{O}_{L,\cal{P}}^\times]]\otimes_O\mathbf{K}_{\cal{G}}$, we define a homomorphism of $\mathbf{A}_L$-modules $d^{\bfcdot}\breve{\cal{G}}^{[\cal{P}]}:\mathbf{h}_L(M;O)\to\mathbf{R}_{\cal{G}}$  by
\[
d^{\bfcdot}\breve{\cal{G}}^{[\cal{P}]}\left(\langle z\rangle\mathbf{T}(y)\right)=\begin{cases}[y_\cal{P}]\otimes\breve{\cal{G}}\left(\langle z\rangle\mathbf{T}(y)\right)\qquad &\text{if}\ y_\cal{P}\in\cal{O}_{L,\cal{P}}^\times\\
0\qquad&\text{otherwise}
\end{cases}
\]
where $\mathbf{R}_{\cal{G}}$ is given the $\mathbf{A}_L$-algebra structure defined by $[(z,a)]\mapsto[a_\cal{P}^{-1}]\otimes\breve{\cal{G}}(\langle z\rangle T(a^{-1},1))$. The nearly ordinary $\mathbf{R}_{\cal{G}}$-adic cuspform $e_{\text{n.o.}}\zeta^*\big(d^{\bfcdot}\breve{\cal{G}}^{[\cal{P}]}\big):\mathbf{h}^\text{n.o.}_F(M;O)\to\mathbf{R}_{\cal{G}}$ over $F$  is given by the composition
\[\xymatrix{
\mathbf{h}^\text{n.o.}_F(M;O)\ar[r]&\mathbf{h}_F(M;O)\ar[r]^\zeta&\mathbf{h}_L(M;O)\ar[rr]^{d^{\bfcdot}\breve{\cal{G}}^{[\cal{P}]}}&&\mathbf{R}_{\cal{G}}.
}\]


\begin{definition}
For a Hida family $\cal{F}\in\overline{\mathbf{S}}^\text{n.o.}(\frak{N};\mathbf{I}_\cal{F})$, we define the set of \emph{crystalline points} $\Sigma_\text{cry}(\mathbf{I}_\cal{F})$ to be the subset of arithmetic points $\text{Q}\in\cal{A}(\mathbf{I}_\cal{F})$ of trivial characters and such that the specialization of the Hida family is old at $p$, i.e., $\cal{F}_\text{Q}=\msf{f}_\text{Q}^{(p)}$ for some eigenform $\msf{f}_\text{Q}$ of prime-to-$p$ level. 
\end{definition}
\begin{lemma}\label{existence}
Let $\breve{\cal{F}}\in e_\cal{F}\overline{\mathbf{S}}^\text{n.o.}(M;\mathbf{K}_\cal{F})$ be a $\mathbf{K}_\cal{F}$-adic cuspform in the $\cal{F}$-isotypic part of $\overline{\mathbf{S}}^\text{n.o.}(M;\mathbf{K}_\cal{F})$. For any $\mathbf{A}_F$-algebra $\mathbf{R}$ and  $\mathbf{R}$-adic cuspform $\cal{H}\in\overline{\mathbf{S}}^\text{n.o.}(M;\mathbf{R})$, there exists an element $J(\cal{H},\breve{\cal{F}})\in\mathbf{R}\otimes_{\mathbf{A}_F}\text{Frac}(\mathbf{K}_{\cal{F}})$ such that for all $(\text{P},\text{Q})\in\cal{A}(\mathbf{R})\times_{\cal{A}(\mathbf{A}_F)}\Sigma_\text{cry}(\mathbf{I}_{\cal{F}})$ we have
\begin{equation}\label{eq: expression}
J(\cal{H},\breve{\cal{F}})(\text{P},\text{Q})=\frac{\big\langle \msf{h}_{\text{P},\text{Q}}^{(p)},\breve{\msf{f}}^{(p)}_\text{Q}\big\rangle}{\big\langle\msf{f}_\text{Q}^{(p)},\msf{f}_\text{Q}^{(p)}\big\rangle},
\end{equation}
where $\msf{h}_{\text{P},\text{Q}}^{(p)}$ denotes the $\msf{f}^{(p)}_\text{Q}$-isotypic projection of the cuspform  $\cal{H}_\text{P}$. The projection is the $p$-stabilization of a cuspform $\msf{h}_{\text{P},\text{Q}}$.
\end{lemma}
\begin{proof}
We follow the same argument of \cite{DR} Lemma 2.19. Both $\breve{\cal{F}}$ and the $\cal{F}$-isotypic projection $e_\cal{F}\cal{H}$ are $\mathbf{R}\otimes_{\mathbf{A}_F}\mathbf{K}_\cal{F}$-linear combinations of the $\mathbf{K}_\cal{F}$-adic cuspforms $\cal{F}_\frak{a}$ for $\frak{a}\mid M/\frak{N}$; hence, the element $J(\cal{H},\breve{\cal{F}})$ exists because we can interpolate  expressions of the form 
\[
\big\langle\msf{f}_{\frak{a}_1, \text{Q}}^{(p)},\msf{f}_{\frak{a}_2, \text{Q}}^{(p)}\big\rangle/\big\langle\msf{f}_{ \text{Q}}^{(p)},\msf{f}_{\text{Q}}^{(p)}\big\rangle
\]
for $\text{Q}\in\Sigma_\text{cry}(\mathbf{I}_\cal{F})$ because of the explicit computations in the proof of Proposition $\ref{prop: rationality}$ and the fact that $M$ is prime to $p$. 
\end{proof}
\begin{definition}
The twisted triple product $p$-adic $L$-function attached to the couple $(\breve{\cal{G}},\breve{\cal{F}})$ adic cuspforms and the element $\Theta_{L/F}\in\prod_{\tau\in I_F}\{\mu\in I_L\lvert\ \mu_{\lvert F}=\tau\}$ is
\[
\mathscr{L}_p\big(\breve{\cal{G}},\breve{\cal{F}}\big)=J\big(e_{\text{n.o.}}\zeta^*\big(d^{\bfcdot}\breve{\cal{G}}^{[\cal{P}]}\big),\breve{\cal{F}}^*\big)\in \mathbf{R}_{\cal{G}}\otimes_{\mathbf{A}_F}\text{Frac}(\mathbf{K}_{\cal{F}}),
\] 
where $\breve{\cal{F}}^*=\breve{\cal{F}}\otimes\psi_\msf{\msf{f}}^{-1}$ is the $\mathbf{K}_\cal{F}$-adic cuspform that satisfies $(\breve{\cal{F}}^*)_\text{Q}=(\breve{\cal{F}}_\text{Q})^*$ for all $\text{Q}\in\Sigma_\text{cry}(\mathbf{I}_\cal{F})$.
\end{definition}
The twisted triple product $p$-adic $L$-function takes values
\begin{equation}\label{eq: value p-adic}
\mathscr{L}_p\big(\breve{\cal{G}},\breve{\cal{F}}\big)(\text{P},\text{Q})=\frac{\big\langle e_\text{n.o.}\zeta^*\big(d^{r}\breve{\msf{g}}_\text{P}^{(p),[\cal{P}]}\big), \breve{\msf{f}}_\text{Q}^{*(p)}\big\rangle}{\big\langle \msf{f}_\text{Q}^{*(p)}, \msf{f}_\text{Q}^{*(p)}\big\rangle}=\frac{\big\langle e_\text{n.o.}\zeta^*\big(d^{r}\breve{\msf{g}}_\text{P}^{[\cal{P}]}\big), \breve{\msf{f}}_\text{Q}^{*(p)}\big\rangle}{\big\langle \msf{f}_\text{Q}^{*(p)}, \msf{f}_\text{Q}^{*(p)}\big\rangle}
\end{equation}
at all $(\text{P},\text{Q})\in\Sigma^\msf{f}_\text{cry}$, where $r\in\bb{Z}[I_L]$ is the unique element, with $r_\mu=0$ if $\mu\notin\Theta_{L/F}$, that makes the weights of $e_\text{n.o.}\zeta^*(d^r\breve{\msf{g}}^{[\cal{P}]}_\text{P})$ and $\breve{\msf{f}}^*_\text{Q}$ match. The second equality of $(\ref{eq: value p-adic})$ follows from Lemma $\ref{thor}$ which shows that $e_\text{n.o.}\zeta^*\big(d^{r}\breve{\msf{g}}_\text{P}^{(p),[\cal{P}]}\big)=e_\text{n.o.}\zeta^*\big(d^{r}\breve{\msf{g}}_\text{P}^{[\cal{P}]}\big)$.

Let $\msf{h}_{\text{P},\text{Q}}=e_{\msf{f}_\text{Q}^*,\text{n.o.}}\zeta^*(d^r \breve{\msf{g}}_\text{P}^{[\cal{P}]})$ with nearly ordinary $p$-stabilization $\msf{h}_{\text{P},\text{Q}}^{(p)}=(1-\beta_{\msf{f}_\text{Q}^*}V(p))\msf{h}_\text{P,Q}$. By definition $e_\text{n.o.}\msf{h}_\text{P,Q}^{(p)}=\msf{h}_\text{P,Q}^{(p)}$, that results in the equality
$\msf{h}_\text{P,Q}^{(p)}=e_\text{n.o.}\msf{h}_\text{P,Q}^{(p)}=(1-\beta_{\msf{f}_\text{Q}^*}\alpha_{\msf{f}_\text{Q}^*}^{-1})e_\text{n.o.}\msf{h}_\text{P,Q}.$
More explicitly, if we set $\msf{E}(\msf{f}_\text{Q}^*)=(1-\beta_{\msf{f}_\text{Q}^*}\alpha_{\msf{f}_\text{Q}^*}^{-1}),$
		\[\left(e_{\msf{f}_\text{Q}^*,\text{n.o.}}\zeta^*(d^r \breve{\msf{g}}_\text{P}^{[\cal{P}]})\right)^{(p)}=\msf{E}(\msf{f}_\text{Q}^*)\cdot e_{\msf{f}_\text{Q}^*,\text{n.o.}}\zeta^*(d^r \breve{\msf{g}}_\text{P}^{[\cal{P}]})
	\]
that allows to rewrite the values of the $p$-adic $L$-function at every $(\text{P},\text{Q})\in\Sigma^\msf{f}_\text{cry}$ as
	\begin{equation}\label{eq: second expression p-adic L-function}
		\mathscr{L}_p\big(\breve{\cal{G}},\breve{\cal{F}}\big)(\text{P},\text{Q})=\frac{1}{\msf{E}(\msf{f}_\text{Q}^*)}\frac{\big\langle e_\text{n.o.}\zeta^*\big(d^{r}\breve{\msf{g}}_\text{P}^{[\cal{P}]}\big), \breve{\msf{f}}_\text{Q}^{*}\big\rangle}{\big\langle \msf{f}_\text{Q}^{*}, \msf{f}_\text{Q}^{*}\big\rangle}.
	\end{equation}

\subsection{Interpolation formulas}
The interpolation formulas satisfied by the twisted triple product $p$-adic $L$-function include Euler factors that depend on whether the primes in $\cal{P}$ are above a prime of $F$ that is split or inert in the extension $L/F$. We denote by $\cal{Q}$ the set of primes of $F$  above $p$ and we partition it, $\cal{Q}=\cal{Q}_\text{inert}\coprod\cal{Q}_\text{split}$, accordingly to the splitting behavior in $L/F$. For every prime $\cal{O}_F$-ideal $\wp\in\cal{Q}$  we denote by $q_\wp$ the cardinality of its residue field.

\paragraph{Inert case.} For a prime ideal $\frak{p}\in\cal{P}$  with $\wp=\frak{p}\cap\cal{O}_F\in\cal{Q}_\text{inert}$, we write $\frak{p}=\wp\cal{O}_L$.
\begin{lemma}\label{lemma: inert}
Let $\frak{N}$ be an $\cal{O}_F$-ideal  coprime to $p$, and $\wp\mid p$ a prime $\cal{O}_F$-ideal inert in $L/F$. Let $(\ell,x)\in\bb{Z}[I_L]^2$, $(k,w)\in\bb{Z}[I_F]^2$ and $r\in\bb{N}[I_L]$ be such that $(k,w)=((\ell+2r)_{\lvert F},(x+r)_{\lvert F})$. Let $\msf{g}\in S_{\ell,x}(\frak{N},L;E)$ be an eigenvector for the Hecke operator $T(\frak{p})$ and let $\msf{f}\in S_{k,w}(\frak{N},F;E)$ be a $p$-nearly ordinary  eigenform. If we denote by $e_{\msf{f},\text{n.o.}}=e_\msf{f}e_\text{n.o.}$ the composition of the $\msf{f}$-isotypic projection with the nearly ordinary projector, we have 
\[
e_{\msf{f},\text{n.o.}}\zeta^*\big(d^r\msf{g}^{[\frak{p}]}\big)=\cal{E}_\wp(\msf{g},\msf{f})e_{\msf{f},\text{n.o.}}\zeta^*\left(d^r\msf{g}\right)\qquad \text{for}\qquad \cal{E}_\wp(\msf{g},\msf{f})=\big(1-\alpha_\msf{g}\alpha_{\msf{f}}^{-1}q_\wp^{-1}\big)\big(1-\beta_\msf{g}\alpha_{\msf{f}}^{-1}q_\wp^{-1}\big),
\]
where $\varpi_\frak{p}^{-v}\alpha_\msf{g},\varpi_\frak{p}^{-v}\beta_\msf{g}$ are the inverses of the roots of the Hecke polynomial $1-\mathsf{a}(\frak{p},\msf{g})X+\text{N}_{L/\bb{Q}}(\frak{p})\psi_{\msf{g},0}(\frak{p})X^2$ of $\msf{g}$  with respect to $T_0(\frak{p})$ and $\alpha_\msf{f}$ is determined by $(e_\text{n.o.}\msf{f})_{\lvert U(\varpi_\wp)}=\alpha_\msf{f}\cdot e_\text{n.o.}\msf{f}$.
\end{lemma}
\begin{proof}
Let $\msf{g}_{\alpha}^{(\frak{p})}=(1-{\beta_{\msf{g}}}V(\varpi_\frak{p}))\msf{g}$, $\msf{g}_{\beta}^{(\frak{p})}=(1-{\alpha_\msf{g}}V(\varpi_\frak{p}))\msf{g}$ be the two $\frak{p}$-stabilizations of $\msf{g}$, they satisfy $U(\varpi_\frak{p})\msf{g}_{\bfcdot}^{(\frak{p})}=(\bfcdot) \msf{g}_{\bfcdot}^{(\frak{p})}$ and $\msf{g}=1/({\alpha_\msf{g}}-{\beta_\msf{g}})\big({\alpha_\msf{g}} \msf{g}_{\alpha}^{(\frak{p})}-{\beta_\msf{g}} \msf{g}_{\beta}^{(\frak{p})}\big)$.
Using Proposition $\ref{prop: V operator}$, we compute 
\[\begin{split}
e_{\msf{f},\text{n.o.}}\zeta^*\left[d^r\left(1-V(\varpi_\frak{p})\circ U(\varpi_\frak{p})\right)\msf{g}_{\bfcdot}^{(\frak{p})}\right]&= e_{\msf{f},\text{n.o.}}\zeta^*\left[\left(1-(\bfcdot) V(\varpi_\frak{p})\right)d^r\msf{g}_{\bfcdot}^{(\frak{p})}\right]\\
	&=e_{\msf{f},\text{n.o.}}\left(1-(\bfcdot) q_\wp^{-1}V(\varpi_\wp)\right)\zeta^*(d^r\msf{g}_{\bfcdot}^{(\frak{p})})\\
	&=\left(1-(\bfcdot) q_\wp^{-1}\alpha_\msf{f}^{-1}\right)e_{\msf{f},\text{n.o.}}\zeta^*(d^r\msf{g}_{\bfcdot}^{(\frak{p})}).
\end{split}\]
Noting that the $\frak{p}$-depletions of the $\frak{p}$-stabilizations are equal, $(\msf{g}_{\alpha}^{(\frak{p})})^{[\frak{p}]}=(\msf{g}_{\beta}^{(\frak{p})})^{[\frak{p}]}=\msf{g}^{[\frak{p}]}$, we deduce the claim:
\[\begin{split}
e_{\msf{f},\text{n.o.}}\zeta^*(d^r\msf{g})&=\frac{1}{{\alpha_\msf{g}}-{\beta_\msf{g}}}\left({\alpha_\msf{g}} e_{\msf{f},\text{n.o.}}\zeta^*(d^r\msf{g}_{\alpha}^{(\frak{p})})-{\beta_\msf{g}} e_{\msf{f},\text{n.o.}}\zeta^*(d^r\msf{g}_{\beta}^{(\frak{p})})\right)\\
	&=\frac{1}{\left(1-{\alpha_\msf{g}}{\alpha_\msf{f}}^{-1} q_\wp^{-1}\right)\left(1-{\beta_\msf{g}}{\alpha_\msf{f}}^{-1} q_\wp^{-1}\right)}e_{\msf{f},\text{n.o.}}\zeta^*(d^r\msf{g}^{[\frak{p}]}).
\end{split}\]
\end{proof}

\paragraph{Split case.}  For a prime ideal $\frak{p}\in\cal{P}$ with $\wp=\frak{p}\cap\cal{O}_F\in\cal{Q}_\text{split}$, we write $\wp\cal{O}_L=\frak{p}_1\frak{p}_2$.
\begin{lemma}\label{thor}
Let $\frak{N}$ be $\cal{O}_F$-ideal and $\msf{g}\in S_{\ell,x}(\frak{N}p,L;E)$ a cuspform. If $i\not=j$ then 
\[
U(p)\zeta^*\big((\msf{g}^{[\frak{p}_j]})_{\lvert V(\varpi_{\frak{p}_i})}\big)=0\qquad\text{which implies}\qquad
e_\text{n.o.}\zeta^*(\msf{g}_{\lvert V(\varpi_{\frak{p}_i})})=e_\text{n.o.}\zeta^*( (U(\varpi_{\frak{p}_j})\msf{g})_{\lvert V(\varpi_\wp)}).
\]
\end{lemma}
\begin{proof}
By definition, the Fourier coeffient $\mathsf{a}_p(y,(\msf{g}^{[\frak{p}_j]})_{\lvert V(\varpi_{\frak{p}_i})})=\varpi_{\frak{p}_i}^{x-t_L}\mathsf{a}_p(y\varpi_{\frak{p}_i}^{-1}, \msf{g}^{[\frak{p}_j]})=0$ if either $(y\varpi_{\frak{p}_i}^{-1})_{\frak{p}_j}\in\frak{p}_j\cal{O}_{L,\frak{p}_j}$ or $y_{\frak{p}_i}\not\in\frak{p}_i\cal{O}_{L,\frak{p}_i}$. If we write an idele of $L$ as $y=\xi b_j^{-1} \nu\in\widehat{\cal{O}_L}L_{\infty,+}^\times$, where $(b_j^{-1} \nu)_p$ is a unit under our assumptions, we can rephrase the conditions on $y$ for which $\mathsf{a}_p(y,(\msf{g}^{[\frak{p}_j]})_{\lvert V(\varpi_{\frak{p}_i})})=0$ by writing that either $\xi\in\frak{p}_j\cal{O}_{L,\frak{p}_j}$ or $\xi\not\in \frak{p}_i\cal{O}_{L,\frak{p}_i}$. We claim that for all $y=\eta a_i^{-1}u\in\widehat{\cal{O}_F}F_{\infty,+}^\times$ we have $\msf{a}_p\big(y, U(p)[\zeta^*\big((\msf{g}^{[\frak{p}_j]})_{\lvert V(\varpi_{\frak{p}_i})}\big)]\big)=0$. Indeed,
\[\begin{split}
\msf{a}_p\big(y, U(p)[\zeta^*\big((\msf{g}^{[\frak{p}_j]})_{\lvert V(\varpi_{\frak{p}_i})}\big)]\big)&=p_p^{t_L-x}\mathsf{a}_p\big(p\eta a_i^{-1}u,\zeta^*\big((\msf{g}^{[\frak{p}_j]})_{\lvert V(\varpi_{\frak{p}_i})}\big)\big)\\
	&=C(a_i,u)\underset{\xi\in(\frak{a}_i\cal{O}_L)_+,\ \text{tr}_{L/F}(\xi)=p\eta}{\sum} p_p^{t_L-x}\mathsf{a}_p(\xi a_i^{-1}u,(\msf{g}^{[\frak{p}_j]})_{\lvert V(\varpi_{\frak{p}_i})}),
\end{split}\]
so it is enough to check that if $\xi\not\in\frak{p}_j\cal{O}_{L,\frak{p}_j}$ and $\xi\in \frak{p}_i\cal{O}_{L,\frak{p}_i}$ then $p\nmid\text{tr}_{L/F}(\xi)$. Suppose $\xi\not\in\frak{p}_j\cal{O}_{L,\frak{p}_j}$ and $\xi\in \frak{p}_i\cal{O}_{L,\frak{p}_i}$, then the Galois conjugate $\xi^\sigma$ satisfies $\xi^\sigma\in\frak{p}_j\cal{O}_{F,\frak{p}_j}$ and $\xi^\sigma\not\in \frak{p}_i\cal{O}_{F,\frak{p}_i}$ and the trace $\text{tr}_{L/F}(\xi)=\xi+\xi^\sigma$ is both a $\frak{p}_i$ and a $\frak{p}_j$-unit.
\end{proof}

\begin{lemma}\label{lemma: split}
	Let $\frak{N}$ be an $\cal{O}_F$-ideal  coprime to $p$, and $\wp\mid p$ a prime $\cal{O}_F$-ideal splitting in $L/F$. Let $(\ell,x)\in\bb{Z}[I_L]^2$, $(k,w)\in\bb{Z}[I_F]^2$ and $r\in\bb{N}[I_L]$ be such that $(k,w)=((\ell+2r)_{\lvert F},(x+r)_{\lvert F})$. Let $\msf{g}\in S_{\ell,x}(\frak{N},L;E)$ be an eigenvector for the Hecke operators $T(\frak{p}_1),T(\frak{p}_2)$ and let $\msf{f}\in S_{k,w}(\frak{N},F;E)$ be a $p$-nearly ordinary  eigenform. If we denote by $e_{\msf{f},\text{n.o.}}=e_\msf{f}e_\text{n.o.}$ the composition of the $\msf{f}$-isotypic projection with the nearly ordinary projector, we have
\[
e_{\msf{f},\text{n.o.}}\zeta^*\big(d^r\msf{g}^{[\frak{p}_i]}\big)=\frac{\scr{E}_{\wp}(\msf{g},\msf{f})}{\scr{E}_{0,\wp}(\msf{g},\msf{f})} e_{\msf{f},\text{n.o.}}\zeta^*(d^r\msf{g})
\]
where 
\[
\scr{E}_{\wp}(\msf{g},\msf{f})=\underset{\bfcdot,\star\in\{\alpha,\beta\}}{\prod}\left(1-\bfcdot_1\star_2 \alpha_\msf{f}^{-1}q_\wp^{-1}\right),\qquad\scr{E}_{0,\wp}(\msf{g},\msf{f})=1-\alpha_1\beta_1\alpha_2\beta_2(\alpha_\msf{f}^{-1}q_\wp^{-1})^2,
\]
for $\varpi_{\frak{p}_\iota}^{-v}\alpha_\iota,\varpi_{\frak{p}_\iota}^{-v}\beta_\iota$ the inverses of the roots of the Hecke polynomial $1-\mathsf{a}(\frak{p}_\iota,f)X+\text{N}_{L/\bb{Q}}(\frak{p}_\iota)\psi_{\msf{g},0}(\frak{p}_\iota)X^2$ with respect to $T_0(\varpi_{\frak{p}_\iota})$, $\iota=1,2$, and $\alpha_\msf{f}$ is determined by $(e_\text{n.o.}\msf{f})_{\lvert U(\varpi_\wp)}=\alpha_\msf{f}\cdot e_\text{n.o.}\msf{f}$.
\end{lemma}
\begin{proof}
Let $\msf{g}_{\alpha_i}^{(\frak{p}_i)}=(1-{\beta_i}V(\varpi_{\frak{p}_i}))\msf{g}$, $\msf{g}_{\beta_i}^{(\frak{p}_i)}=(1-{\alpha_i}V(\varpi_{\frak{p}_i}))\msf{g}$ be the two $\frak{p}_i$-stabilizations of $\msf{g}$, they satisfy $U(\varpi_{\frak{p}_i})\msf{g}_{\bfcdot}^{(\frak{p}_i)}=(\bfcdot) \msf{g}_{\bfcdot}^{(\frak{p}_i)}$ and $\msf{g}=1/({\alpha_i}-{\beta_i})\big({\alpha_i} \msf{g}_{\alpha_i}^{(\frak{p}_i)}-{\beta_i} \msf{g}_{\beta_i}^{(\frak{p}_i)}\big)$. 
Using Lemma $\ref{thor}$ we compute
\begin{equation}\label{split computation}\begin{split}
e_{\msf{f},\text{n.o.}}&\zeta^*\left[d^r\left(\msf{g}_{\bfcdot}^{(\frak{p}_i)}\right)^{[\frak{p}_i]}\right]=e_{\msf{f},\text{n.o.}}\zeta^*\left[d^r(1-(\bfcdot)V(\varpi_{\frak{p}_i}))\msf{g}_{\bfcdot}^{(\frak{p}_i)}\right]\\
	&= e_{\msf{f},\text{n.o.}}\zeta^*\left[d^r\msf{g}_{\bfcdot}^{(\frak{p}_i)}\right]-(\bfcdot)e_{\msf{f},\text{n.o.}}\zeta^*\left[d^r(U(\varpi_{\frak{p}_j})\msf{g}_{\bfcdot}^{(\frak{p}_i)})_{\lvert V(\varpi_{\wp})}\right]\\
	&= e_{\msf{f},\text{n.o.}}\zeta^*\left[d^r\msf{g}_{\bfcdot}^{(\frak{p}_i)}\right]-(\bfcdot)\alpha_\msf{f}^{-1}q_\wp^{-1}e_{\msf{f},\text{n.o.}}\zeta^*\left[d^r(T(\frak{p}_j)-\text{N}_{L/\bb{Q}}(\frak{p}_j)\psi_\msf{g}(\frak{p}_j)V(\varpi_{\frak{p}_j}))\msf{g}_{\bfcdot}^{(\frak{p}_i)}\right].
\end{split}\end{equation}
Recall that $\msf{g}_{\bfcdot}^{(\frak{p}_i)}$ is an eigenform for the operator $T(\frak{p}_j)$ of eigenvalue $\varpi_j^v\mathsf{a}(\frak{p}_j,\msf{g})=\alpha_j+\beta_j$ and such that $\text{N}_{L/\bb{Q}}(\frak{p}_j)\psi_\msf{g}(\frak{p}_j)=\alpha_j\beta_j$, then the chain of identities in $(\ref{split computation})$ continues as:
\[\begin{split}
e_{\msf{f},\text{n.o.}}&\zeta^*\left[d^r\left(\msf{g}_{\bfcdot}^{(\frak{p}_i)}\right)^{[\frak{p}_i]}\right]=e_{\msf{f},\text{n.o.}}\zeta^*\left[d^r\msf{g}_{\bfcdot}^{(\frak{p}_i)}\right]-	(\bfcdot)\alpha_\msf{f}^{-1}q_\wp^{-1}\Big[(\alpha_j+\beta_j)e_{\msf{f},\text{n.o.}}\zeta^*\left[d^r\msf{g}_{\bfcdot}^{(\frak{p}_i)}\right]+\\
&\qquad\qquad\qquad\qquad\qquad\quad\ \ -\text{N}_{L/\bb{Q}}(\frak{p}_j)\psi_\msf{g}(\frak{p}_j)e_{\msf{f},\text{n.o.}}\zeta^*\left[d^r(U(\varpi_{\frak{p}_i})\msf{g}_{\bfcdot}^{(\frak{p}_i)})_{\lvert V(\varpi_\wp)}\right]\Big]\\
	&=\left(1-(\bfcdot)\alpha_\msf{f}^{-1}q_\wp^{-1}(\alpha_j+\beta_j)+\text{N}_{L/\bb{Q}}(\frak{p}_j)\psi_\msf{g}(\frak{p}_j)\left[(\bfcdot)\alpha_\msf{f}^{-1}q_\wp^{-1}\right]^2\right)e_{\msf{f},\text{n.o.}}\zeta^*\left(d^r\msf{g}_{\bfcdot}^{(\frak{p}_i)}\right)\\
	&=\left(1-(\bfcdot)\alpha_j\alpha_\msf{f}^{-1}q_\wp^{-1}\right)\left(1-(\bfcdot)\beta_j\alpha_\msf{f}^{-1}q_\wp^{-1}\right)e_{\msf{f},\text{n.o.}}\zeta^*\left(d^r\msf{g}_{\bfcdot}^{(\frak{p}_i)}\right).
\end{split}\]
Finally, noting that $\big(\msf{g}_{\alpha_{i}}^{(\frak{p}_i)}\big)^{[\frak{p}_i]}= \big(\msf{g}_{\beta_i}^{(\frak{p}_i)}\big)^{[\frak{p}_i]}=\msf{g}^{[\frak{p}_i]}$, we can put together the previous identities to prove the claim 
\[\begin{split}
e_{\msf{f},\text{n.o.}}\zeta^*(d^r\msf{g})=&\frac{1}{\alpha_i-\beta_i}\left(\alpha_ie_{\msf{f},\text{n.o.}}\zeta^*(d^r\msf{g}_{\alpha_i}^{(\frak{p}_i)})-\beta_ie_{\msf{f},\text{n.o.}}\zeta^*(d^r\msf{g}_{\beta_i}^{(\frak{p}_i)})\right)\\
	&=\frac{1-\alpha_i\beta_i\alpha_j\beta_j(\alpha_\msf{f}^{-1}q_\wp^{-1})^2}{\underset{\bfcdot,\star\in\{\alpha,\beta\}}{\prod}(1-\bfcdot_i\star_j\alpha_\msf{f}^{-1}q_\wp^{-1})}e_{\msf{f},\text{n.o.}}\zeta^*(d^r\msf{g}^{[\frak{p}_i]}).
\end{split}\]

\end{proof}

\paragraph{Interpolation formulas.}
\begin{theorem}\label{thm: interpolation formulas}
For all $(\text{P},\text{Q})\in\Sigma^\msf{f}_\text{cry}$, the twisted triple product $p$-adic $L$-function   $
	\mathscr{L}_p(\breve{\cal{G}},\breve{\cal{F}})\in \mathbf{R}_{\cal{G}}\otimes_{\mathbf{A}_F}\text{Frac}(\mathbf{K}_{\cal{F}})$ satisfies the equality
	\[
	\mathscr{L}_p\big(\breve{\cal{G}},\breve{\cal{F}}\big)(\text{P},\text{Q})=\frac{1}{\msf{E}(\msf{f}_\text{Q}^*)}
	\left(\prod_{\wp\in\cal{Q}_\text{inert}}\cal{E}_\wp(\msf{g}_\text{P},\msf{f}_\text{Q}^*)
	\prod_{\wp\in\cal{Q}_\text{split}}\frac{\scr{E}_\wp(\msf{g}_\text{P},\msf{f}_\text{Q}^*)}{\scr{E}_{0,\wp}(\msf{g}_\text{P},\msf{f}_\text{Q}^*)}\right)
	\frac{\big\langle \zeta^*\left(\delta^r\breve{\msf{g}}_\text{P}\right), \breve{\msf{f}}_\text{Q}^{*}\big\rangle}{\big\langle \msf{f}_\text{Q}^{*} \msf{f}_\text{Q}^{*}\big\rangle},
	\]
	where the Euler factors appearing in the formula are defined in Lemmas $\ref{lemma: inert}$ and $\ref{lemma: split}$, and $r\in\bb{N}[I_L]$ is the unique element, with $r_\mu=0$ if $\mu\notin\Theta_{L/F}$, that makes the weights of $\zeta^*(\delta^r\breve{\msf{g}}_\text{P})$ and $\breve{\msf{f}}_\text{Q}$ match.
\end{theorem}
\begin{proof}
We use $(\ref{eq: second expression p-adic L-function})$ for the value of the $p$-adic $L$-function  at a point $(\text{P},\text{Q})\in\Sigma^\msf{f}_\text{cry}$, and Lemmas $\ref{lemma: inert},\ref{lemma: split}$ give us
\[\begin{split}
\mathscr{L}_p\big(\breve{\cal{G}},\breve{\cal{F}}\big)(\text{P},\text{Q})&=\frac{1}{\msf{E}(\msf{f}_\text{Q}^*)}\frac{\big\langle e_\text{n.o.}\zeta^*\big(d^{r}\breve{\msf{g}}_\text{P}^{[\cal{P}]}\big), \breve{\msf{f}}_\text{Q}^{*}\big\rangle}{\big\langle \msf{f}_\text{Q}^{*}, \msf{f}_\text{Q}^{*}\big\rangle}\\
	&=\frac{1}{\msf{E}(\msf{f}_\text{Q}^*)}\left(\prod_{\wp\in\cal{Q}_\text{inert}}\cal{E}_\wp(\msf{g}_\text{P},\msf{f}_\text{Q}^*)
\prod_{\wp\in\cal{Q}_\text{split}}\frac{\scr{E}_\wp(\msf{g}_\text{P},\msf{f}_\text{Q}^*)}{\scr{E}_{0,\wp}(\msf{g}_\text{P},\msf{f}_\text{Q}^*)}\right)\frac{\big\langle e_{\msf{f}_\text{Q}^{*},\text{n.o}}\zeta^*\big(d^r\breve{\msf{g}}_\text{P}\big),\breve{\msf{f}}_\text{Q}^{*}\big\rangle}{\big\langle \msf{f}_\text{Q}^{*}, \msf{f}_\text{Q}^{*}\big\rangle}.
\end{split}\]
We conclude the proof applying Lemma $\ref{lemma: auto comparison}$ to compare $p$-adic and real analitic differential operators on cuspforms: $e_{\msf{f}_\text{Q}^{*},\text{n.o}}\zeta^*\big(d^r\breve{\msf{g}}_\text{P}\big)=e_{\msf{f}_\text{Q}^{*},\text{n.o}}\Pi^\text{hol}\zeta^*\big(\delta^r\breve{\msf{g}}_\text{P}\big)$.
\end{proof}

\begin{remark} Recall that for every $(\text{P},\text{Q})\in\Sigma^\msf{f}_\text{cry}$
	  there is a unitary
	  automorphic representation $\Pi_{\text{P},\text{Q}}$ of prime-to-$p$ level. The Euler factors in Theorem $\ref{thm: interpolation formulas}$ also
	appear the expression for the local $L$-factor $L_p(\frac{1}{2},\Pi_{\text{P},\text{Q}},r)$. Indeed, if $\wp\in\cal{Q}_\text{inert}$ by using $(\ref{eq:Satake1})$, $(\ref{eq:Satake3})$ 
	   we compute
\[
\begin{split}
\cal{E}_{\wp}(\msf{g},\msf{f}^*)&=\left(1-\alpha_\msf{g}\alpha_{\msf{f}^*}^{-1}q_\wp^{-1}\right)\left(1-\beta_\msf{g}\alpha_{\msf{f}^*}^{-1}q_\wp^{-1}\right)\\
	&=\left(1-\chi_{1,\frak{p}}(\varpi_\wp)\psi_{i,\wp}(\varpi_\wp)q_\wp^{-1/2}\right)\left(1-\chi_{2,\frak{p}}(\varpi_\wp)\psi_{k,\wp}(\varpi_\wp)q_\wp^{-1/2}\right).\\
\end{split}
\]
Similarly if $\wp\in\cal{Q}_\text{split}$ by using $(\ref{eq:Satake1})$, $(\ref{eq:Satake2})$ get
\[
\begin{split}
\scr{E}_{\wp}(\msf{g},\msf{f}^*)&=\underset{\bfcdot,\star\in\{\alpha,\beta\}}{\prod}\left(1-\bfcdot_1\star_2 \alpha_{\msf{f}^*}^{-1}q_\wp^{-1}\right)\\
	&=\underset{i,j\in\{1,2\}}{\prod}\left(1-\chi_{i,\frak{p}_1}(\varpi_\wp)\chi_{j,\frak{p}_2}(\varpi_\wp) \psi_{k,\wp}(\varpi_\wp)q_\wp^{-1/2}\right).
\end{split}
\]
\end{remark}


\section{Geometric theory}
\label{sec: Geom_part}
\subsection{Geometric Hilbert modular forms}
\label{subsec: Geom_HMFS}
Let $F$ be a totally real number field and $G=\text{Res}_{L/\bb{Q}}(\text{GL}_{2,F})$.
For any open compact subgroup $K\subset G(\bb{A}^\infty)$ we consider the Shimura variety
\[
\text{Sh}_K(G)(\bb{C})=G(\bb{Q})\backslash (\frak{H}^{\pm})^{I}\times G(\bb{A}^\infty)/ K
\]
where $\gamma\in G(\bb{Q})=\text{GL}_2(F)$ acts on $z=(z_\tau)_\tau\in (\frak{H}^{\pm})^{I}$ via Moebius transformations $
\gamma\cdot z=(\tau(\gamma)z_\tau)_\tau$. The complex manifold
$\text{Sh}_K(G)(\bb{C})$ has a canonical structure of quasi-projective variety over its reflex field $\bb{Q}$ (\cite{Milne} Chapter II, Theorem 5.5).
Let $\omega$ be the tautological quotient bundle on $\bb{P}^1_\bb{C}$, $
\omega=\{(x,z)\in\bb{P}^1_\bb{C}\times\bb{C}^2\lvert z\in x\}$,
with $p_1:\omega\to \bb{P}^1_\bb{C}$ the projection on the first component. The group $\text{GL}_2(\bb{C})$ acts on $\bb{P}^1_\bb{C}$ via Moebius transformations and there is a natural way to define a $\text{GL}_2(\bb{C})$-action on $\omega$ such that the projection $p_1$ is equivariant. For any weight $(k,w)\in\bb{Z}[I]^2$ such that $k-2w=mt$, one can define a line bundle
\begin{equation}\label{complex line bundle}
\underline{\omega}^{(k,w)}=\underset{\tau\in I}{\bigotimes}\ \text{pr}_\tau^*\bigg(\omega^{\otimes k_\tau}\otimes \text{det}^{\frac{m+k_\tau}{2}}\bigg)
\end{equation}
on $(\bb{P}^1_\bb{C})^{I}$ with $G(\bb{C})$-action given as follows. For each $\tau\in I$, the action of $G(\bb{C})$ on $\text{pr}_\tau^*\big(\omega^{\otimes k_\tau}\otimes \text{det}^{\frac{m+k_\tau}{2}}\big)$ factors through the $\tau$-copy of $\text{GL}_2(\bb{C})$, which in turn acts as the product of $\text{det}^{\frac{m+k_\tau}{2}}$ and the $k_\tau$-th power of the natural action on $\omega$. One has to twist the action by such a power of the determinant because it allows the line bundle to descend to the Galois closure $F^{\text{Gal}}$ of $F$ over $\bb{Q}$. Indeed, consider the subgroup $
Z_s=\text{Ker}\big(\text{N}_{F/\bb{Q}}:\text{Res}_{F/\bb{Q}}(\bb{G}_m)\to\bb{G}_m\big)$
of the center $Z=\text{Res}_{F/\bb{Q}}(\bb{G}_m)$ of $G$ and denote by $G^c$ the quotient of $G$ by $Z_s$. The action of $G(\bb{C})$ on $\underline{\omega}^{(k,w)}$ factors through $G^c(\bb{C})$, thus $\underline{\omega}^{(k,w)}$ descends to an algebraic invertible sheaf on $\text{Sh}_K(G)_\bb{C}$ if $K$ is sufficiently small by (\cite{Milne} Chapter III,  Proposition 2.1), and it has a canonical model over $F^{\text{Gal}}$ by  (\cite{Milne} Chapter III, Theorem 5.1). 

Suppose $F\not=\bb{Q}$, then for every field  $L$, $F^\text{Gal}\subset L\subset\bb{C}$, and sufficiently small compact open subgroup $K\triangleleft\ G(\bb{A}^\infty)$, one can give a geometric interpretation of Hilbert modular forms of weight $(k,w)$, level $K$, defined over $L$ as $
M_{k,w}(K;L)= H^0(\text{Sh}_K(G)_L, \underline{\omega}^{(k,w)})$. To consider cuspforms and treat the case $F=\bb{Q}$, one has to consider compactifications of the Shimura variety $\text{Sh}_K(G)_\bb{Q}$, which we discuss in Section $\ref{subsect: Compactifications and p-adic theory}$.

\subsubsection{Integral models}
\label{subsubsec: Integral models}
Fix $p$ a rational prime unramified in $F$ and consider a level structure of type $K=K^pK_p$, where $K^p$ is an open compact subgroup of $G\big(\widehat{\cal{O}_F}^{p}\big)$ and $K_p=\text{GL}_2(\cal{O}_F\otimes\bb{Z}_p)$. The determinant map $\det: G\to \text{Res}_{F/\bb{Q}}(\bb{G}_m)$ induces a bijection between the set of geometric connected components of $\text{Sh}_K(G)$ and $\text{cl}^+_F(K)$, the strict class group of $K$, $\text{cl}^+_F(K)=F^\times_+\backslash \bb{A}_F^{\infty,\times}/\det(K)$.
Since $\det(K)\subseteq \widehat{\cal{O}_{F}}^\times$, there is a surjection $\text{cl}^+_F(K)\twoheadrightarrow \text{cl}^+_F$ to the strict ideal class group of $F$, which one uses to label the geometric components of the Shimura variety $\text{Sh}_K(G)$. 
Fix fractional ideals $\frak{c}_1,\dots,\frak{c}_{h^+_F}$, coprime to $p$, forming a set of representatives of $\text{cl}^+_F$. Then by strong approximation there is a decomposition
\[
\text{Sh}_K(G)(\bb{C})=G(\bb{Q})^+\backslash \frak{H}^{I}\times G(\bb{A}^\infty)/K=\underset{[\frak{c}]\in \text{cl}^+_F}{\coprod} \text{Sh}^\frak{c}_K(G)(\bb{C}),
\]
where each $\text{Sh}^\frak{c}_K(G)(\bb{C})$ is the disjoint union of quotients of $\frak{H}^I$ by groups of the form $\Gamma(g,K)=gKg^{-1}\cap G(\bb{Q})^+$. A different choice  $\frak{c}'$ of fractional ideal representing $[\frak{c}]\in\text{cl}_F^+$ produces a canonically isomorphic manifold $\text{Sh}^{\frak{c}'}_K(G)(\bb{C})\cong\text{Sh}^\frak{c}_K(G)(\bb{C})$ (\cite{Hilbert} Remark 2.8). Suppose $K^p$ is sufficiently small so there exists a smooth, quasi-projective $\bb{Z}_{(p)}$-scheme $\cal{M}_K^\frak{c}$ representing the moduli problem of isomorphism classes of quadruples $(A,\iota,\lambda,\alpha_{K^p})_{/S}$ where $(A,\iota)$ is a Hilbert-Blumenthal abelian variety over $S$ of dimension $g=[F:\bb{Q}]$, $\lambda$ a $\frak{c}$-polarization and $\alpha_{K^p}$ a level-$K^p$ structure, (\cite{Hilbert} Section 2.3). 

The group of totally positive units $\cal{O}^\times_{F,+}$ acts on $\cal{M}_K^\frak{c}$ by modifying the $\frak{c}$-polarization. The subgroup $(K\cap\cal{O}^\times_F)^2$ of $\cal{O}^\times_{F,+}$ acts trivially, where by $K\cap\cal{O}^\times_F$ we mean the intersection of $K$ and $\cal{O}^\times_F\hookrightarrow Z(\bb{A}^\infty)$ in $G(\bb{A}^\infty)$. Therefore, the finite group $\cal{O}^\times_{F,+}/(K\cap\cal{O}^\times_F)^2$ acts on the moduli scheme $\cal{M}_K^\frak{c}$ and the stabilizer of each geometric connected component is $
\big(\det(K)\cap\cal{O}_{F,+}^\times\big)/(K\cap\cal{O}^\times_F)^2$.

\begin{proposition}\label{neige}
There is an isomorphism between the quotient of $\cal{M}_K^{\frak{c}\frak{d}^{-1}}(\bb{C})$ by the finite group $\cal{O}^\times_{F,+}/(K\cap\cal{O}^\times_F)^2$ and $\text{Sh}_K^\frak{c}(\bb{C})$. Moreover, if $\det(K)\cap\cal{O}_{F,+}^\times=(K\cap\cal{O}^\times_F)^2$, then the quotient map $\cal{M}_K^{\frak{c}\frak{d}^{-1}}(\bb{C})\to\text{Sh}_K^\frak{c}(\bb{C})$ induces an isomorphism between any geometric connected component of $\cal{M}_K^{\frak{c}\frak{d}^{-1}}(\bb{C})$ and its image.
\end{proposition}
\begin{proof}
This is (\cite{Hilbert} Proposition 2.4). However, note the shift by the absolute different in the indeces in our statement. It is necessary for the conventions for the complex uniformization used in their work (\cite{Hida} Section 4.1.3).
\end{proof}

\begin{definition}\label{def: Integral model}
Let $p$ be a rational prime unramified in $F$ and $K=K^pK_p$ a compact open subgroup of $G(\bb{A}^\infty)$ such that
$K^p$ is sufficiently small, $K_p=\text{GL}_2(\cal{O}_F\otimes\bb{Z}_p)$ and $\det(K)\cap\cal{O}_{F,+}^\times=(K\cap\cal{O}^\times_F)^2$. The integral model of the Shimura variety $\text{Sh}_K(G)$ over $\bb{Z}_{(p)}$  is  the quotient of $\cal{M}_K=\coprod_{[\frak{c}]\in\text{cl}^+_F}\cal{M}^\frak{c}_K$ by $\cal{O}_{F,+}^\times/(K\cap\cal{O}_F)^2$, which we denote $\mathbf{Sh}_K(G)$. 
\end{definition}

Note that the assumptions on the level $K$ in the definition are always satisfied up to replacing $K^p$ by an open compact subgroup (\cite{Hilbert} Lemma 2.5). Moreover, by Proposition $\ref{neige}$, the scheme $\mathbf{Sh}_K(G)$ is smooth  quasi-projective over $\bb{Z}_{(p)}$ and has an abelian scheme with real multiplication over it.

\begin{remark}
The scheme $\cal{M}_K^{\frak{d}^{-1}}$ is the integral model of the Shimura variety for the algebraic group $G^*$ \cite{Rapoport}. We denote it $\mathbf{Sh}_K(G^*)$ and we let $\xi:\mathbf{Sh}_K(G^*)\to \mathbf{Sh}_K(G)$ be the natural morphism.
\end{remark}


\subsubsection{Diagonal morphism.}
\label{subsec: Diagonal morphism}
Let $L/F$ be an extension of totally real fields with $[F:\bb{Q}]=g$. Consider the map of algebraic groups $\zeta: G_F\longrightarrow G_L$
define by the natural inclusion $\zeta(B):\text{GL}_{2}(B\otimes_\bb{Q}F)\to \text{GL}_{2}(B\otimes_\bb{Q}L)$ of groups for any $\bb{Q}$-algebra $B$. For compact open subgroups $K\subset G_L(\bb{A}^\infty)$ and $K'\subset K\cap G_F(\bb{A}^\infty)$ we have a commutative diagram
\begin{equation}\label{eq: complex proof}
\xymatrix{
\text{Sh}_{K'}(G_F)(\bb{C})\ar[r]^\zeta\ar[d]_\det& \text{Sh}_{K}(G_L)(\bb{C})\ar[d]^\det\\
\text{cl}^+_F(K')\ar[r]^\zeta\ar[d]& \text{cl}^+_L(K)\ar[d]\\
\text{cl}^+_F\ar[r]^\zeta& \text{cl}^+_L,
}\end{equation}
hence for every fractional ideal $\frak{c}$ of $F$ there is an induced map $\zeta:\text{Sh}_{K'}^\frak{c}(G_F)(\bb{C})\to \text{Sh}_{K}^\frak{c}(G_L)(\bb{C})$. Suppose that $K\subset G_L(\bb{A}^\infty)$ and $K'\subset K\cap G_F(\bb{A}^\infty)$ satisfy the assumptions in Definition $\ref{def: Integral model}$. There is a morphism of $\bb{Z}_{(p)}$-schemes $\zeta:\mathbf{Sh}_{K'}(G_F)\longrightarrow\mathbf{Sh}_{K}(G_L)$ induced by morphisms $\tilde{\zeta}:\cal{M}_{K',F}\to\cal{M}_{K,L}$  that maps any quadruple $[A, \iota,\lambda, \alpha_{(K')^p}]_{/S}$ over a $\bb{Z}_{(p)}$-scheme $S$ to the quadruple $\tilde{\zeta}\left([A, \iota,\lambda, \alpha_{(K')^p}]\right)=[A', \iota', \lambda', \alpha_{K^p}']_{/S}$ over $S$ defined as follows. First, the abelian scheme $A'$ is $A\otimes_{\cal{O}_F}\cal{O}_{L}$, then we can describe the $\cal{O}_L$-action on $\cal{O}_L$ via a ring homomorphism $\bar{\iota}:\cal{O}_L\to \text{M}_g(\cal{O}_F)$ by choosing an $\cal{O}_F$-basis of $\cal{O}_L$; the choice of basis induces an identification between $A\otimes_{\cal{O}_F}\cal{O}_L$ and $A^g$. Thus, the ring homomorphism $\iota':\cal{O}_L\to \text{End}_S(A')$ is defined as the arrow that makes the following diagram commute
\[\xymatrix{
\cal{O}_L\ar@{.>}[dr]^{\iota'}\ar[r]^{\bar{\iota}}& \text{M}_g(\cal{O}_F)\ar[d]\\
& \text{M}_g\big(\text{End}_S(A)\big)\cong\text{End}_S(A').
}\]
Following (\cite{burgos} Lemma 5.11), one can compute the dual abelian scheme $\left(A'\right)^\vee\cong A^\vee\otimes_{\cal{O}_F}\frak{d}_{L/F}^{-1}$ and realize that if $\lambda: (\frak{c},\frak{c}^+)\overset{\sim}{\to}(\text{Hom}_{\cal{O}_F}^\text{sym}(A,A^\vee), \text{Hom}_{\cal{O}_F}^\text{sym}(A,A^\vee)^+)$ is a $\frak{c}$-polarization of $A$ then $\lambda'=\lambda\otimes\text{id}$ is a $\frak{c}\otimes_{\cal{O}_F}\frak{d}_{L/F}^{-1}$-polarization of $A'=A\otimes_{\cal{O}_F}\cal{O}_L$ . Finally, it is enough to define $\tilde{\zeta}$ for principal $\frak{N}$-level structures, for $\frak{N}$ an $\cal{O}_F$-ideal. A principal $\frak{N}$-level structure is an $\cal{O}_F$-linear isomorphism of group schemes $(\cal{O}_F/\frak{N})^2\overset{\sim}{\to} A[\frak{N}]$, by tensoring such an isomorphism  with $\cal{O}_L$ over $\cal{O}_F$  we obtain a principal $\frak{N}$-level structure on $A'$.

\begin{remark}
For any fractional ideal $\frak{c}$ of $F$ there is a commutative diagram
\[
\xymatrix{
\cal{M}_{K',F}^{\frak{c}\frak{d}_F^{-1}}\ar[r]^{\tilde{\zeta}}\ar[d]& \cal{M}_{K,L}^{\frak{c}\frak{d}_L^{-1}}\ar[d]\\
\mathbf{Sh}_{K'}^\frak{c}(G_F)\ar[r]^{\zeta}& \mathbf{Sh}_{K}^{\frak{c}}(G_L)}, \quad\text{implying }\quad \xymatrix{
\mathbf{Sh}_{K'}(\text{GL}_{2,\bb{Q}})\ar@{.>}[dr]_{\zeta}\ar[r]^{\tilde{\zeta}}& \cal{M}_{K,L}^{\frak{d}_L^{-1}}\ar[d]^{\xi}\\
& \mathbf{Sh}_{K}(G_L)}\quad\text{when $F=\bb{Q}$}.
\]
\end{remark}

\subsection{Compactifications and $p$-adic theory}\label{subsect: Compactifications and p-adic theory}
Sometimes we drop part of the decorations from the symbols denoting Shimura varieties when we believe it does not cause confusion, both to symplify the notation and to state facts that hold for both groups $G$ and $G^*$. We denote by $\mathbf{Sh}_{K,R}^*$ the minimal compactification of $\mathbf{Sh}_{K,R}$, it is a normal projective scheme and there is a natural open immersion $\oint:\mathbf{Sh}_{K,R}^\text{tor}\longrightarrow\mathbf{Sh}_{K,R}^*$.
By choosing some auxiliary data $\Sigma$, one can construct an arithmetic toroidal compactification $\mathbf{Sh}_{K,\Sigma}^\text{tor}$ smooth and  projective over $\bb{Z}_{(p)}$, such that the natural morphism $\xi$ extends to $\xi:\mathbf{Sh}_{K',\Sigma'}^\text{tor}(G^*)\to\mathbf{Sh}^\text{tor}_{K,\Sigma}(G)$. 
 There is a natural open immersion $\mathbf{Sh}_K\hookrightarrow \mathbf{Sh}^\text{tor}_{K}$ such that the boundary $\text{D}=\mathbf{Sh}^\text{tor}_{K}\setminus\mathbf{Sh}_K$ is a relative simple normal crossing Cartier divisor. The Hilbert-Blumenthal abelian scheme $\cal{A}$ over $\mathbf{Sh}_K$ extends to a semi-abelian scheme $\cal{A}^\text{sa}\to \mathbf{Sh}^\text{tor}_{K}$ with an $\cal{O}_F$-action and a $K$-level structure \cite{Rapoport}, (\cite{Lan13} Chapter  $\text{VI}$). There is a canonical way to extend the rank $2$ vector bundle of relative de Rham cohomology $\bb{H}^1_\text{dR}(\cal{A}/\mathbf{Sh}_{K,R})$ to an $(\cal{O}_{\mathbf{Sh}^\text{tor}_{K}}\otimes_{\bb{Z}}\cal{O}_F)$-module $\bb{H}^1$ locally free of rank $2$ over $\mathbf{Sh}^\text{tor}_{K,R}$ together with Gauss-Manin connection and Kodaira-Spencer isomorphism. If $\underline{\omega}=e^*\big(\Omega^1_{\cal{A}^\text{sa}/ \mathbf{Sh}^\text{tor}_{K}}\big)$ is the cotangent space at the origin of the universal semi-abelian scheme, the vector bundle $\bb{H}^1$ has an $\cal{O}_F$-equivariant Hodge filtration
\[\xymatrix{
0\ar[r]& \underline{\omega}\ar[r]& \bb{H}^1\ar[r]& \text{Lie}((\cal{A}^\text{sa})^\vee)\ar[r]& 0.
}\]
Let $R$ be an $\cal{O}_{F^\text{Gal}}$-algebra in which the discriminant $d_{F/\bb{Q}}$ is invertible. For a coherent $(\cal{O}_{\mathbf{Sh}^\text{tor}_{K,R}}\otimes_{\bb{Z}}\cal{O}_F)$-module $M$, we denote by $M=\bigoplus_{\tau\in I} M_\tau$ its canonical decomposition for the $\cal{O}_F$-action (\cite{KatzCM} Lemma 2.0.8): $M_\tau$ is the direct summand of $M$ on which $\cal{O}_F$ acts via $\tau:\cal{O}_F\to R\to \cal{O}_{\mathbf{Sh}^\text{tor}_{K,R}}$. Then the $\tau$-component of the Hodge filtration is
\[\xymatrix{
0\ar[r]& \underline{\omega}_\tau\ar[r]& \bb{H}^1_\tau\ar[r]& \bigwedge^2(\bb{H}^1_\tau)\otimes\underline{\omega}_\tau^{-1}\ar[r]& 0
}.\]
For a weight $(k,w)\in\bb{Z}[I]^2$ with $k-2w=mt$, we define the integral model of the line bundle $(\ref{complex line bundle})$ by
\[
\underline{\omega}^{(k,w)}:=\underset{\tau\in I}{\bigotimes}\Big((\wedge^2\bb{H}^1_\tau)^{-\frac{m+k_\tau}{2}}\otimes\underline{\omega}_\tau^{k_\tau}\Big)
\]
as a sheaf over $\mathbf{Sh}^\text{tor}_{K,R}(G)$. The geometric interpretation of modular forms for general totally real number field $F$ and of cuspforms is achieved by $M_{k,w}(K;R)=H^0\big(\mathbf{Sh}_K^\text{tor}(G)_{R},\underline{\omega}^{(k,w)}\big)$ and  $S_{k,w}(K;R)=H^0\big(\mathbf{Sh}_K^\text{tor}(G)_{R},\underline{\omega}^{(k,w)}(-\text{D})\big)$.

\begin{remark}
A general compact open subgroup $K\triangleleft G(\bb{A}^\infty)$ of prime-to-$p$ level doesn't satisfy the assumptions in Definition $\ref{def: Integral model}$. Anyway, one can work with modular forms of level $K$ by considering a subgroup $K'$ that does satisfy them and then take $K/K'$-invariants (\cite{Hilbert} Section 6.4).
\end{remark}

\begin{definition}\label{def: geometric cuspforms}
Let $R$ be an $(\cal{O}_{F}^\text{Gal})_{(p)}$-algebra  and let $(k,w)\in\bb{Z}[I]^2$ be any weight. We define a line bundle over $\mathbf{Sh}^\text{tor}_{K,R}(G^*)$ by
\[
\underline{\omega}_{G^*}^{(k,u)}:=\underset{\tau\in I}{\bigotimes}\Big((\wedge^2\bb{H}^1_\tau)^{u_\tau-k_\tau}\otimes\underline{\omega}_\tau^{k_\tau}\Big).
\]
It provides a geometric incarnation of automorphic forms on $G^*$ of weight $(k,u)\in\bb{Z}[I]^2$ by $
M^{*}_{k,u}(K;R)=H^0\big(\mathbf{Sh}_K^\text{tor}(G^*)_{R},\underline{\omega}_{G^*}^{(k,u)}\big)$ and $ S^{*}_{k,u}(K;R)=H^0\big(\mathbf{Sh}_K^\text{tor}(G^*)_{R},\underline{\omega}_{G^*}^{(k,u)}(-\text{D})\big)$.
\end{definition}
A weight $(k,w)\in\bb{Z}[I]^2$, $k-2w=mt$, is cohomological if $2-m\ge k_\tau\ge2$ for all $\tau\in I$. For any cohomological weight we define the vector bundle $\cal{F}^{(k,w)}$ on $\mathbf{Sh}_{K,R}^\text{tor}(G)$ by $\cal{F}^{(k,w)}:=\bigotimes_{\tau\in I}\cal{F}_\tau^{(k,w)}$ for $\cal{F}_\tau^{(k,w)}:=(\wedge^2\bb{H}^1_\tau)^{\frac{2-m-k_\tau}{2}}\otimes\text{Sym}^{k_\tau-2}\bb{H}^1_\tau$. Similarly, a weight $(k,u)\in\bb{Z}[I]^2$ is cohomological if $k\ge2t_F$ and $t+u\ge k$. For any cohomological weight we define the vector bundle $\cal{F}_{G^*}^{(k,u)}$ on $\mathbf{Sh}_{K,R}^\text{tor}(G^*)$ by $
\cal{F}_{G^*}^{(k,u)}:=\bigotimes_{\tau\in I}\cal{F}_{G^*,\tau}^{(k,u)}$ for $\cal{F}_{G^*,\tau}^{(k,u)}:=(\wedge^2\bb{H}^1_\tau)^{1+u_\tau-k_\tau}\otimes\text{Sym}^{k_\tau-2}\bb{H}^1_\tau$. The extended Gauss-Manin connection on $\bb{H}^1$ induces by functoriality integrable connections $\nabla:\cal{F}^{(k,w)}\to \cal{F}^{(k,w)}\otimes\Omega^1_{\mathbf{Sh}^\text{tor}_{K,R}(G)}(\log \text{D})$ and  $\nabla:\cal{F}_{G^*}^{(k,u)}\to \cal{F}_{G^*}^{(k,u)}\otimes\Omega^1_{\mathbf{Sh}^\text{tor}_{K,R}(G^*)}(\log \text{D})$
out of which one can form the complexes 
\[
\text{DR}_c^{\bullet}\big(\cal{F}^{(k,w)}\big) \ = \ \left[0\to\cal{F}^{(k,w)}(-\text{D})\overset{\nabla}{\longrightarrow} \cdots\overset{\nabla}{\longrightarrow}\cal{F}^{(k,w)}(-\text{D})\otimes\Omega^g_{\mathbf{Sh}^\text{tor}_{K,R}(G)}(\log \text{D})\to 0\right],
\]
\[
\text{DR}_c^{\bullet}\big(\cal{F}_{G^*}^{(k,u)}\big) \ = \ \left[0\to\cal{F}_{G^*}^{(k,u)}(-\text{D})\overset{\nabla}{\longrightarrow} \cdots\overset{\nabla}{\longrightarrow}\cal{F}_{G^*}^{(k,u)}(-\text{D})\otimes\Omega^g_{\mathbf{Sh}^\text{tor}_{K,R}(G^*)}(\log \text{D})\to 0\right]
\]
equipped with their natural Hodge filtration. One can associate to $\text{DR}_c^{\bullet}\big(\cal{F}^{(k,w)}\big)$ and $\text{DR}_c^{\bullet}\big(\cal{F}_{G^*}^{(k,u)}\big)$ a dual BGG complex. We recall the definition of $\text{BGG}_c(\cal{F}_{G^*}^{(k,u)})$ and we refer to (\cite{Hilbert}  Section 2.15 ) for the definition of $\text{BGG}_c(\cal{F}^{(k,w)})$. For any subset $J\subset I$, let $s_J\in\{\pm1\}^I$ be the element whose $\tau$-component is $-1$ if $\tau\not\in J$ and $1$ if $\tau\in J$. For $0\le j\le g$ we put 
\[
\text{BGG}_c^j(\cal{F}_{G^*}^{(k,u)})=\bigoplus_{J\subset I,\ \#J=j}\underline{\omega}_{G^*}^{s_J\cdot(k,u)}(-D)e_J
\]
for $e_J$ the Cech symbol and $\underline{\omega}_{G^*}^{s_J\cdot(k,u)}=\big(\bigotimes_{\tau\not\in J}(\wedge^2\bb{H}^1_\tau)^{u_\tau-1}\otimes\underline{\omega}_\tau^{2-k_\tau}\big)\otimes\big(\bigotimes_{\tau\in J}(\wedge^2\bb{H}^1_\tau)^{u_\tau-k_\tau}\otimes\underline{\omega}_\tau^{k_\tau}\big)$.
There are differential operators $d^j:\text{BGG}_c^j(\cal{F}_{G^*}^{(k,u)})\to \text{BGG}_c^{j+1}(\cal{F}_{G^*}^{(k,u)})$ given on local sections by $
d^j: \msf{f}e_J\mapsto\sum_{\tau_0\not\in J}\Theta_{\tau_0,k_{\tau_0-1}}(\msf{f})e_{\tau_0}\wedge e_J$ where  $\Theta_{\tau_0,k_{\tau_0}-1}(\msf{f})=\frac{(-1)^{k_{\tau_0}-2}}{(k_{\tau_0}-2)!}\sum_\xi \tau_0(\xi)^{k_{\tau_0}-1}(\xi)a_\xi q^\xi$ if the local section is written as $\msf{f}=\sum_\xi a_\xi q^\xi$.

\begin{theorem}{(\cite{Hilbert} Theorem 2.16, \cite{BGG} remark 5.24)}\label{th: dual BGG quasi-iso}
Let $R$ be an $F^\text{Gal}$-algebra, then for $\cal{S}=\cal{F}^{(k,w)}$ (resp. $\cal{S}=\cal{F}_{G^*}^{(k,u)}$) there is a canonical quasi-isomorphic embedding $
\text{BGG}_c^{\bullet}\big(\cal{S}\big)\hookrightarrow \text{DR}_c^{\bullet}\big(\cal{S}\big)$ of complexes of abelian sheaves on $\mathbf{Sh}^\text{tor}_{K,R}(G)$ (resp. $\mathbf{Sh}^\text{tor}_{K,R}(G^*)$). Moreover, the Hodge spectral sequences for both complexes degenerate at the first page.
\end{theorem}


\subsubsection{$p$-adic theory}
\label{subsubsec: p-adic theory}
Katz's idea for a geometric theory of $p$-adic modular forms \cite{p-adicpropertiesmodularschemes} consists in removing from the relevant Shimura variety the preimages, under the specialization map, of those points in the special fiber that correspond to non-ordinary abelian varieties. 

Let $E\subset\bb{C}$ be a number field containing $F^\text{Gal}$. The fixed embedding $\iota_p:\overline{\bb{Q}}\hookrightarrow\overline{\bb{Q}}_p$ determines a prime ideal $\wp\mid p$ of $E$. We denote by $E_\wp$ the completion, $\cal{O}_\wp$ the ring of integers and $k$ the residue field. Let $\cal{A}^\text{sa}_k$ be the semi-abelian scheme over the special fiber $\mathbf{Sh}^\text{tor}_{K,k}$ of the Shimura variety. The determinant of the map, induced by Vershiebung $\text{V}: (\cal{A}^\text{sa}_k)^{(p)}\to\cal{A}^\text{sa}_k$ between cotangent spaces at the origin, corresponds to a characteristic $p$ Hilbert modular form  $h\in H^0(\mathbf{Sh}^\text{tor}_{K,k}, \det(\underline{\omega})^{\otimes (p-1)})$,
called the Hasse invariant. The ordinary locus $\mathbf{Sh}^{\text{tor},\text{ord}}_{K,k}$ is the complement of the zero locus of the Hasse invariant. Let $\scr{S}^\text{tor}_{K}$ denote the formal completion of $\mathbf{Sh}^\text{tor}_{K,\cal{O}_\wp}$ along its special fiber and 
$j:\big]\mathbf{Sh}^{\text{tor},\text{ord}}_{K,k}\big[\hookrightarrow \scr{S}^\text{tor}_{K,\text{rig}}$ the inverse image  of the ordinary locus under the specialization map $\text{sp}:\scr{S}^\text{tor}_{K,\text{rig}}\to \mathbf{Sh}^{\text{tor}}_{K,k}$. For the minimal compactification $\mathbf{Sh}^{*}_{K,k}$ one can similarly define the ordinary locus $\mathbf{Sh}^{*,\text{ord}}_{K,k}$ of the special fiber, which is an affine scheme, since $\det(\underline{\omega})$ is an ample line bundle on $\mathbf{Sh}^{*}_{K,k}$. This is a very convenient feature because it implies the existence of a fundamental system of strict affinoid neighborhoods of $]\mathbf{Sh}^{*,\text{ord}}_{K,k}[$. 

\begin{theorem}\label{th: p-adic cohomology}
We recall that overconvergent cuspforms of weight $(k,w)\in\bb{Z}[I]^2$ are defined as $S^\dagger_{k,w}(K;E_\wp)=H^0\big(\scr{S}^\text{tor}_{K,\text{rig}}, j^\dagger\big(\underline{\omega}^{(k,w)}(-\text{D})\big)\big)$. For any cohomological weight $(k,w)\in\bb{Z}[I]^2$, $k-2w=mt$, the hypercohomology group $\bb{H}^g\big(\scr{S}^\text{tor}_{K,\text{rig}}(G), j^\dagger\text{DR}^\bullet_c\big(\cal{F}^{(k,w)}\big)\big)$ can be computed either as
\[
 \frac{S^\dagger_{k,w}(K;E_\wp)}{\sum_{\tau\in I}\Theta_{\tau,k_\tau-1}\big(S^\dagger_{s_{\tau}\cdot (k,w)}(K;E_\wp)\big)}\qquad\text{or}\qquad\frac{H^0_\text{rig}\big(\scr{S}^\text{tor}_{K,\text{rig}}(G), j^\dagger\big(\cal{F}^{(k,w)}\otimes\Omega_{\scr{S}^\text{tor}_{K,\text{rig}}(G)}^g\big) \big)}{\nabla H^0_\text{rig}\big(\scr{S}^\text{tor}_{K,\text{rig}}(G), j^\dagger\big(\cal{F}^{(k,w)}\otimes\Omega_{\scr{S}^\text{tor}_{K,\text{rig}}(G)}^{g-1}\big) \big)}.
\]
\end{theorem}
\begin{proof}
The first claim is (\cite{Hilbert} Theorem 3.5). For the second, Theorem $\ref{th: dual BGG quasi-iso}$ states that we have a quasi-isomorphism of complexes $\text{DR}^{\bullet}_c(\cal{F}^{(k,w)})\cong\text{BGG}^{\bullet}_c(\cal{F}^{(k,w)})$ , thus the isomorphisms  
$\bb{H}^g\big(\scr{S}^\text{tor}_{K,\text{rig}}, j^\dagger\text{DR}^\bullet_c\big(\cal{F}^{(k,w)}\big)\big)\cong\bb{H}^g\big(\scr{S}^*_{K,\text{rig}}, j^\dagger\oint_*\text{BGG}_c(\cal{F}^{(k,w)}\big)\big)\cong \bb{H}^g\big(\scr{S}^*_{K,\text{rig}}, j^\dagger\oint_*\text{DR}^\bullet_c\big(\cal{F}^{(k,w)}\big)\big)$ follow by applying the Leray spectral sequence for the composition $\scr{S}^\text{tor}_{K,\text{rig}}\to\scr{S}^*_{K,\text{rig}}\to\text{Spa}\bb{Q}_p$ and the vanishing of the higher derived images of subcanonical automorphic bundles (\cite{HigherLan} Theorem 8.2.1.2). We conclude that
\[
\bb{H}^g\big(\scr{S}^\text{tor}_{K,\text{rig}}, j^\dagger\text{DR}^\bullet_c\big(\cal{F}^{(k,w)}\big)\big)\cong\frac{H^0_\text{rig}\big(\scr{S}^\text{tor}_{K,\text{rig}}(G), j^\dagger\big(\cal{F}^{(k,w)}\otimes\Omega_{\scr{S}^\text{tor}_{K,\text{rig}}(G)}^g\big) \big)}{\nabla H^0_\text{rig}\big(\scr{S}^\text{tor}_{K,\text{rig}}(G), j^\dagger\big(\cal{F}^{(k,w)}\otimes\Omega_{\scr{S}^\text{tor}_{K,\text{rig}}(G)}^{g-1}\big) \big)}
\]
since on the minimal compactification there is a fundamental system of affinoid neighborhoods of the ordinary locus.
\end{proof}
\begin{remark}
	The theorem holds if we replace the group $G$ by $G^*$. Indeed, if we define overconvergent cuspforms for $G^*$ of weight $(k,u)\in\bb{Z}[I]^2$ by $S^{*,\dagger}_{k,u}(K;E_\wp)=H^0\big(\scr{S}^\text{tor}_{K,\text{rig}}(G^*), j^\dagger\big(\underline{\omega}_{G^*}^{(k,u)}(-\text{D})\big)\big)$, then for any cohomological weight $(k,u)\in\bb{Z}[I]^2$, the group $\bb{H}^g\big(\scr{S}^\text{tor}_{K,\text{rig}}(G^*), j^\dagger\text{DR}^\bullet_c\big(\cal{F}^{(k,u)}\big)\big)$ can be computed either as
$
 S^{*,\dagger}_{k,u}(K;E_\wp)/\sum_{\tau\in I}\Theta_{\tau,k_\tau-1}\big(S^{*,\dagger}_{s_{\tau}\cdot (k,u)}(K;E_\wp)\big)$ or $H^0_\text{rig}\big(\scr{S}^\text{tor}_{K,\text{rig}}(G^*), j^\dagger\big(\cal{F}_{G^*}^{(k,u)}\otimes\Omega_{\scr{S}^\text{tor}_{K,\text{rig}}(G^*)}^g\big) \big)/\nabla H^0_\text{rig}\big(\scr{S}^\text{tor}_{K,\text{rig}}(G^*), j^\dagger\big(\cal{F}_{G^*}^{(k,u)}\otimes\Omega_{\scr{S}^\text{tor}_{K,\text{rig}}(G^*)}^{g-1}\big) \big)$.
\end{remark}

\begin{lemma}\label{frobenius in cohomology}
Let $\frak{p}\mid p$ be a prime $\cal{O}_F$-ideal. The partial Frobenius $\text{Fr}_\frak{p}$ (\cite{Hilbert} Section 3.12) acts on the image of $S^{\dagger}_{k,w}(K,E_\wp)$ in  the hypercohomology group $\bb{H}^g\big(\scr{S}^\text{tor}_{K,\text{rig}}(G), j^\dagger\text{DR}^\bullet_c\big(\cal{F}^{(k,w)}\big)\big)$  as $\text{Fr}_\frak{p}=\text{N}_{F/\bb{Q}}(\frak{p})V(\frak{p})$.
\end{lemma}
\begin{proof}
	Taking into account the action of the partial Frobenius on $j^\dagger\Omega^g_{\scr{S}^\text{tor}_{K,\text{rig}}(G)}$, the same computation as in (\cite{ClassicColeman} Remark p.339) shows that $\text{Fr}_\frak{p}$ acts on the image of $S^{\dagger}_{k,w}(K, E_\wp)$ in $\bb{H}^g\big(\scr{S}^\text{tor}_{K,\text{rig}}(G), j^\dagger\text{DR}^\bullet_c\big(\cal{F}^{(k,w)}\big)\big)$ as $\varpi_\frak{p}^{k-t+\frac{(2-m)t-k}{2}}[\frak{p}]$, since $[\frak{p}]$ is the operator that acts on q-expansion by $\msf{a}(y,\msf{f}_{\lvert [\frak{p}]})=\msf{a}(y\varpi_\frak{p}^{-1},\msf{f})$. We conclude noting that $[\frak{p}]=\varpi_\frak{p}^{t-w}V(\frak{p})$ as operators on $S^{\dagger}_{k,w}(K, E_\wp)$.
\end{proof}
If we denote by $U_\frak{p}^\text{geom}$ the operator defined in (\cite{Hilbert} Section 3.18), the equality $U_\frak{p}\text{Fr}_\frak{p}=\langle \frak{p}^{-1}\rangle \varpi_\frak{p}^t$ of (\cite{Hilbert} Lemma 3.20) implies that $U(\frak{p})=U_\frak{p}^\text{geom}\langle \frak{p}\rangle$ as operators on $S^\dagger_{k,w}(K,E_\wp)$. In particular, we can restate (\cite{Hilbert} Corollary 3.24) by saying that if $\msf{f}\in S^\dagger_{s_J\cdot(k,w)}(K,E_\wp)$ is a generalized eigenform for $U_0(\frak{p})$ with non-zero eigenvalue $\lambda_\frak{p}$, then
\begin{equation}\label{bound on slopes}
\text{val}_p(\lambda_\frak{p})\ge\sum_{\tau\in I_{\frak{p}}\setminus J}(k_\tau-1)
\end{equation}
where $I_{\frak{p}}$ is the subset of those embeddings $F\hookrightarrow\overline{\bb{Q}}$ that induce the prime $\frak{p}$ when composed with the fixed $p$-adic embedding $\iota_p:\overline{\bb{Q}} \hookrightarrow\overline{\bb{Q}}_p$.

\begin{corollary}\label{overconvergence in the split case}
	Let $F/\bb{Q}$ be a real quadratic field in which $p\cal{O}_F=\frak{p}_1\frak{p}_2$ splits. Let $\frak{N}$ be an $\cal{O}_F$-ideal prime to $p$ and $\msf{f}\in S_{k,w}(\frak{N},\overline{\bb{Q}})$ a $p$-nearly ordinary eigenform. Then the $p$-adic cuspforms $d_1^{1-k_1}(\msf{f}^{[\frak{p}_1, \frak{p}_2]})$ and $d_2^{1-k_2}(\msf{f}^{[\frak{p}_1, \frak{p}_2]})$ are overconvergent.
\end{corollary}
\begin{proof}
We prove the corollary building on an idea of Loeffler, Skinner and Zerbes (\cite{LoefflerSkinnerZerbes}, Proposition 4.5.3). Let $1-\msf{a}(\varpi_{\frak{p}_2},\msf{f})X+\epsilon_{\msf{f}}(\frak{p}_2)\varpi_{\frak{p}_2}^{(m+1)t-2v}X^2=(1-\alpha_{0,2}X)(1-\beta_{0,2}X)$  be the Hecke polynomial of $\msf{f}$ for $T_0(\frak{p}_2)$. By assumption $\msf{f}$ is nearly ordinary at $\frak{p}_2$, hence we can assume without loss of generality that $\text{val}_p(\alpha_{0,2})=0$ and $\text{val}_p(\beta_{0,2})=\text{val}_p(\varpi_{\frak{p}_2}^{k-t})$. We denote by $\msf{f}_{\alpha_2},\msf{f}_{\beta_2}$ the two $\frak{p}_2$-stabilizations of $\msf{f}$. If we write $\Theta_i=\Theta_{\tau_i,k_{\tau_i}-1}$ for $i=1,2$, then the classes of $\msf{f}_{\alpha_2}^{[\frak{p}_1]},\msf{f}_{\beta_2}^{[\frak{p}_1]}$ are trivial in the quotient $\frac{S^\dagger_{k,w}(K;E_\wp)}{\text{Im}(\Theta_{1})+\text{Im}(\Theta_{2})}$ because they are annihilated by the invertible operator $U_0(\frak{p}_1)$. Consider the Hecke-equivariant projections $\text{pr}_i: \text{Im}(\Theta_{1})+\text{Im}(\Theta_{2})\to\frac{\text{Im}(\Theta_{i})}{\text{Im}(\Theta_{1})\cap\text{Im}(\Theta_{2})}$ for $i=1,2$. We immediately see that $\text{pr}_2(\msf{f}_{\alpha_2}^{[\frak{p}_1]})=0$ because of the lower bound $(\ref{bound on slopes})$ on the slopes of $U_0(\frak{p}_2)$, therefore $\text{pr}_2(\msf{f}^{[\frak{p}_1]})=\frac{\beta_2}{\beta_2-\alpha_2}\text{pr}_2(\msf{f}^{[\frak{p}_1]}_{\beta_2})$ which implies $U_0(\frak{p}_2)\text{pr}_2(\msf{f}^{[\frak{p}_1]})=\beta_{0,2}\cdot\text{pr}_2(\msf{f}^{[\frak{p}_1]})$. We claim that $[\frak{p}_2]\text{pr}_2(\msf{f}^{[\frak{p}_1]})=\frac{1}{\beta_{0,2}}\text{pr}_2(\msf{f}^{[\frak{p}_1]})$. Indeed, the equality of Hecke operators $T_0(\frak{p}_2)=U_0(\frak{p}_2)+\alpha_{0,2}\beta_{0,2}[\frak{p}_2]$ allows us to compute that
	\[\begin{split}
		[\frak{p}_2]\text{pr}_2(\msf{f}^{[\frak{p}_1]})&=\frac{1}{\alpha_{0,2}\beta_{0,2}}\big[T_0(\frak{p}_2)\text{pr}_2(\msf{f}^{[\frak{p}_1]})-U_0(\frak{p}_2)\frac{\beta_2}{\beta_2-\alpha_2}\text{pr}_2(\msf{f}^{[\frak{p}_1]}_{\beta_2})\big]\\
		&=\frac{1}{\alpha_{0,2}\beta_{0,2}}\big[\msf{a}(\frak{p}_2,\msf{f})\text{pr}_2(\msf{f}^{[\frak{p}_1]})-\beta_{0,2}\text{pr}_2(\msf{f}^{[\frak{p}_1]})\big]=\frac{1}{\beta_{0,2}}\text{pr}_2(\msf{f}^{[\frak{p}_1]}).
	\end{split}
	\]
	Thus it is clear that $\text{pr}_2(\msf{f}^{[\frak{p}_1, \frak{p}_2]})=0$. By exchanging the roles of the two primes $\frak{p}_1,\frak{p}_2$ we also have that $\text{pr}_1(\msf{f}^{[\frak{p}_1, \frak{p}_2]})=0$, which proves $\msf{f}^{[\frak{p}_1, \frak{p}_2]}\in \text{Im}(\Theta_{1})\cap\text{Im}(\Theta_{2})$.
\end{proof}

\section{A $p$-adic Gross-Zagier Formula}
\subsection{De Rham realization of modular forms}\label{realization of modular forms}
Let $E$ be a number field, following Voevodsky \cite{Voe} we consider two categories of motives over $E$: the category of effective Chow motives denoted $\text{CHM}^\text{eff}$ with a natural functor $\text{h}:\text{SmProj}_{/E}\to\text{CHM}^\text{eff}$
 from the category $\text{SmProj}_{/E}$ of smooth and projective schemes over $E$, and the triangulated category $\text{DM}^\text{eff}$ of effective geometric motives with the natural functor $\text{M}_\text{gm}:\text{Sm}_{/E}\to\text{DM}^\text{eff}$ from the category $\text{Sm}_{/E}$ of smooth schemes over $E$. Since number fields have characteristic zero, these two categories are related by a full embedding $\text{CHM}^\text{eff}\to\text{DM}^\text{eff}$ such that makes the diagram
\[\xymatrix{
\text{SmProj}_{/\bb{Q}}\ar[r]\ar[d]_{\text{h}} & \text{Sm}_{/\bb{Q}}\ar[d]^{\text{M}_\text{gm}}\\
\text{CHM}^\text{eff}\ar[r] & \text{DM}^\text{eff}
}\]
commutes (\cite{Voe} Proposition 2.1.4 and Remark). 

Let $F$ be a totally real number field of degree $g$ over $\bb{Q}$ and let $E$ be any field containing $F^\text{Gal}$. The Shimura variety $\text{Sh}_K(G^*)_\bb{Q}$ has a universal  Hilbert-Blumenthal abelian scheme $\cal{A}\to \text{Sh}_K(G^*)$, the $\cal{O}_F$-action induces a ring homomorphism $F\hookrightarrow \text{End}_{\text{Sh}_K(G^*)}(\cal{A})\otimes_\bb{Z}\bb{Q}$. We denote by $\text{CMH}(\text{Sh}_K(G^*))$ the category of Chow motives over $\text{Sh}_K(G^*)$ \cite{DM}. Since the decomposition of the Chow motive $\text{h}(\cal{A}/\text{Sh}_K(G^*))=\bigoplus_i\text{h}_i(\cal{A}/\text{Sh}_K(G^*))$ of $\cal{A}$ over $\text{Sh}_K(G^*)$ is functorial (\cite{DM} Theorem 3.1), there is an isomorphism of $\bb{Q}$-vector spaces (\cite{K} Proposition 2.2.1)
\[
\text{End}_{\text{Sh}_K(G^*)}(\cal{A})\otimes_\bb{Z}\bb{Q}\overset{\sim}{\longrightarrow}\text{End}_{\text{CHM}(\text{Sh}_K(G^*))}\left(\text{h}_1(\cal{A}/\text{Sh}_K(G^*))\right)\otimes_\bb{Z}\bb{Q}.
\]
One denotes by $e_\tau\in \text{End}_{\text{CHM}(\text{Sh}_K(G^*))}\left(\text{h}_1(\cal{A}/\text{Sh}_K(G^*))\right)\otimes_\bb{Z}E$, $\tau\in I_F$,  the idempotents coming from  $\prod_\tau F=F\otimes E\hookrightarrow\text{End}_{\text{CHM}(\text{Sh}_K(G^*))}\left(\text{h}_1(\cal{A}/\text{Sh}_K(G^*))\right)\otimes_\bb{Z}E$.

\begin{definition}
Let $k\in \bb{N}[I_F]$, $k\ge 2t_F$. The relative motive $\cal{V}^k\in\text{CHM}(\text{Sh}_K(G^*))_E$ is defined as
\[
\cal{V}^k=\underset{\tau\in I_F}{\bigotimes}\text{Sym}^{k_\tau-2}\text{h}_1(\cal{A}/\text{Sh}_K(G^*))^{e_\tau}
\]
following the conventions of (\cite{K} p.72) for the symmetric products. The motive $\cal{V}^k$ is a direct factor of $\text{h}(\cal{A}^{k-2g}/\text{Sh}_K(G^*))$, where $\cal{A}^{k-2g}$ denotes the $(\lvert k\rvert-2g)$-fold fiber product of $\cal{A}$ over $\text{Sh}_K(G^*)$, thus it corresponds to an idempotent $e_k\in\text{CH}^{g(\lvert k\rvert-2g)}\left(\cal{A}^{k-2g}\times_{\text{Sh}_K(G^*)}\cal{A}^{k-2g}\right)\otimes_\bb{Z}E$ such that $\text{M}_\text{gm}(\cal{A}^{k-2g})^{e_k}=\cal{V}^k$.
\end{definition}


\begin{proposition}{(\cite{Wieldeshaus} Corollary 3.9)}\label{comp}
Suppose $k>2t_F$ and let $U_{k-2g}$ be any smooth compactification of $\cal{A}^{k-2g}$, then the graded part of weight zero with respect to the motivic weight structure on $\text{CHM}^\text{eff}_{E}$, $\text{Gr}_0\text{M}_{\text{gm}}\left(\cal{A}^{k-2g}\right)^{e_{k}}$,  is canonically a direct factor of the Chow motive $\text{M}_\text{gm}(U_{k-2g})$. Hence, it corresponds to an idempotent $\theta_k\in\text{CH}^{g(\lvert k\rvert-2g+1)}(U_{k-2g}\times_\bb{Q} U_{k-2g})\otimes_\bb{Z}E$.
\end{proposition}

\begin{proposition}\label{prop: elliptic comp}
Suppose $F=\bb{Q}$ and let $k>2$ be an integer. For any smooth compactification $W_{k-2}$ of the $(k-2)$th-fold product of the universal elliptic curve $\cal{E}$ over the modular curve $\text{Sh}_K(\text{GL}_{2,\bb{Q}})$, there exists an idempotent $\theta_k\in\text{CH}^{k-1}(W_{k-2}\times_\bb{Q}W_{k-2})\otimes_\bb{Z}\bb{Q}$ such that $\theta_k^* H_\text{dR}^*(W_{k-2}/\bb{Q})=\theta_k^*H_\text{dR}^{k-1}(W_{k-2}/\bb{Q})$ is functorially isomorphic to $H^1_\text{par}(\text{Sh}^\text{tor}_K(\text{GL}_{2,\bb{Q}}),\cal{F}^{(k,k-1)},\nabla)$ with its Hodge filtration.
\end{proposition}
\begin{proof}
Proposition $\ref{comp}$ provides an idempotent $\theta_k$ such that $\theta_k^*\text{M}_\text{gm}(W_{k-2})=\text{Gr}_0\text{M}_\text{gm}(\cal{E}^{k-2})^{e_k}$. We claim that the proof of (\cite{BDP} Lemma 2.2) applies to our situation. Indeed, the main ingredient of that proof is a result of Scholl (\cite{Scholl} Theorem 3.1.0), which can be applied to any smooth compactification $W_{k-2}$ since the motive considered by Scholl is isomorphic to $\text{Gr}_0\text{M}_\text{gm}(\cal{E}^{k-2})^{e_k}$ by (\cite{Chownonprojective} Corollary 3.4(b)). Note that the idempotent $e$ in (\cite{Chownonprojective} Definition 3.1) acts as the idempotent $e_k$ on $\text{M}_\text{gm}(\cal{E}^{k-2})$ because the action of the torsion appearing in $e$ is trivial since $\cal{E}^{k-2}\to\text{Sh}_K(\text{GL}_{2,\bb{Q}})$ is an abelian scheme.
\end{proof}

\begin{proposition}\label{prop: Hilbert comp}
Let $L/\bb{Q}$ be a real quadratic extension and $\ell\in\bb{N}[I_F],\ \ell>2t_F$ a non-parallel weight. For any smooth compactification $U_{\ell-4}$ of the $(\lvert\ell\rvert-4)$th-fold product of the universal abelian surface over $\text{Sh}_K(G_L^*)$, there exists an idempotent $\theta_\ell\in\text{CH}^{2(\lvert\ell\rvert-3)}(U_{\ell-4}\times_\bb{Q}U_{\ell-4})\otimes_\bb{Z}L$ such that $\theta_\ell^*H_\text{dR}^{i+\lvert\ell\rvert-4}(U_{\ell-4}/\bb{Q})$ is functorially isomorphic to $\bb{H}^i(\text{Sh}_K(G_L^*),\text{DR}^{\bullet}(\cal{F}_{G^*_L}^{(\ell,\ell-t)}))$ with its Hodge filtration.
\end{proposition}
\begin{proof}
Since the weight $\ell$ is not parallel, Proposition $\ref{comp}$ and (\cite{Wieldeshaus} Corollary 3.7) provide an idempotent $\theta_\ell$ such that $\theta_k^*\text{M}_\text{gm}(U_{k-4})=\cal{V}^\ell$ . Then Kings proved in (\cite{K} Corollary 2.3.4) that the $(i+\lvert\ell\rvert-4)$-th cohomology of the de Rham realization of $\cal{V}^\ell$ is isomorphic to $\bb{H}^i(\text{Sh}_K(G_L^*),\text{DR}^{\bullet}(\cal{F}_{G^*_L}^{(\ell,\ell-t)}))$.
\end{proof}

\subsection{Generalized Hirzebruch-Zagier cycles}\label{Generalized Hirzebruch-Zagier cycles}
Let $L/\bb{Q}$ be a real quadratic extension and denote by $\xi:\text{Sh}_K(G_L^*)\to\text{Sh}_K(G_L)$ the map of Shimura varieties derived from the inclusion $G_L^*\hookrightarrow G_L$.
Let $\breve{\msf{g}}\in S_{\ell,x}(M,L;\overline{\bb{Q}})$ be a eigenform of either parallel weight $\ell=2t_L$ or non-parallel weight $\ell>2t_L$ such that $\ell-2x=nt_L$. Let $\breve{\msf{f}}\in S_{k,w}(M;\overline{\bb{Q}})$ be an elliptic eigenform for the good Hecke operators, such that $k-2w=m$, and we denote by $\msf{f}$ the newform corrsponding to the system of eigenvalues. We suppose that the weights of $\msf{g}$ and $\msf{f}$ are balanced, which implies the equality $2n=m$. We consider $E/\bb{Q}$ a finite Galois extension containing the Fourier coefficients of $\msf{g}$ and $\msf{f}$. 

We want to realize these modular forms in the de Rham cohomology of some proper and smooth variety. The pullback $\xi^*\breve{\msf{g}}$ lives in $S_{\ell, x}^*(M,L;E)$, which by Proposition $\ref{prop: change central character}$ is isomorphic to $S_{\ell,\ell-t_L}^*(M,L;E)$.  Thanks to Theorem $\ref{th: dual BGG quasi-iso}$  we can realize the latter space as a subgroup of the hypercohomology group $\bb{H}^2(\text{Sh}_K^\text{tor}(G_L^*)_E,\text{DR}_c^{\bfcdot}(\cal{F}^{\ell,\ell-t}))$, which is simply the de Rham cohomology group $H^2_\text{dR}(\text{Sh}_K^\text{tor}(G_L^*)/E)$ when $\ell=2t_L$. Instead, when $\ell>2t_L$ is not parallel, let $U_{\ell-4}$ be any smooth compactification of $\cal{A}^{\ell-4}$; then, we can invoke Proposition $\ref{prop: Hilbert comp}$ to establish that the differential attached to $\Psi_{x,\ell-t_L}(\xi^*\breve{\msf{g}})$ lives in $\text{F}^{\lvert\ell\rvert-2}H^{\lvert\ell\rvert-2}_\text{dR}(U_{\ell-4}/E)$. Similarly, if $k=2$, $\Psi_{w,1}(\breve{\msf{f}})\in S_{2,1}(M;E)\cong\text{F}^1 H^1_\text{dR}(\mathbf{Sh}_K^\text{tor}(\text{GL}_2)_E)$, while when $k>2$ we can consider any smooth compactification $W_{k-2}$ of $\cal{E}^{k-2}$ to see that the class of the differential $\omega_{\Psi_{w,k-1}(\breve{\msf{f}})}$ lives in $H_\text{dR}^{k-1}(W_{k-2}/E)$, by Proposition $\ref{prop: elliptic comp}$.
\begin{definition}\label{definition of classes}
Choose a prime $p$ coprime to $M$, let $E_\wp$ be the closure of $\iota_p(E)$ in $\overline{\bb{Q}}_p$ and suppose that $\breve{\msf{g}}, \breve{\msf{f}}$ are $p$-nearly ordinary. We write $\omega$ for the differential $\omega_{\Psi_{x,\ell-t_L}(\xi^*\breve{\msf{g}})}$ and we take $\eta$ to be the class in the $\Psi_{w,k-1}(\msf{f})$-isotypic part of $H^1_\text{par}(\text{Sh}_K^\text{tor}(\text{GL}_{2,\bb{Q}})_{E_\wp},\cal{F}^{k,k-1},\nabla)^\text{u.r.}$ whose image in the $0$-th graded piece, $H^1(\text{Sh}_K^\text{tor}(\text{GL}_{2,\bb{Q}})_{E_\wp},\underline{\omega}^{2-k})$, is equal to the image of $\frac{1}{\langle\msf{f}^*,\msf{f}^*\rangle}\overline{\omega_{\Psi_{w,k-1}(\breve{\msf{f}}^*)}}$. \end{definition}
The class $\eta\in H^1_\text{par}(\text{Sh}_K^\text{tor}(\text{GL}_{2,\bb{Q}})_{E_\wp},\cal{F}^{k,k-1},\nabla)^\text{u.r.}$ satisfies
\begin{equation}\label{computation of the eigenvalue}
		\text{Fr}_p(\eta)=\alpha_{\msf{f}^*}p^{w-1}\eta,
\end{equation} where the eigenvalue is a p-adic unit since $\msf{f}^*$ is $p$-nearly ordinary. Indeed, by definition $\eta=[c\cdot\Psi_{w,k-1}(\breve{\msf{f}}_\beta)]$ for some non-zero constant $c$, and applying Lemmas $\ref{lemma: central character and V(p)}$ and $\ref{frobenius in cohomology}$ we can compute
	\[\begin{split}
	\text{Fr}_p(\eta)&=pV(p)[c\cdot\Psi_{w,k-1}(\breve{\msf{f}}_{\beta})]=p\cdot p^{k-1-w}[c\cdot\Psi_{w,k-1}(V(p)\breve{\msf{f}}_{\beta})]\\
	&=p^{k-w}[c\cdot\Psi_{w,k-1}(U(p)^{-1}\breve{\msf{f}}_{\beta})]=p^{k-w}\beta_{\msf{f}}^{-1}\eta=\alpha_{\msf{f}^*}p^{w-1}\eta,
	\end{split}\]
	since $\beta_{\msf{f}}^{-1}=\alpha_{\msf{f}}\psi_{\msf{f}}(p)^{-1}p^{-1}=\alpha_{\msf{f}^*}p^{-m-1}$.

For all $s\ge0$ we want to consider the cohomology class
\[
\pi_1^*\omega\cup\pi_2^*\eta\in \text{F}^{\lvert\ell\rvert-2-s} H_\text{dR}^{\lvert\ell\rvert+k-3}\big(U_{\ell-4}\times_{E_\wp}W_{k-2}\big).
\]
Our goal is to define a null-homologous cycle on $U_{\ell-4}\times_EW_{k-2}$ whose syntomic Abel-Jacobi map can be evaluated at $\pi_1^*\omega\cup\pi_2^*\eta$. 
Let $\scr{Z}_{\ell,k}$ be a proper smooth model of $U_{\ell-4}\times_{E_\wp}W_{k-2}$ over $\cal{O}_{E_\wp}$ of relative dimension $d$, and denote by $Z_{\ell,k}$ its generic fiber. For all integers $i\ge0$, the syntomic cohomology groups of $\scr{Z}_{\ell,k}$ sit in a short exact sequence of the form
\[\xymatrix{
0\ar[r]& H^{2i-1}_\text{dR}(Z_{\ell,k})/F^i\ar[r]^i & H^{2i}_\text{syn}(\scr{Z}_{\ell,k},i)\ar[r]^p&\text{F}^iH^{2i}_\text{dR}(Z_{\ell,k})\ar[r]&0.
}\]
The syntomic cycle class map (\cite{Bes00} Proposition 5.4) is compatible with the de Rham cycle class map producing a commuting diagram 
\[\xymatrix{
\text{CH}^i(\scr{Z}_{\ell,k})\ar[r]^{\text{cl}_\text{syn}}\ar[d]^{\text{Res}}& H^{2i}_\text{syn}(\scr{Z}_{\ell,k},i)\ar[d]^p\\
\text{CH}^i(Z_{\ell,k})\ar[r]^{\text{cl}_\text{dR}}& F^iH^{2i}_\text{dR}(Z_{\ell,k}),
}\]
where on the left hand side are the Chow group of algebraic cycles modulo rational equivalence. The restriction of the syntomic cycle class map $\text{cl}_\text{syn}$ to the subgroup of de Rham null-homologous cycles $\text{CH}^i(\scr{Z}_{\ell,k})_0$, i.e., the kernel of the composition $\text{cl}_\text{dR}\circ\text{Res}$, has image landing in $H^{2i-1}_\text{dR}(Z_{\ell,k})/F^i$. The syntomic Abel-Jacobi map 
\begin{equation}\label{eq: Abel-Jacobi}
\text{AJ}_p: \text{CH}^i(\scr{Z}_{\ell,k})_0\longrightarrow \left( F^{d-i+1}H_{dR}^{2(d-i)+1}(Z_{\ell,k})\right)^\vee
\end{equation}
is obtained by identifying the target using Poincar\'e duality.

We determine the positive integer $s$ and make sure the numerology works. The dimension of the variety $U_{\ell-4}\times_EW_{k-2}$ is $d=2\lvert\ell\rvert+k-7$, therefore the cycle we want has to be of dimension $d-i$ such that $2(d-i)+1=\lvert\ell\rvert+k-3$ and $s\ge0$ has to satisfy $\lvert\ell\rvert-2-s=(d-i)+1$. Hence 
\begin{equation}\label{eq: numerology}
(d-i)=\frac{\lvert\ell\rvert+k-4}{2},\qquad s=\frac{\lvert\ell\rvert-k-2}{2}
\end{equation}
with $s\ge0$ since the weights are balanced. 

\subsubsection{Definition of the cycles.}
We treat separately the case $(\ell,k)=(2t_L,2)$ and the general case $(\ell,k)>(2t_L,2)$ with $\ell$ not parallel. Set $r+1=\frac{\lvert\ell\rvert+k-4}{2}$ and consider the closed embedding
\[
\varphi:\cal{E}^r\longrightarrow\cal{A}^{\ell-4}\times_E\cal{E}^{k-2},\qquad (x,P_1,\dots,P_r)\mapsto(\zeta(x),P_{1}'\otimes1,\dots,P_{\lvert\ell\rvert-4}'\otimes1; x,P_{\lvert\ell\rvert-3}',\dots,P_{2r}')
\] where $(P_1',\dots,P_{2r}')=(P_1,\dots,P_r,P_1,\dots,P_r)$ and $P_i'\otimes 1$ is the point  $P_i'\otimes 1\to\cal{E}\otimes_\bb{Z}\cal{O}_F\to\cal{A}$. The definition makes sense because $2r=\lvert\ell\rvert-4+k-2$. The variety $\cal{E}^r$ has dimension equal  to $r+1$ and we will define the null-homologous cycle by first compactifying and then by applying an appropriate correspondence.
Take smooth and projective compactifications $W_r, U_{\ell-4}, W_{k-2}$ of $\cal{E}^r,\cal{A}^{\ell-4},\cal{E}^{k-2}$ respectively, the map $\varphi$ defines a rational morphism $\xymatrix{\varphi: W_r \ar@{.>}[r]& U_{\ell-4}\times_E W_{k-2}}$.
Using Hironaka's work on resolution of singularities (\cite{Hironaka} Chapter 0.5, Question (E)), we can assume the rational map has a representative $\varphi: W_r\longrightarrow U_{\ell-4}\times_E W_{k-2}$ up to replacing the smooth and projective compatification of $\cal{E}^r$. By spreading out, there is an open of $\text{Spec}(\cal{O}_E)$ over which all our geometric objects can be defined simultaneously and retain their relevant features: we have smooth and projective models $\scr{W}_r, \scr{U}_{\ell-4}, \scr{W}_{k-2}$ of $W_r, U_{\ell-4}, W_{k-2}$ respectively and the map $\varphi$ extends to a map $\tilde{\varphi}:\scr{W}_r\longrightarrow \scr{U}_{\ell-4}\times\scr{W}_{k-2}$. 

When $\ell=2t_L$ and $k=2$ we define correspondences on $\scr{U}_0\times\scr{W}_0$ as follows. We assume the number field $E$ is large enough such that $U_{0/E}$ (resp. $W_{0/E}$) is the disjoint union of its geometrically connected components $U_{0/E}=\coprod_iU_{0,i}$ (resp. $W_{0/E}=\coprod_jW_{0,j}$) and we pick an $E$-rational point $a_i\in U_{0,i}$ (resp. $b_j\in W_{0,j}$) for every such component. We define correspondences on $Z=U_0\times_E W_0$ by $P_{i,j}=\text{graph}(q_{i,j}:Z\to U_{0,i}\times_E W_{0,j}\to Z)$, $P_{a_i,j}=\text{graph}(q_{a_i,j}:Z\to W_{0,j}\to Z)$, $P_{i,b_j}=\text{graph}(q_{i,b_j}:Z\to U_{0,i}\to Z)$, $P_{a_i,b_j}=\text{graph}(q_{a_i,b_j}:Z\to \{a_i\}\times\{b_j\}\to Z)$, all elements of $\text{CH}^6(Z\times_E Z)$. We set 
\[
P=\sum_{i,j}\left(P_{i,j}-P_{a_i,j}-P_{i,b_j}+P_{a_i,b_j}\right),
\]
that acts on $\text{CH}^{\bfcdot}(Z)$ by $P_*=\text{pr}_{2,*}(P\cdot\text{pr}_{1}^*)$; in particular, for any cycle $S\in\text{CH}^{\bfcdot}(Z)$, we have \[
P_*(S)=\sum_{i,j}\left[(q_{i,j})_*-(q_{a_i,j})_*-(q_{i,b_j})_*+(q_{a_i,b_j})_*\right](S).
\]
For $i,j$ running in the set of indices of the geometrically connected components of $U_0$ and $W_0$ the correspondences $(P_{i,j}-P_{a_i,j}-P_{i,b_j}+P_{a_i,b_j})$ are idempotents and orthogonal to each other, hence $P\circ P=P$ in $\text{CH}^6(Z\times_E Z)$, i.e., $P$ is a projector. We denote by $\tilde{P}$ the correspondence on $\scr{U}_0\times\scr{W}_0$ defined over some open of $\text{Spec}(\cal{O}_E)$ obtained by spreading out $P$.

When $(\ell,k)>(2t_L,2)$ with $\ell$ non-parallel, we obtain a correspondence on $\scr{U}_{\ell-4}\times\scr{W}_{k-2}$ by spreading out those correspondences considered in Section $\ref{realization of modular forms}$. Indeed, the idempotents $\theta_\ell\in\text{CH}^{2(\lvert \ell\rvert-3)}(U_{\ell-4}\times_EU_{\ell-4})\otimes_\bb{Z}L$ and $\theta_k\in\text{CH}^{\lvert k\rvert-1}(W_{k-2}\times_EW_{k-2})\otimes_\bb{Z}\bb{Q}$ extend to elements $\tilde{\theta}_\ell\in\text{CH}^{2(\lvert \ell\rvert-3)}(\scr{U}_{\ell-4}\times\scr{U}_{\ell-4})\otimes_\bb{Z}L$ and $\tilde{\theta}_k\in\text{CH}^{\lvert k\rvert-1}(\scr{W}_{k-2}\times\scr{W}_{k-2})\otimes_\bb{Z}\bb{Q}$ respectively.

\begin{definition}
	For all but finitely many primes $p$, we define the Hirzebruch-Zagier cycle of weight $(2t_L,2)$ to be 
	\[
	\Delta_{2t_L,2}=\tilde{P}_*\tilde{\varphi}_*[\scr{W}_0]\in\text{CH}^2(\scr{U}_{0}\times_{\cal{O}_{E,\wp}}\scr{W}_{0}).
	\]
\end{definition}

\begin{proposition}
	The Hirzebruch-Zagier cycle $\Delta_{2t_L,2}\in \text{CH}^2(\scr{U}_{0}\times_{\cal{O}_{E,\wp}}\scr{W}_{0})
	$ is de Rham null-homologous.
\end{proposition}
\begin{proof}
	To verify that $\text{cl}_{\text{dR}}(\Delta_{2t_L,2})$ is zero in $H_\text{dR}^4(Z/E_\wp)$, it suffices to show that $P_*H_\text{dR}^4(Z/E)=0$ since our cycle starts his life over $E$. After base-change to $\bb{C}$, via the fixed complex embedding $\iota_\infty:\overline{\bb{Q}}\hookrightarrow\bb{C}$, Poincar\'e duality tells us that it is enough to prove the projector annihilates the second singular homology, i.e., $P_*H_2(Z(\bb{C}))=0$. By Kunneth formula and the fact that each connected component of $U_0(\bb{C})$ is simply connected, we compute that $P_*H_2(Z(\bb{C}))=P_*(H_0(U_0(\bb{C}))\otimes H_2(W_0(\bb{C}))\oplus H_2(U_0(\bb{C}))\otimes H_0(W_0(\bb{C})))$, which we can show it is zero by the explicit definition of the projector $P$. Indeed, let $[x]\otimes [C]\in H_0(U_0(\bb{C}))\otimes H_2(W_0(\bb{C}))$ be a simple tensor for $x\in U_0(\bb{C})$ a point, then for all $i,j$ we find
	\[\begin{split}
	\left(P_{i,j}-P_{a_i,j}-P_{i,b_j}+P_{a_i,b_j}\right)([x]\otimes [C])&=\left((q_{i,j})_*-(q_{a_i,j})_*-(q_{i,b_j})_*+(q_{a_i,b_j})_*\right)([x]\otimes [C])\\
		&=[a_i]\otimes[C_j]-[a_i]\otimes[C_j]=0,
	\end{split}
	\]
	where $(q_{i,b_j})_*([x]\otimes [C])=0=(q_{a_i,b_j})_*([x]\otimes [C])$ because the dimension of the pushforward drops. Similarly, if $[D]\otimes [y]\in H_2(U_0(\bb{C}))\otimes H_0(W_0(\bb{C}))$ is a simple tensor for $y\in W_0(\bb{C})$ a point, then $\left(P_{i,j}-P_{a_i,j}-P_{i,b_j}+P_{a_i,b_j}\right)([D]\otimes [y])=0$ for all $i,j$.
\end{proof}

\begin{definition}
Let $\ell\in\bb{Z}[I_L]$, $\ell>2t_L$, be a non-parallel weight and $k>2$ an integer such that $(\ell,k)$ is a balanced triple. For all but finitely many primes $p$, the generalized Hirzebruch-Zagier cycle of weight $(\ell,k)$ is 
\[
\Delta_{\ell,k}=(\tilde{\theta}_\ell,\tilde{\theta}_k)_*\tilde{\varphi}_*[\scr{W}_r]\in \text{CH}^i(\scr{U}_{\ell-4}\times_{\cal{O}_{E,\wp}}\scr{W}_{k-2})\otimes_\bb{Z}L.
\]
\end{definition}


\begin{proposition}
Let $\ell\in\bb{Z}[I_L]$, $\ell>2t_L$, be a non-parallel weight and $k>2$ an integer such that $(\ell,k)$ is a balanced triple. The generalized Hirzebruch-Zagier cycle $\Delta_{\ell,k}\in \text{CH}^i(\scr{U}_{\ell-4}\times_{\cal{O}_{E,\wp}}\scr{W}_{k-2})\otimes_\bb{Z}L
$ is de Rham null-homologous. 
\end{proposition}
\begin{proof}
The class $\text{cl}_\text{dR}(\Delta_{\ell,k})$ belongs to $(\theta_\ell,\theta_k)_*H_{\text{dR}}^{2i}(U_{\ell-4}\times_{E_\wp}W_{k-2})$ and by Poincar\'e duality, it is trivial if and only if
\begin{equation}\label{eq: vanishing}
(\theta_\ell,\theta_k)^*H_{\text{dR}}^{2(d-i)}(U_{\ell-4}\times_{E_\wp}W_{k-2})=\bigoplus_{p+q=2(d-i)}(\theta_\ell)^*H_{\text{dR}}^{p}(U_{\ell-4})\otimes(\theta_k)^*H_{\text{dR}}^{q}(W_{k-2})
\end{equation}
is trivial. By Propositions $\ref{prop: Hilbert comp}$ and $\ref{prop: elliptic comp}$, we have $\theta_\ell^*H_{\text{dR}}^{p}(U_{\ell-4})=\bb{H}^{p-\lvert\ell\rvert+4}(\text{Sh}_K(G_L^*),\text{DR}^{\bullet}(\cal{F}_{G^*_L}^{(\ell,\ell-t)}))$ and $\theta_k^* H_\text{dR}^*(W_{k-2})=\theta_k^*H_\text{dR}^{k-1}(W_{k-2})=H^1_\text{par}(\text{Sh}^\text{tor}_K(\text{GL}_{2,\bb{Q}}),\cal{F}^{(k,k-1)},\nabla)$. Hence, $q=k-1$ forces $p$ to be $p=\lvert\ell\rvert-3$ and the group \begin{equation}\label{eq: H1 vanishing}
\theta_\ell^*H_{\text{dR}}^{\lvert\ell\rvert-3}(U_{\ell-4})=\bb{H}^{1}(\text{Sh}_K(G_L^*),\text{DR}^{\bullet}(\cal{F}_{G^*_L}^{(\ell,\ell-t)}))
\end{equation}
is trivial. Indeed, by (\cite{ES-Nekovar} A6.20), the cohomology group $\bb{H}^{1}(\text{Sh}_K(G_L^*),\text{DR}^{\bullet}(\cal{F}_{G^*_L}^{(\ell,\ell-t)}))$ is identified with the intersection cohomology of the Baily-Borel compactification of $\text{Sh}_K(G_L^*)$, that in turn is trivial in degree 1 by computations using Lie algebra cohomology (\cite{ES-Nekovar} Sections 5.11, 6.5, 6.6).
\end{proof}

\subsubsection{Evaluation of syntomic Abel-Jacobi.}
Let $p$ be a prime splitting in $L/\bb{Q}$, $p\cal{O}_L=\frak{p}_1\frak{p}_2$. We are interested in computing $\text{AJ}_p(\Delta_{\ell,k})(\pi_1^*\omega\cup\pi_2^*\eta)$ and to relate it to some value of the twisted triple product $p$-adic $L$-function outside the range of interpolation. Let $\tilde{\omega}$ (resp. $\tilde{\eta}$) be a lift of $\omega$ (resp. $\eta$) to $\text{fp}$-cohomology,
since the Hirzebruch-Zagier cycle is null-homologous the computation is independent of the choice of lifts. We start by treating the case $(\ell,k)=(2t_L,2)$: 
\[\begin{split}
\text{AJ}_p(\Delta_{2t_L,2})(\pi_1^*\omega\cup\pi_2^*\eta)&=\langle \text{cl}_\text{syn}(\Delta_{2t_L,2}), \pi_1^*\tilde{\omega}\cup\pi_2^*\tilde{\eta}\rangle_\text{fp}=\langle \tilde{P}_*\text{cl}_\text{syn}(\tilde{\varphi}_*[\scr{W}_0]), \pi_1^*\tilde{\omega}\cup\pi_2^*\tilde{\eta}\rangle_\text{fp}\\
	&=\langle \text{cl}_\text{syn}(\tilde{\varphi}_*[\scr{W}_0]),\sum_{i,j}(\tilde{P}_{i,j}-\tilde{P}_{a_i,j}-\tilde{P}_{i,b_j}+\tilde{P}_{a_i,b_j})^*(\pi_1^* \tilde{\omega}\cup\pi_2^*\tilde{\eta})\rangle_\text{fp}\\
	&=\langle \text{cl}_\text{syn}(\tilde{\varphi}_*[\scr{W}_0]),\pi_1^*\tilde{\omega}\cup\pi_2^*\tilde{\eta}\rangle_\text{fp}\\
(\cite{Bes00}\ \text{Equation}\ (20))\qquad	&=\text{tr}_{\scr{W}_0}(\tilde{\varphi}^*(\pi_1^*\tilde{\omega}\cup\pi_2^*\tilde{\eta}))=\text{tr}_{\scr{W}_0}(\tilde{\zeta}^*\tilde{\omega}\cup\tilde{\eta}).
\end{split}
\]
The fourth equality is justified by the vanishing  $H_\text{fp}^1(\text{Spec}(\cal{O}_{E,\wp}),0)=0=H_\text{fp}^2(\text{Spec}(\cal{O}_{E,\wp}),2)$, which imply that $\sum_{i,j}\tilde{P}_{i,j}^*=(\text{id}_{\scr{U}_0\times\scr{W}_0})^*$ and that all the other pullbacks are zero. 

To deal with the general case, we first need to analyze the action of the correspondences $\tilde{\theta}_k,\tilde{\theta}_\ell$ on $\text{fp}$-cohomology. The exact sequence in (\cite{Bes00} (8)) induces a functorial isomorphism $H_\text{fp}^{k-1}(\scr{W}_{k-2},0)\cong H^{k-1}_\text{dR}(W_{k-2})$, we denote by $\tilde{\eta}$ the preimage of $\eta\in \theta_k^*H^{k-1}_\text{dR}(W_{k-2})$ that satisfies $\tilde{\theta}_k^*\tilde{\eta}=\tilde{\eta}$ since $\theta_k^*\eta=\eta$. By functoriality of the short exact sequence (\cite{Bes00} (8)), there is a commuting diagram
\[\xymatrix{
0\ar[r]& H^{\lvert\ell\rvert-3}_\text{dR}(U_{\ell-4})/F^{\lvert\ell\rvert-2-s}\ar[d]^{\theta_\ell^*=0}\ar[r]^\iota & H^{\lvert\ell\rvert-2}_\text{fp}(\scr{U}_{\ell-4},\lvert\ell\rvert-2-s)\ar[d]^{\tilde{\theta}_\ell^*}\ar[r]^p&\text{F}^{\lvert\ell\rvert-2-s}H^{\lvert\ell\rvert-2}_\text{dR}(U_{\ell-4})\ar[d]^{\theta_\ell^*}\ar@{.>}[dl]\ar[r]&0\\
0\ar[r]& H^{\lvert\ell\rvert-3}_\text{dR}(U_{\ell-4})/F^{\lvert\ell\rvert-2-s}\ar[r]^\iota & H^{\lvert\ell\rvert-2}_\text{fp}(\scr{U}_{\ell-4},\lvert\ell\rvert-2-s)\ar[r]^p&\text{F}^{\lvert\ell\rvert-2-s}H^{\lvert\ell\rvert-2}_\text{dR}(U_{\ell-4})\ar[r]&0,
}\]
where the leftmost vertical arrow is zero because of the vanishing $(\ref{eq: H1 vanishing})$. Therefore, there is a canonical lift $\tilde{\omega}^\text{can}=\tilde{\theta}_\ell^*\omega$ to $H^{\lvert\ell\rvert-2}_\text{fp}(\scr{U}_{\ell-4},\lvert\ell\rvert-2-s)$ of any class $\omega\in \theta_\ell^*\text{F}^{\lvert\ell\rvert-2-s}H^{\lvert\ell\rvert-2}_\text{dR}(U_{\ell-4})$,  with the property $\tilde{\theta}_\ell^*\tilde{\omega}^\text{can}=\tilde{\omega}^\text{can}$. 
At this point we can to compute 
\[\begin{split}
\text{AJ}_p(\Delta_{\ell,k})(\pi_1^*\omega\cup\pi_2^*\eta)&=\langle \text{cl}_\text{syn}(\Delta_{\ell,k}), \pi_1^*\tilde{\omega}\cup\pi_2^*\tilde{\eta}\rangle_\text{fp}=\langle (\tilde{\theta}_\ell,\tilde{\theta}_k)_*\text{cl}_\text{syn}(\tilde{\varphi}_*[\scr{W}_r]), \pi_1^*\tilde{\omega}\cup\pi_2^*\tilde{\eta}\rangle_\text{fp}\\
	&=\langle \text{cl}_\text{syn}(\tilde{\varphi}_*[\scr{W}_r]),\pi_1^*\tilde{\theta}_\ell^* \tilde{\omega}\cup\pi_2^*\tilde{\theta}_k^*\tilde{\eta}\rangle_\text{fp}\\
	&=\langle \text{cl}_\text{syn}(\tilde{\varphi}_*[\scr{W}_r]),\pi_1^*\tilde{\omega}^\text{can}\cup\pi_2^*\tilde{\eta}\rangle_\text{fp}\\
&=\text{tr}_{\scr{W}_r}(\tilde{\varphi}^*(\pi_1^*\tilde{\omega}^\text{can}\cup\pi_2^*\tilde{\eta}))=\text{tr}_{\scr{W}_r}(\tilde{\varphi}_1^*\tilde{\omega}^\text{can}\cup\tilde{\varphi}_2^*\tilde{\eta}),\\
\end{split}
\]
where $\tilde{\varphi}_i=(\pi_i\circ\tilde{\varphi})$. The fundamental exact sequence of $\text{fp}$-cohomology induces an isomorphism $\iota: H_\text{dR}^{\lvert\ell\rvert-3}(W_r)\overset{\sim}{\to} H_\text{fp}^{\lvert\ell\rvert-2}(\scr{W}_r,\lvert\ell\rvert-2-s)$, since the filtered piece $\text{F}^nH_\text{dR}^j(W_r)$ is trivial for $n>\dim_{E_\wp} W_r$ and indeed $\lvert\ell\rvert-2-s$ is greater than $\dim_{E_\wp} W_r=r+1$. Therefore, if we write $\tilde{\varphi}_1^*\tilde{\omega}^\text{can}=\iota\Upsilon({\omega})$, we can rewrite the quantity we want to evaluate as 
\begin{equation}\label{eq: syntomic AJ}
\text{AJ}_p(\Delta_{\ell,k})(\pi_1^*\omega\cup\pi_2^*\eta)=\text{tr}_{W_r}(\Upsilon({\omega})\cup_\text{dR}\tilde{\varphi}_2^*\eta)=\langle\Upsilon({\omega}),\tilde{\varphi}_2^*\eta\rangle_\text{dR},
\end{equation}
for the Poincar\'e pairing $\langle\ ,\rangle_{\text{dR}}: H_\text{dR}^{\lvert\ell\rvert-3}(W_r)\times H_\text{dR}^{k-1}(W_r)\overset{\cup}{\longrightarrow }H_\text{dR}^{\lvert\ell\rvert+k-4}(W_r)\overset{\text{tr}_\text{dR}}{\longrightarrow}E_\wp$.

\subsection{Description of $\text{AJ}_p(\Delta_{\ell,k})$ in terms of $p$-adic modular forms}
The final step is to interpret the expression $(\ref{eq: syntomic AJ})$ in terms of $p$-adic modular forms. By taking a large enough power of the Hasse invariant, one can lift it to $\bb{Z}_p$ (\cite{AG} Lemma 11.10); let $\scr{X}_K\hookrightarrow\scr{S}_K(G_L^*)$ be the complement of the zero locus of the lift. By pulling back the lift along $\zeta:\scr{S}_{K'}(\text{GL}_{2,\bb{Q}})\to\scr{S}_K(G_L^*)$, we can define another formal scheme $\scr{Y}_{K'}\hookrightarrow\scr{S}_{K'}(\text{GL}_{2,\bb{Q}})$ equipped with a map $\zeta:\scr{Y}_{K'}\longrightarrow\scr{X}_K$. Let $\scr{A}\to\scr{X}_K$ be the universal abelian surface (resp. $\scr{E}\to\scr{Y}_{K'}$ the universal elliptic curve), we have a commuting diagram
\[\xymatrix{
\scr{E}^r\ar[d]^\nu\ar[r]^{\tilde{\varphi}_1}&\scr{A}^{\ell-4}\ar[d]^\upsilon\\
\scr{W}_r\ar[r]^{\tilde{\varphi}_1}& \scr{U}_{\ell-4}
}\qquad
\text{that induces}\qquad 
\xymatrix{
\tilde{\theta}_\ell^*\tilde{H}_\text{fp}^{\ell-2}(\scr{U}_{\ell-4},\lvert\ell\rvert-2-s)\ar[r]^{\upsilon^*}\ar[d]^{\tilde{\varphi}_1^*}& \tilde{\theta}_\ell^*\tilde{H}_\text{fp}^{\ell-2}(\scr{A}^{\ell-4},\lvert\ell\rvert-2-s)\ar[d]^{\tilde{\varphi}_1^*}\\
\tilde{H}_\text{fp}^{\ell-2}(\scr{W}_{r},\lvert\ell\rvert-2-s)\ar[r]^{\nu^*}& \tilde{H}_\text{fp}^{\ell-2}(\scr{E}^{r},\lvert\ell\rvert-2-s),
}\]
where we consider the Gros-style version of $\text{fp}$-cohomology. We choose to work with $\scr{A}^{\ell-4}$ because the pull back $\upsilon^*(\tilde{\omega}^\text{can})\in\tilde{\theta}_\ell^*\tilde{H}_\text{fp}^{\ell-2}(\scr{A}^{\ell-4},\lvert\ell\rvert-2-s)\overset{\sim}{\to}\theta_\ell^*\text{F}^{\lvert\ell\rvert-2-s}H_\text{dR}^{\ell-2}(A^{\ell-4})$ can be directly described in terms of $p$-adic modular forms. Indeed, the group $\theta_\ell^*\text{F}^{\lvert\ell\rvert-2-s}H_\text{dR}^{\ell-2}(A^{\ell-4})$ is the same as the de Rham realization of the motive $\cal{V}^\ell$ over the rigid fiber $X_K$, which is isomorphic to the rigid cohomology $\bb{H}^2(\scr{S}^\text{tor}_{K,\text{rig}}(G_L^*); j^\dagger\text{DR}^{\bfcdot}(\cal{F}_{G_L^*}^{(\ell,\ell-t_L)}))$ by the comparison theorem of Baldassarri and Chiarellotto (\cite{BC} Corollary 2.6). 

To express the class $\upsilon^*(\tilde{\omega}^\text{can})$ explicitly we need to make a judicious choice of a polynomial. From the form of the Euler factors appearing in Theorem  $\ref{thm: interpolation formulas}$ and the information given by Corollary $\ref{overconvergence in the split case}$ we are led to consider the polynomial $P(T)^2$ for $P(T)=\prod_{\bfcdot,\star\in\{\alpha,\beta\}}(1-\bfcdot_1\star_2T)$. Following (\cite{LoefflerSkinnerZerbes} Proposition 4.5.5), if we set $T=T_1T_2$, we can write $P(T_1,T_2)=a_2(T_1,T_2)P_1(T_1)+b_1(T_1,T_2)P_2(T_2)$ for $P_i(T_i)=(1-\alpha_iT_i)(1-\beta_iT_i)$ and
\[
a_2(T_1,T_2)=\alpha_1\beta_1\alpha_2\beta_2(\alpha_2+\beta_2)T_1^2T_2^3-\alpha_1\beta_1\alpha_2\beta_2T_1^2T_2^2-\alpha_2\beta_2(\alpha_1+\beta_1)T_1T_2^2+1,
  \]
  \[
b_1(T_1,T_2)=\alpha_1^2\beta_1^2\alpha_2\beta_2T_1^4T_2^2-\alpha_1\beta_1(\alpha_2+\beta_2)T_1^2T_2-\alpha_1\beta_1T_1^2+(\alpha_1+\beta_1)T_1.
  \]
  The index $2$ in $a_2$ (resp. the index $1$ in $b_1$) is there to reminds us that the monomial composing the polynomial are of the form $T_1^{a_1}T_2^{a_2}$ with $a_1\le a_2$ (resp. $T_1^{b_1}T_2^{b_2}$ with $b_1> b_2$). The polyonomial $P(T_1,T_2)$ is symmetric in the indices $1,2$, hence we can also write $P(T_1,T_2)=a_1(T_1,T_2)P_2(T_1)+b_2(T_1,T_2)P_1(T_2)$ where $a_1(T_1,T_2)$ (resp. $b_2(T_1,T_2)$) is obtained from $a_2(T_1,T_2)$ (resp. $b_1(T_1,T_2)$) by swapping all the indices. Therefore,
  \[\begin{split}
  	P(T_1,T_2)^2&=a_1a_2P_1P_2+a_2P_1b_2P_1+a_1P_2b_1P_2+b_1b_2P_1P_2\\
		&=a_1a_2P_1P_2+(P-b_1P_2)b_2P_1+(P-b_2P_1)b_1P_2+b_1b_2P_1P_2\\
		&=(a_1a_2-b_1b_2)P_1P_2+P(b_2P_1+b_1P_2),
  \end{split}
  \]
  where the symmetric polynomial $(a_1a_2-b_1b_2)(T_1,T_2)$ satisfies
  \[
  (a_1a_2-b_1b_2)(T)=(1-\alpha_1\beta_2\alpha_2\beta_2T^2)\prod_{\bfcdot,\star\in\{\alpha,\beta\}}(1-\bfcdot_1\star_2T).
  \]
The class of $\omega_{\breve{\msf{g}}^{[\frak{p}_i]}}$ is zero in $\bb{H}^2(\scr{S}^\text{tor}_{K,\text{rig}}, j^\dagger\text{DR}_c^{\bfcdot}(\cal{F}^{(\ell,x)}))$, hence there are overconvergent cuspforms $\msf{g}_j^{(i)}\in S^\dagger_{s_{\tau_i}\cdot(\ell,x)}(K;E_\wp)$ such that $\breve{\msf{g}}^{[\frak{p}_i]}=d_1^{\ell_1-1}(\msf{g}^{(i)}_1)+d_2^{\ell_2-1}(\msf{g}^{(i)}_2)$. Furthermore, $d_1^{1-\ell_1}\breve{\msf{g}}^{[\frak{p}_1, \frak{p}_2]}$ is overconvergent by Corollary $\ref{overconvergence in the split case}$. It follows we can write $P(V(p))^2\breve{\msf{g}}$ as
\[\begin{split}
P(V(p))^2\breve{\msf{g}}&=\big[a_1a_2-b_1b_2](V(p))\breve{\msf{g}}^{[\frak{p}_1, \frak{p}_2]}+P(V(p))[b_2(V(\frak{p}_1),V(\frak{p}_2))\breve{\msf{g}}^{[\frak{p}_1]}+b_1(V(\frak{p}_1),V(\frak{p}_2))\breve{\msf{g}}^{[\frak{p}_2]}\big]\\
	&=d_1^{\ell_1-1}(\msf{h})+d_1^{\ell_1-1}(\msf{h}_1)+d_2^{\ell_2-1}(\msf{h}_2),
\end{split}
\]
where $\msf{h}=[a_1a_2-b_1b_2](V(p))d_1^{1-\ell_1}\breve{\msf{g}}^{[\frak{p}_1, \frak{p}_2]}$, $\msf{h}_1=P(V(p))[b_2\msf{g}_1^{(1)}+b_1\msf{g}_1^{(2)}]$ and $\msf{h}_2=P(V(p))[b_2\msf{g}_2^{(1)}+b_1\msf{g}_2^{(2)}]$.

\begin{proposition}
	Let $L/\bb{Q}$ be a real quadratic extension and $\msf{g}\in S^\dagger_{\ell,x}(K,L;E_\wp)$ an overconvergent cuspform whose class $\omega_\msf{g}$ in $\bb{H}^2\big(\scr{S}^\text{tor}_{K,\text{rig}}(G_L), j^\dagger\text{DR}^\bullet_c\big(\cal{F}^{(\ell,x)}\big)\big)$ is trivial. By Theorem $\ref{th: p-adic cohomology}$ there are $p$-adic modular forms $\msf{g}_j\in S^\dagger_{s_{j}\cdot(\ell,x)}(K;E_\wp)$ for $j=1,2$, such that $\msf{g}=d_1^{\ell_1-1}(\msf{g}_1)+d_2^{\ell_2-1}(\msf{g}_2)$. Then we can construct sections $G_j\in H^0_\text{rig}\big(\scr{S}^\text{tor}_{K,\text{rig}}(G_L), j^\dagger\big(\cal{F}^{(\ell,x)}\otimes\Omega_{\tau_j}^{1}\big) \big)$, $j=1,2$, that satisfy the equation $\omega_\msf{g}=\nabla(G_1+G_2)$ in $H^0_\text{rig}\big(\scr{S}^\text{tor}_{K,\text{rig}}(G_L), j^\dagger\big(\cal{F}^{(\ell,x)}\otimes\Omega^{2}\big) \big)$.
\end{proposition}
\begin{proof}
	Set $v_j^{(a,b)}=\omega_j^a\eta_j^b$, $w_j=\omega_j\wedge\eta_j$ and consider the sections 
	\[
	G_1=\sum_{i=0}^{\ell_{1}-2}(-1)^{i}\frac{(\ell_{1}-2)!}{(\ell_{1}-2-i)!}d_{1}^{\ell_{1}-2-i}(\msf{g}_{1})\big( w_2^{\frac{2-n-\ell_2}{2}}\otimes v_2^{(\ell_2-2,0)}\otimes w_1^{\frac{2-n-\ell_1}{2}}\otimes v_1^{(\ell_1-2-i,i)}\big)\otimes\frac{dq_2}{q_2}.
	\]
	\[
	G_2=-\sum_{i=0}^{\ell_{2}-2}(-1)^{i}\frac{(\ell_{2}-2)!}{(\ell_{2}-2-i)!}d_{1}^{\ell_{2}-2-i}(\msf{g}_{2})\big(w_1^{\frac{2-n-\ell_1}{2}}\otimes v_1^{(\ell_1-2,0)}\otimes w_2^{\frac{2-n-\ell_2}{2}}\otimes v_2^{(\ell_2-2-i,i)}\big)\otimes\frac{dq_1}{q_1}.
	\]
	of $H^0_\text{rig}\big(\scr{S}^\text{tor}_{K,\text{rig}}(G_L), j^\dagger\big(\cal{F}^{(\underline{\ell},2-n)}\otimes\Omega^{1}\big) \big)$. Differentiating them we obtain telescopic sums which collapse to
	\[
	\nabla(G_j)=d_j^{\ell_j-1}(\msf{g}_j)\bigotimes_{c=1}^2\big(w_c^{\frac{2-n-\ell_c}{2}}\otimes v_c^{(\ell_c-2,0)}\big)\otimes(\frac{dq_1}{q_1}\wedge\frac{dq_2}{q_2}).
	\]
	Therefore, $\omega_{\msf{g}}=\nabla(G_1)+\nabla(G_2)$ as claimed.
	\end{proof}
	
It follows that there are sections $G_\msf{h}, G_{\msf{h}_1},G_{\msf{h}_2}$ associated with $\msf{h}$, $\msf{h}_1$, $\msf{h}_2$ respectively, that satisfy $P(p^{-t_L}\text{Fr}_p)^2\omega_{\msf{g}}=\nabla(G_\msf{h}+ G_{\msf{h}_1}+G_{\msf{h}_2})$ 
since $\text{Fr}_p=p^{t_L}V(p)$ in cohomology (Lemma $\ref{frobenius in cohomology}$). The pullback by the morphism $\xi$ gives $P(p^{-t_L}\text{Fr}_p)^2\omega_{\xi^*\msf{g}}=\nabla(G_{\xi^*\msf{h}}+ G_{\xi^*\msf{h}_1}+G_{\xi^*\msf{h}_2})$ and to land in the right cohomology group we need to change the central character using the isomophism $\Psi=\Psi_{x,\ell-t_L}$. Lemma $\ref{lemma: central character and V(p)}$ implies
\[
P(p^{x-\ell}\text{Fr}_p)^2\omega_{\Psi\xi^*\msf{g}}=\nabla(G_{\Psi\xi^*\msf{h}}+ G_{\Psi\xi^*\msf{h}_1}+G_{\Psi\xi^*\msf{h}_2}).
\] 
We set $G=G_{\Psi\xi^*\msf{h}}+ G_{\Psi\xi^*\msf{h}_1}+G_{\Psi\xi^*\msf{h}_2}$ and we let $\epsilon_{\ell}=\frac{1}{\prod_{\tau}(\ell_\tau-2)!}\sum_{\sigma\in\prod_\tau S_{\ell_\tau-2}}\sigma$ be the symmetrization projector 
	\[
	\epsilon_\ell: \bigotimes_\tau(\bb{H}_\tau^1)^{\ell_\tau-2}\to \bigotimes_\tau\text{Sym}^{\ell_\tau-2}\bb{H}_\tau^1
	\]
	which identifies the target sheaf with a subsheaf of the first. Finally, if we set $Q(T)=P(p^{x-\ell}T)$, then the cohomology class $\upsilon^*(\tilde{\omega}^\text{can})$ is represented by $[\omega,\epsilon_\ell G]$ in $\tilde{H}_{\text{f},Q^2}^{\lvert\ell\rvert-2}\big(\scr{A}^{\ell-4},\lvert\ell\rvert-2-s\big)$.

\begin{proposition}
The cohomology class $\nu^*(\tilde{\varphi}_1^*\tilde{\omega}^\text{can})$ is represented by $[0,\tilde{\varphi}_1^*\epsilon_\ell G]$ in $ \tilde{H}_{\text{f},Q^2}^{\lvert\ell\rvert-2}\big(\scr{E}^{r},\lvert\ell\rvert-2-s\big)$ and the section $\tilde{\varphi}_1^*\epsilon_\ell G$ has image under the unit-root splitting equal to
\[
\text{Spl}_\text{ur}(\tilde{\varphi}_1^*\epsilon_\ell G)=(-1)^ss!\Psi_{w,k-1+s}\zeta^*\big(d_1^{\ell_1-2-s}(\msf{h}) +d_1^{\ell_1-2-s}(\msf{h}_1)+ d_2^{\ell_2-2-s}(\msf{h}_2) \big)
\]
in $S^\dagger_{k,k-1+s}(K,E_\wp)$.
\end{proposition} 
\begin{proof}
	 The class $\nu^*(\tilde{\varphi}_1^*\tilde{\omega}^\text{can})=\tilde{\varphi}_1^*\upsilon^*(\tilde{\omega}^\text{can})=[\tilde{\varphi}_1^*\omega,\tilde{\varphi}_1^*\epsilon_\ell G]$ is represented by $[0,\tilde{\varphi}_1^*\epsilon_\ell G]$ because $\text{F}^{\lvert\ell\rvert-2-s}H_\text{dR}^{\lvert\ell\rvert-2}(E)=0$. Since $\varphi_1: E^r\to A^{\ell-4}$ is a map over $X_K$, the pull-back $\varphi_1^*: H_\text{dR}^{\lvert\ell\rvert-3}(A^{\ell-4})\to H_\text{dR}^{\lvert\ell\rvert-3}(E^r)$ is compatible with the pull-back maps between the terms of the Leray spectral sequences for $A^{\ell-4}\to X_K\to\text{Spec}\ \bb{Q}_p$ and $E^r\to X_K\to\text{Spec}\ \bb{Q}_p$. We deduce the existence of a map 
	\[\begin{split}
	\varphi_1^*: H^1_\text{dR}(X_K, \text{DR}^{\bfcdot}(\cal{F}_{G_L^*}^{(\ell,\ell-t_L)}))\longrightarrow &H^1_\text{dR}(X_K,(H^1(E/X_K)^{\otimes(k-2)}\otimes H^{2}(E/X_K)^{\otimes s},\nabla))\\
	&=H^1_\text{dR}(X_K,\zeta_*(H^1(E/Y_{K'})^{\otimes(k-2)}\otimes H^{2}(E/Y_{K'})^{\otimes s}),\nabla))\\
	&=H^1_\text{dR}(Y_{K'},\text{DR}^{\bfcdot}(\cal{F}^{(k,k-1+s)})),
	\end{split}\] since $\zeta$ is a finite morphism. We can describe the pullback $\tilde{\varphi}_1^*\epsilon_\ell G$ explicitly as follows.  Recall that $r=-s+\lvert\ell\rvert-4$ and we consider the subsets $B_1=\{1,\dots,\ell_1-2\}$, $B_2=\{r-(\ell_2-2)+1,\dots, r\}$ of $\{1,\dots r\}$ of cardinalities $\ell_{1}-2,\ell_{2}-2$ respectively, such that $B_1\cup B_2=\{1,\dots r\}$. This is possible because $(\ell,k)$ is a balanced triple. Set $A=\{1,\dots r\}\setminus(B_1\cap B_2)$ whose cardinality is $r-s$. Then we can define a map of sheaves over the modular curve
	\begin{equation}\label{eq: map of sheaves}
		(\bb{H}^1)^{\ell_{1}-2}\otimes(\bb{H}^1)^{\ell_{2}-2}\longrightarrow (\bb{H}^1)^{r-s}\otimes(\wedge^2\bb{H}^1)^{s}.
		\end{equation}
	The pullback $\zeta^*(\epsilon_{\ell}G)$ is an overconvergent section of $\left(\bigotimes_{j=1}^2(\bb{H}^1)^{\ell_{j}-2}\right)\otimes\Omega^1$ and the image in $\text{Sym}^{r-s}\bb{H}^1\otimes(\wedge^2\bb{H}^1)^s\otimes\Omega^1$
	 under the map $(\ref{eq: map of sheaves})$ is $\tilde{\varphi}_1^*\epsilon_\ell G$.
	 A direct calculation reveals that
	 \[\begin{split}
	\text{Spl}_\text{ur}\zeta^*(\epsilon_{\underline{\ell}}G_j)&=(-1)^ss!\zeta^*\Psi_{x,\ell-t_L}\big(d_1^{\ell_1-2-s}(\xi^*\msf{h}) +d_1^{\ell_1-2-s}(\xi^*\msf{h}_1)+ d_2^{\ell_2-2-s}(\xi^*\msf{h}_2) \big)\\
	&=(-1)^ss!\Psi_{x-s-1,k-1+s}\zeta^*\big(d_1^{\ell_1-2-s}(\msf{h}) +d_1^{\ell_1-2-s}(\msf{h}_1)+ d_2^{\ell_2-2-s}(\msf{h}_2) \big)\\
	&=(-1)^ss!\Psi_{w,k-1+s}\zeta^*\big(d_1^{\ell_1-2-s}(\msf{h}) +d_1^{\ell_1-2-s}(\msf{h}_1)+ d_2^{\ell_2-2-s}(\msf{h}_2) \big)
	 \end{split}\]
	 is an element of $\in H^0(\scr{S}^\text{tor}_{K,\text{rig}}(\text{GL}_{2,\bb{Q}})^\text{ord},\underline{\omega}^{r-s}\otimes(\wedge^2\bb{H}^1)^{s}\otimes\Omega^1)=S^\dagger_{k,k-1+s}(K, E_\wp).$
\end{proof}

The diagram
\[\xymatrix{
H_\text{dR}^{\lvert\ell\rvert-3}(W_r)\times H_\text{dR}^{k-1}(W_r)\ar[d]^{\nu^*}\ar[rr]^\cup && H_\text{dR}^{\lvert\ell\rvert+k-4}(W_r)\ar[d]^{\nu^*}\\
H_\text{dR}^{\lvert\ell\rvert-3}(E^r)\times H_\text{dR}^{k-1}(E^r)\ar[rr]^\cup && H_\text{dR}^{\lvert\ell\rvert-3}(E^r)
}\]
commutes and by construction the pair of classes $(\nu^*\Upsilon({\omega}^\text{can}), \nu^*\varphi_2^*\eta)$ belongs to the rigid cohomology groups $\bb{H}^1(\scr{S}^\text{tor}_{K,\text{rig}}(\text{GL}_{2,\bb{Q}}),\text{DR}^\bullet(\cal{F}^{(k,k-1+s)}))\times \bb{H}^1(\scr{S}^\text{tor}_{K,\text{rig}}(\text{GL}_{2,\bb{Q}}),\text{DR}^\bullet(\cal{F}^{(k,k-1)}))$. 
\begin{lemma}\label{vanishing residues}
The pair of classes $(\nu^*\Upsilon({\omega}^\text{can}), \nu^*\varphi_2^*\eta)$ belongs to the subspace of rigid forms with vanishing residues at the lift of supersingular points, $H^1_\text{par}(\text{Sh}^\text{tor}_K(\text{GL}_{2,\bb{Q}}),\cal{F}^{(k,k-1+s)},\nabla)\times H^1_\text{par}(\text{Sh}^\text{tor}_K(\text{GL}_{2,\bb{Q}}),\cal{F}^{(k,k-1)},\nabla) $ . 
\end{lemma}
\begin{proof}
The pair $(\nu^*\Upsilon({\omega}^\text{can}), \nu^*\varphi_2^*\eta)$ comes from a class of $H^1_\text{par}(\text{Sh}^\text{tor}_K(\text{GL}_{2,\bb{Q}}),\cal{F}^{(k,k-1+s)},\nabla)\times H^1_\text{par}(\text{Sh}^\text{tor}_K(\text{GL}_{2,\bb{Q}}),\cal{F}^{(k,k-1)},\nabla) $ because it has vanishing residues (\cite{BDP} Proposition 3.9). Indeed, by construction $\varphi_2^*\eta \in H_\text{dR}^{k-1}(W_r)$ (resp. $\Upsilon({\omega}^\text{can})\in H_\text{dR}^{\lvert\ell\rvert-3}(W_r)$) which is pure of weight $k-1$ (resp. $\lvert\ell\rvert-3$); hence $\nu^*\varphi_2^*\eta$ (resp. $\nu^*\iota^{-1}\varphi_1^*\omega^\text{can}$) lives in the subspace where $\text{Fr}_p$ acts with eigenvalues of complex absolute value $p^{(k-1)/2}$ (resp. $p^{(\lvert\ell\rvert-3)/2}$) while $\text{Fr}_p$ acts on the target of the residue map with eigenvalues of complex absolute value $p^{k/2}$ (resp. $p^{k/2+s}=p^{(\lvert\ell\rvert-3)/2+1}$).
\end{proof}
Lemma $\ref{vanishing residues}$ implies we can compute $(\ref{eq: syntomic AJ})$ using $p$-adic modular forms:
\begin{equation}\label{AJ in terms of p-adic modular forms}
\langle \Upsilon({\omega}^\text{can}), \varphi_2^*\eta\rangle=\text{tr}_{W_r}(\Upsilon({\omega}^\text{can})\cup_\text{dR}\varphi_2^*\eta)=\text{tr}_\text{par}(\nu^*\Upsilon({\omega}^\text{can})\cup_\text{dR} \nu^*\varphi_2^*\eta).
\end{equation}

\begin{lemma}\label{ordinary projector and pairing}
	Let $(\omega,\eta)\in H^1_\text{par}(\text{Sh}^\text{tor}_K(\text{GL}_{2,\bb{Q}}),\cal{F}^{(k,k-1+s)},\nabla)\times H^1_\text{par}(\text{Sh}^\text{tor}_K(\text{GL}_{2,\bb{Q}}),\cal{F}^{(k,k-1)},\nabla) $ be a pair such that $\text{Fr}_p\eta=\alpha\eta$ for $\alpha$ a $p$-adic unit, then $\langle\omega,\eta\rangle=\langle e_\text{n.o.}\omega,\eta\rangle$.
\end{lemma}
\begin{proof}
	We have the equalities of operators $\text{Fr}_p=pV(p)$ and $U_0(p)=p^rU(p)$, therefore the compoutation
	\[\begin{split}
	\langle\omega,\eta\rangle&=\alpha^{-1}\langle\omega,\text{Fr}_p\eta\rangle=\alpha^{-1}\text{Fr}_p\langle\text{Fr}_p^{-1}\omega,\eta\rangle\\
	&=\alpha^{-1}p^{r+1}\langle\text{Fr}_p^{-1}\omega,\eta\rangle=\alpha^{-1}p^{r+1}\langle p^{-1}U(p)\omega,\eta\rangle=\alpha^{-1}\langle U_0(p)\omega,\eta\rangle,
	\end{split}\]
	implies that $\langle\omega,\eta\rangle=\underset{n\to\infty}{\lim}\alpha^{-n!}\langle U_0(p)^{n!}\omega,\eta\rangle=\langle e_\text{n.o.}\omega,\eta\rangle$.
\end{proof}

\begin{theorem}\label{p-adic GZ}
	Let $L/\bb{Q}$ be a real quadratic extension and $M$ a positive integer. Consider $\breve{\msf{g}}\in S_{\ell,x}(M,L;\overline{\bb{Q}})$ a cuspform of either parallel weight $\ell=2t_L$ or non-parallel weight $\ell>2t_L$ over $L$ and $\breve{\msf{f}}\in S_{k,w}(M;\overline{\bb{Q}})$ an elliptic eigenform for the good Hecke operators.  Suppose that the weights $(\ell,k)$ are balanced and choose a prime $p$ splitting in $F$, $p\cal{O}_F=\frak{p}_1\frak{p}_2$, coprime to $M$, such that both cuspforms are $p$-nearly ordinary and the cycle $\Delta_{\ell,k}$ is defined. Then 
	\[
	\text{AJ}_p(\Delta_{\ell,k})(\pi_1^*\omega\cup\pi_2^*\eta)=s!(-1)^s\frac{1-\alpha_1\beta_1\alpha_2\beta_2(\alpha_{\msf{f}^*}^{-1}p^{-1})^2}{\prod_{\bfcdot,\star\in\{\alpha,\beta\}}(1-\bfcdot_1\star_2\alpha_{\msf{f}^*}^{-1}p^{-1})}\frac{\langle e_\text{n.o.}\zeta^*(d_1^{-1-s}\breve{\msf{g}}^{[\frak{p}_1]}),\breve{\msf{f}}^*\rangle}{\langle\msf{f}^*,\msf{f}^*\rangle}
	\]
	where $\omega$ and $\eta$ are the classes in Definition $\ref{definition of classes}$ and $s=\frac{\lvert\ell\rvert-k-2}{2}$. 
\end{theorem}
\begin{proof}
	Equations $(\ref{eq: syntomic AJ})$ and $(\ref{AJ in terms of p-adic modular forms})$ gives us 
\[\text{AJ}_p(\Delta_{\ell,k})(\pi_1^*\omega\cup\pi_2^*\eta)=\text{tr}_\text{par}(\nu^*\Upsilon({\omega}^\text{can})\cup_\text{dR} \nu^*\varphi_2^*\eta)=\langle\nu^*\Upsilon({\omega}^\text{can}),  \nu^*\varphi_2^*\eta\rangle
\]
for $\langle\ ,\rangle$ the Poincar\'e pairing in parabolic cohomology  which lands in $E_\wp(-r-1)$, a one dimensional space on which $\text{Fr}_p$ acts as multiplication by $p^{r+1}$. The isomorphism $\iota: H_\text{dR}^{\lvert\ell\rvert-3}(W_r)\overset{\sim}{\to} H_{\text{f},Q^2}^{\lvert\ell\rvert-2}(\scr{W}_r,\lvert\ell\rvert-2-s)$ is given by $\iota(-)=[0,Q(\text{Fr}_p)^2(-)]$, therefore  $Q(\text{Fr}_p)^2\nu^*\Upsilon({\omega}^\text{can})=\tilde{\varphi}_1^*(\epsilon_{\ell}G)$.
On the one hand,
\[
\text{tr}_\text{par}(Q(\text{Fr}_p)^2\nu^*\Upsilon({\omega}^\text{can})\cup_\text{dR} \nu^*\varphi_2^*\eta)=Q(p^{r+1}\alpha_{\msf{f}^*}^{-1}p^{1-w})^2\langle\nu^*\Upsilon({\omega}^\text{can}),  \nu^*\varphi_2^*\eta\rangle,
\]
because we computed in $(\ref{computation of the eigenvalue})$ that $\text{Fr}_p\big(\nu^*\varphi_2^*\eta\big)=\alpha_{\msf{f}^*}p^{w-1} \big(\nu^*\varphi_2^*\eta\big)$. On the other hand,
\[\begin{split}
&\text{tr}_\text{par}(Q(\text{Fr}_p)^2\nu^*\Upsilon({\omega}^\text{can})\cup_\text{dR} \nu^*\varphi_2^*\eta)=\text{tr}_\text{par}(\tilde{\varphi}_1^*(\epsilon_{\ell}G)\cup\nu^*\varphi_2^*\eta)\\
	&=s!(-1)^s\frac{\langle \Psi_{w,k-1}e_\text{n.o.}\zeta^*\big(d_1^{\ell_1-2-s}(\msf{h}) +d_1^{\ell_1-2-s}(\msf{h}_1)+ d_2^{\ell_2-2-s}(\msf{h}_2) \big),\Psi_{w,k-1}(\breve{\msf{f}}^*)\rangle}{\langle\msf{f}^*,\msf{f}^*\rangle}\\
	&=s!(-1)^s\frac{\langle e_{\msf{f}^*\text{n.o.}}\zeta^*\big(d_1^{\ell_1-2-s}(\msf{h}) +d_1^{\ell_1-2-s}(\msf{h}_1)+ d_2^{\ell_2-2-s}(\msf{h}_2) \big),\breve{\msf{f}}^*\rangle}{\langle\msf{f}^*,\msf{f}^*\rangle}.
\end{split}\] 
Indeed, the class of $\tilde{\varphi}_1^*(\epsilon_{\ell}G)$ in $\bb{H}^1(\scr{S}^\text{tor}_{K,\text{rig}}(\text{GL}_{2,\bb{Q}}),\text{DR}^\bullet(\cal{F}^{(k,k-1+s)}))$ is represented by an overconvergent cuspform whose nearly ordinary projection is equal to $e_\text{n.o.}\text{Spl}_\text{ur}\tilde{\varphi}_1^*(\epsilon_{\ell}G)$ (see \cite{DR} Lemma 2.7), then Lemma $\ref{ordinary projector and pairing}$ justifies the computation.

We claim that $e_{\msf{f}^*,\text{n.o.}}\zeta^*d_j^{\ell_j-2-s}(\msf{h}_j)=0$ for $j=1,2$. Indeed,
 \[\begin{split}
e_{\msf{f}^*,\text{n.o.}}\zeta^*d_j^{\ell_j-2-s}(\msf{h}_j)&=e_{\msf{f}^*,\text{n.o.}}\zeta^*P(V(p))[d_j^{\ell_j-2-s}(b_2\msf{g}_j^{(1)}+b_1\msf{g}_j^{(2)})]\\
	&=P(\alpha_{\msf{f}^*}^{-1}p^{-1})e_{\msf{f}^*,\text{n.o.}}\zeta^*[d_j^{\ell_1-2-s}(b_2\msf{g}_j^{(1)}+b_1\msf{g}_j^{(2)})]=0
\end{split}\]
because the cuspform $\msf{g}_j^{(i)}$ is $\frak{p}_i$-depleted and $b_\iota$ can be written as a polynomial only in the variables $V(p), V(\frak{p}_\iota)$ divisible by $V(\frak{p}_\iota)$  (see Lemma $\ref{thor}$).

The cuspforms $e_{\msf{f}^*\text{n.o.}}\zeta^*(d_1^{\ell_1-2-s}\msf{h})$, $e_{\msf{f}^*,\text{n.o.}}\zeta^*(d_1^{-1-s}\breve{\msf{g}}^{[\frak{p}_1]})$ are very closely related:
\[\begin{split}
e_{\msf{f}^*,\text{n.o.}}\zeta^*(d_1^{\ell_1-2-s}(\msf{h}))&=e_{\msf{f}^*,\text{n.o.}}\zeta^*[a_1a_2-b_1b_2](V(p))(d_1^{-1-s}\breve{\msf{g}}^{[\frak{p}_1]})\\
	&=(1-\alpha_1\beta_1\alpha_2\beta_2(\alpha_{\msf{f}^*}^{-1}p^{-1})^2)\prod_{\bfcdot,\star\in\{\alpha,\beta\}}(1-\bfcdot_1\star_2\alpha_{\msf{f}^*}^{-1}p^{-1})e_{\msf{f}^*,\text{n.o.}}\zeta^*(d_1^{-1-s}\breve{\msf{g}}^{[\frak{p}_1, \frak{p}_2]})\\
	&=(1-\alpha_1\beta_1\alpha_2\beta_2(\alpha_{\msf{f}^*}^{-1}p^{-1})^2)\prod_{\bfcdot,\star\in\{\alpha,\beta\}}(1-\bfcdot_1\star_2\alpha_{\msf{f}^*}^{-1}p^{-1})e_{\msf{f}^*,\text{n.o.}}\zeta^*(d_1^{-1-s}\breve{\msf{g}}^{[\frak{p}_1]}),
\end{split}\]
where the last equality follows apllying Lemma $\ref{thor}$. Finally, the last bit we need to unravel is the polynomial $Q(p^{r+1}\alpha_{\msf{f}^*}^{-1}p^{1-w})^2$; we compute it to be
\[\begin{split}
Q(p^{r+1}\alpha_{\msf{f}^*}^{-1}p^{1-w})^2&=\prod_{\bfcdot,\star\in\{\alpha,\beta\}}(1-\bfcdot_1\star_2p^{x-\ell}p^{r+1}\alpha_{\msf{f}^*}^{-1}p^{1-w})^2\\
	&=\prod_{\bfcdot,\star\in\{\alpha,\beta\}}(1-\bfcdot_1\star_2\alpha_{\msf{f}^*}^{-1}p^{-n+\frac{m}{2}-1})^2=\prod_{\bfcdot,\star\in\{\alpha,\beta\}}(1-\bfcdot_1\star_2\alpha_{\msf{f}^*}^{-1}p^{-1})^2
\end{split}\] since under our assumptions $2n=m$. Hence, putting all together
	\[
	\text{AJ}_p(\Delta_{\ell,k})(\pi_1^*\omega\cup\pi_2^*\eta)=s!(-1)^s\frac{1-\alpha_1\beta_1\alpha_2\beta_2(\alpha_{\msf{f}^*}^{-1}p^{-1})^2}{\prod_{\bfcdot,\star\in\{\alpha,\beta\}}(1-\bfcdot_1\star_2\alpha_{\msf{f}^*}^{-1}p^{-1})}\frac{\langle e_\text{n.o.}\zeta^*(d_1^{-1-s}\breve{\msf{g}}^{[\frak{p}_1]}),\breve{\msf{f}}^*\rangle}{\langle\msf{f}^*,\msf{f}^*\rangle}.
	\]
\end{proof}

\begin{corollary}
	Let $L/\bb{Q}$ be a real quadratic field and $(\ell,k)$ a balanced triple. Let $p$ be a prime splitting in $L$ for which the generalized Hirzebruch-Zagier cycle $\Delta_{\ell,k}$ is defined. Then for all $(\text{P},\text{Q})\in\Sigma^\text{bal}_\text{cry}(\ell,k)$  we have
	\[
	\scr{L}_p(\breve{\cal{G}},\breve{\cal{F}})(\text{P},\text{Q})=\frac{(-1)^s}{s!\msf{E}(\msf{f}_\text{Q}^*)}\frac{\scr{E}_p(\msf{g}_\text{P},\msf{f}_\text{Q}^*)}{\scr{E}_{0,p}(\msf{g}_\text{P},\msf{f}_\text{Q}^*)}\text{AJ}_p(\Delta_{\ell,k})(\pi_1^*\omega_\text{P}\cup\pi_2^*\eta_\text{Q}).
	\]
\end{corollary}
\begin{proof}
	It follows from the formula $(\ref{eq: second expression p-adic L-function})$ and Theorem $\ref{p-adic GZ}$.
\end{proof}


\bibliography{p3L_GZ}
\bibliographystyle{alpha}

\end{document}